\newif\iffinal
\finaltrue	

\documentclass[letterpaper,11pt,reqno]{amsart}
\RequirePackage[utf8]{inputenc}
\usepackage[portrait,margin=2.5cm]{geometry}
\iffinal\else\usepackage[notref,notcite]{showkeys}\fi
\usepackage{mathrsfs}
\usepackage{verbatim}
\usepackage[colorlinks, linkcolor=blue, citecolor=violet, pdfborder={0 1 0}]{hyperref}
\usepackage[foot]{amsaddr}
\usepackage{amssymb,amsthm,amsfonts,amsbsy,latexsym,dsfont}
\usepackage[textsize=small]{todonotes}       
\usepackage{graphicx}
\usepackage[numeric,initials,nobysame]{amsrefs}
\usepackage{upref,setspace}
\usepackage{xcolor,colortbl}
\usepackage{enumerate}
\usepackage{graphicx}
\usepackage{subcaption}   
\usepackage{bm}
\usepackage{tikz}
\usepackage{bigints}

\usetikzlibrary{calc,intersections,through,backgrounds,shapes.geometric}
\usetikzlibrary{graphs}
\tikzset{every path/.style={line width=.07 cm}}
\usepackage[normalem]{ulem}

\allowdisplaybreaks

\definecolor{darkblue}{rgb}{0,0,0.6}

\newenvironment{enumeratei}{\begin{enumerate}[\upshape (i)]}{\end{enumerate}}
\newenvironment{enumeratea}{\begin{enumerate}[\upshape (a)]}{\end{enumerate}}

\usepackage{bbm}

\usepackage{paralist}

\newenvironment{inparaenuma}{\begin{inparaenum}[\upshape \bfseries (a) ]}{\end{inparaenum}}

\numberwithin{equation}{section}
\numberwithin{figure}{section}
\numberwithin{table}{section}

\newtheorem{thm}{Theorem}[section]
\newtheorem{lem}[thm]{Lemma}
\newtheorem{theorem}[thm]{Theorem}

\newtheorem{prop}[thm]{Proposition}
\newtheorem{defn}[thm]{Definition}

\theoremstyle{definition}
\newtheorem{rem}{Remark}

\newcommand{\ind}{\mathds{1}}
\newcommand{\vep}{\varepsilon}

\newcommand{\set}[1]{\left\{#1\right\}}

\newcommand{\convas}{\stackrel{\mathrm{a.s.}}{\longrightarrow}}

\newcommand{\lab}{{\mathrm lab}}

\newcommand{\mycomment}[1]{}

\def\qed{ \hfill $\blacksquare$}

\newcommand{\cA}{\mathcal{A}}\newcommand{\cB}{\mathcal{B}}\newcommand{\cC}{\mathscr{C}}
\newcommand{\cE}{\mathcal{E}}
\newcommand{\cG}{\mathcal{G}}
\newcommand{\cL}{\mathcal{L}}
\newcommand{\cM}{\mathcal{M}}

\newcommand{\cS}{\mathcal{S}}\newcommand{\cT}{\mathcal{T}}

\newcommand{\vC}{\mathbf{C}}

\newcommand{\vx}{\boldsymbol{x}}

\newcommand{\mvxi}{\boldsymbol{\xi}}

\newcommand{\fB}{\mathfrak{B}}

\newcommand{\fP}{\mathfrak{P}}
\newcommand{\fU}{\mathfrak{U}}
\newcommand{\fV}{\mathfrak{V}}

\newcommand{\bL}{\mathbb{L}}
\newcommand{\bN}{\mathbb{N}}
\newcommand{\bR}{\mathbb{R}}
\newcommand{\bT}{\mathbb{T}}

\newcommand{\bZ}{\mathbb{Z}}

\newcommand{\dU}{\mathds{U}}

\DeclareMathOperator{\E}{\mathbb{E}}
\DeclareMathOperator{\pr}{\mathbb{P}}

 \DeclareMathOperator{\Poi}{Poisson}

\newcommand{\sss}{\scriptscriptstyle}

\newcommand{\erdos}{Erd\H{o}s-R\'enyi }

\newcommand{\ldown}{\ell^2_{\sss\searrow}}

\newcommand{\convd}{\stackrel{d}{\longrightarrow}}
\newcommand{\convp}{\stackrel{\sss \prob}{\longrightarrow}}

\usepackage{pdfsync}

\newcommand{\op}{o_{\sss \mathrm{P}}}

\definecolor{aqua}{rgb}{0.0, 1.0, 1.0}
\definecolor{boo}{rgb}{1.0, 0.0, 1.0}
\definecolor{stred}{rgb}{1.0, 0.44, 0.37}

\newcommand{\Bin}{{\sf Bin} }

\newcommand{\MBRW}{{\sf MBRW}}

\usepackage{accents}

\newcommand{\prob}{\mathbb{P}}

\DeclareMathAlphabet\mathbfcal{OMS}{cmsy}{b}{n}

\usepackage{accents}

\def\beq{ \begin{equation} }
\def\eeq{ \end{equation} }

\newcommand{\chs}[1]{{{#1}}}

\usepackage{framed}
\usepackage{tgpagella,eulervm}

\newcommand{\indic}[1]{\mathbbm{1}_{\{#1\}}}

\newcommand{\eqn}[1]{\begin{equation}#1\end{equation}}
\newcommand{\eqan}[1]{\begin{align}#1\end{align}}
\newcommand{\nn}{\nonumber}
\newcommand{\rO}{{\mathrm{\scriptstyle{O}}}}
\newcommand{\rY}{{\mathrm{\scriptstyle{Y}}}}
\newcommand{\srO}{{\mathrm{\scriptscriptstyle{O}}}}
\newcommand{\srY}{{\mathrm{\scriptscriptstyle{Y}}}}
\newcommand{\e}{{\mathrm e}}
\newcommand{\expec}{\mathbb{E}}
\newcommand{\invisible}[1]{}

\newcommand{\Tree}{\textsf{T}}

\newcommand{\N}{\mathbb{N}}
\newcommand{\emp}{\varnothing}
\newcommand{\PPT}{\mathrm{PPT}}
\newcommand{\btr}{$\triangleright$}
\newcommand{\din}{d^{\sss(\mathrm{in})}}
\newcommand{\dif}{\mathrm{d}}

\newcommand{\clusterold}{\mathscr{C}}
\newcommand{\susceptibilitypi}[2]{s^\pi_{#2}(#1)}
\newcommand{\susceptibility}[2]{s_{#2}(#1)}
\newcommand{\csusceptibilitypi}[2]{S^\pi_{#2}(#1)}

\newcommand{\clusterpi}[2]{\mathscr{C}^\pi_{#1}(#2)}

\newcommand{\Gpi}[1]{G_{#1}^\pi}
\newenvironment{subproof}[1]{%
  \begin{proof}[Proof of part (#1)]%
}{%
  \end{proof}%
}

\DeclareSymbolFont{extraup}{U}{zavm}{m}{n}
\DeclareMathSymbol{\varheart}{\mathalpha}{extraup}{86}
\DeclareMathSymbol{\vardiamond}{\mathalpha}{extraup}{87}
\newcommand{\ensymboldefinition}{$\blacktriangleleft$}

\newcommand{\eps}{\varepsilon}

\begin{document}

\title[Subcritical percolation and coagulation]{Non-equilibrium coagulation processes and\\ subcritical percolation on evolving networks}

\date{}

 \author[Banerjee]{Sayan Banerjee$^1$}
\author[Bhamidi]{Shankar Bhamidi$^1$}
\address{$^1$Department of Statistics and Operations Research, 304 Hanes Hall, University of North Carolina, Chapel Hill, NC 27599}
\email{sayan@email.unc.edu; bhamidi@email.unc.edu; r.w.v.d.hofstad@tue.nl;\\
rounak\_ray@brown.edu}
\author[van der Hofstad]{\\Remco van der Hofstad$^2$}
\address{$^2$ Department of Mathematics and Computer Science, Eindhoven University of Technology, Netherlands}
\author[Ray]{Rounak Ray$^3$}
\address{$^3$Department of Applied Mathematics,  Brown University}

\begin{abstract}
Coagulation processes play an important role in an array of domains ranging from colloidal chemistry to combinatorial optimization. Models for such processes typically start with a large, but fixed, number of particles,  which then coalesce forming components whose subsequent sizes recursively drive the evolution of the system. 

The goal of this paper is to understand the asymptotics of coagulation dynamics mediated by {\em new particles} entering the system, and forging connections between existing components in the system. One major motivation is to understand percolation on growing networks where the evolution of connected components resembles a non-equilibrium version of the multiplicative coalescent. The supercritical $\pi> \pi_c$ regime for a host of such models was conjectured in statistical physics, and then rigorously proven in mathematics, to exhibit behavior similar to the BKT infinite-order phase transition as $\pi\searrow \pi_c$. It has further been conjectured that the entire regime $\pi<\pi_c$ for such growing networks are ``critical'' with power-law cluster size distributions having a non-universal exponent for all values of $\pi \in (0, \pi_c)$. 

In this paper, we study percolation on a specific model, namely, the uniform attachment model, as a concrete template in order to develop general tools based on stochastic approximation, local convergence, branching random walks and tree-graph inequalities to prove the above conjectured phenomena. For each $\pi \in (0,\pi_c)$, we show there exists an explicit $\alpha(\pi) \in (0,\tfrac{1}{2}) $ such that the maximal component size, as well as the size of the component containing any fixed vertex, all re-scaled by $n^{\alpha(\pi)}$, converge almost surely to strictly positive random variables as the network size $n \to \infty$.  The emergence of this phenomenon is driven by the delicate interplay between the rate of {\em new vertices} entering the system, creating their own empires, versus {\em mergers} of existing large components. These dynamics lead to novel phenomena, compared to classical \emph{static} models in which the number of vertices is fixed and vertex roles are exchangeable, including long-range dependence and fixation of the {\bf identity} of the maximal component, within finite time, among a finite number of `early' components. Moreover, in contrast with most static network models, we show that the \emph{susceptibility}, that is, the expected size of the component of a uniformly chosen vertex, remains bounded as the network grows and $\pi$ approaches $\pi_c$ from below. The general tools developed in this paper will be used in follow-up work to understand percolation for general growing network evolution models.

\end{abstract}

\maketitle

\tableofcontents

\section{Introduction}
\label{sec:int}
The study of percolation as a model of disordered media has played an important role in the development of modern discrete probability and its applications. See, e.g.,  \cite{Grim99}, as well as the references therein, for percolation and related interacting particle systems on lattices, and \cite{Boll01} for classical work in the mean-field setting.  The main goal of this paper is to understand large network asymptotics for percolation on random {\em dynamic} network models \cites{Boll01,Hofs24,Hofs17,Durr06,BolJanRio07}.  One major finding in this setting is the connection between {\bf dynamic formulations} of percolation on random graphs, and  application areas such as colloidal chemistry, in particular coagulation and fragmentation processes, see, e.g., \cites{Aldo97,Aldo99,Bert06}.
\smallskip

To fix ideas, consider the \erdos random graph $G_n(n, t/n)$ starting with vertex set $[n]:=\set{1,2,,\ldots, n}$, with each pair of vertices connected with probability $t/n$, independently across pairs. One can view the above as a graph-valued process by coupling all values of $t$ via generating uniform random variables $U_{i,j}$ for each $i\neq j\in [n]$, and retaining only those edges $\{i,j\}$ satisfying $U_{i,j}\leq t/n$, where we think of $t$ as corresponding to {\em time}. Then, in the genesis of the field of random graphs, Erd\H{o}s and R\'enyi \cite{ErdRen60} proved that the critical point is $t=1$: For $t> 1$, there is a unique component containing a positive proportion of the vertices in the network (often called the {\em giant component}), while, for $t< 1$, the size of the maximal connected component is of order $\log{n}$. At the critical value $t=1$, the size of the largest and second largest component are both order $n^{2/3}$, with {\em random variables} arising as multiplicative factors in the scaling limit.
\smallskip

The intricate evolution of such systems as they pass through the critical point $t=1$, and in fact the entire scaling window for times $t=1+\lambda n^{-1/3}$, was worked out in increasing detail in \cites{Lucz90a,JanKnuLucPit93} and culminating in \cite{Aldo97}.  Fix $\lambda \in \bR$, and write $\cC_{{\sss(i)}, n}(\lambda)$ for the $i^{\rm th}$ largest component in $G_n(n, n^{-1}(1+\lambda n^{-1/3}))$, breaking ties at random. Let $\vC_n(\lambda) = \big(n^{-2/3}|\cC_{{\sss(i)}, n}(\lambda)|\big)_{i\geq 1} $ denote the entire vector of component sizes (appending infinitely many zeros after the graph is exhausted), viewed as a random object in    
    \[
    \ldown = \set{\vx = (x_1, x_2, \ldots), x_1 \geq x_2 \geq \cdots, \sum_{i=1}^{\infty} x_i^2 < \infty}.
    \]
Then, Aldous  \cite{Aldo97} showed that, for each fixed $\lambda$, $\vC_n(\lambda) \convd \mvxi(\lambda)$, where $\mvxi(\lambda)$ is an infinite-dimensional random vector obtained from the ordered excursions of an associated stochastic process. Further, Aldous \cite{Aldo97} connected the critical regime to the large body of probabilistic work in colloidal chemistry, surveyed in \cite{Aldo99}. More precisely,  the component sizes $(\vC_n(\lambda))_{-\infty < \lambda < \infty}$, viewed as a process in $\lambda$ and interpreting $\lambda$ as time, converges on $\ldown$ to a Markov process which is now known as the {\bf standard multiplicative coalescent}, whose transitions are described as follows:
\begin{quote}
    \emph{Given a current configuration of component sizes $\vx = (x_1, x_2, \ldots)$, components $i,j$ merge at rate $x_i\cdot x_j$ leading to a new component of size $x_i+x_j$.}
\end{quote}

The multiplicative coalescent, as well as structural (topological) properties of the above limits of maximal components, have played a major role in proving {\bf universality} in the critical regime for a wide array of other random graph models,  as well as understanding key objects in combinatorial optimization such as strong disorder models of information diffusion and
the minimal spanning tree with random edge weights \cites{AddBroGol12,AddBroGolMie17,BhaBroSenWan14,BhaSen24,BasBhaBroSenWan25}. Let us now describe the main model considered in this paper, and then elaborate on its connections to the above discussion. 

\subsection{Evolving network models}
\label{sec-model}
We start by describing the main model considered in this paper and then describe the general model class that the techniques in this paper are geared to understand.
\subsubsection{Model in this paper: Uniform attachment}
Fix $m=2$. Let $(G_n)_{n\geq 1}$ be a sequence of networks grown using the uniform attachment schemes with $m$ out-edges. In more detail,  having constructed the network $G_n$,  a new vertex $n+1$ enters the system with $m$ edges which it uses to connect to the existing network, via sequentially connecting each edge to an existing vertex $v\in G_n$ uniformly at random. The main object of interest in this paper is {\em percolation} on the constructed network. Fix a parameter $\pi \in (0,1)$ and let $G^{\pi}_n$ denote the network obtained by independently retaining each edge with probability $\pi$, and deleting the edge with probability $1-\pi$. This model, and its brethren that we describe next, in the setting of percolation,  are all conjectured to belong to the same {\bf universality class}. 
\subsubsection{General framework}
\label{sec:model-gen}
The above model is the simplest example of the following general class of models.  
Fix weight parameters $a, \delta \in [0,\infty)$ and consider an attachment function $f\colon \bN_0 \to \bR_+$ given by
\eqn{
    \label{eqn:linear-att-def}
    f(x) = a x + \delta, \qquad x\in \bN_0.
    }
Let $(G_n)_{n\geq 1}$ be a sequence of networks grown using preferential attachment dynamics with the following variations:
\begin{enumeratea}
    \item {\bf Fixed out-degree preferential attachment:} Fix $m\geq 2$. Having constructed the network $G_n$, and writing $\deg(v,n)$ for the degree of an existing vertex $v$ in $G_n$,  a new vertex $n+1$ enters the system with $m$ edges which it uses to connect to the existing network, via sequentially connecting each edge to an existing vertex $v\in G_n$ in the network with probability proportional to $f(\deg(v,n))$, independently across edges.
    \item {\bf Bernoulli preferential attachment:} Fix $\beta >0$.  Conditionally on $G_n$, connect vertex $n+1$ to every vertex $v\in G_n$ with probability $\beta \cdot \frac{f(\deg(v,n))}{\sum_{u\in G_n}f(\deg(u,n))}\wedge 1$, independent across vertices. 
\end{enumeratea}
 Once again, the main objective is percolation on $G_n$, where each edge in $G_n$ is independently retained with probability $\pi$, and deleted with probability $1-\pi$, and we denote the percolated graph by $G_n^\pi$.   
\smallskip

The model in this paper is a special case of the above family with $m=2$, $a=0$ and arbitrary $\delta >0$. A {\bf natural question} is why focus on this specific case. As it turns out, this specific model {\bf provides the template} to develop general stochastic approximation tools required to understand percolation on all the general models of network evolution above, whilst using a less technically and notationally heavy model. 
The general set of tools developed in this paper are used in \cite{BanBhaHazHofRay26} to understand subcritical percolation on general uniform and preferential attachment models.  
\color{black}

\subsubsection{Dynamic construction of percolation and non-equilibrium coagulation processes}
In the interest of readability, we now specialize the discussion to the main model studied in this paper, i.e., the uniform attachment model with  $m=2$. The formulation in the previous section describes the construction of percolation on the dynamic graph sequence via a {\em static} process, in the sense that we first construct $G_n$ and then independently retain or delete edges with the prerequisite probability.  One can equivalently construct the percolation process {\em dynamically} as we grow the graph. We start with the discrete-time construction:

\begin{enumeratea}
    \item Having constructed the graph $G^{\pi}_n$ till step $n$,  a new vertex $n+1$ enters the system. This new vertex has three potential configurations:
    \begin{enumeratei}
        \item It has no edges with probability $(1-\pi)^2$.
        \item It has one edge with probability  $2\pi(1-\pi)$. 
        \item It has both edges with probability $\pi^2$. 
    \end{enumeratei}
    \item The vertex then selects, uniformly and independently for each retained edge, among the existing vertex set $[n]$, and connects to this vertex. Thus, for each edge, an existing connected component $\cC\subseteq G^{\pi}_n $ is selected with probability $|\cC|/n$, where $|\cC|$ denotes the number of vertices in $\cC$. 
\end{enumeratea}
Write $(G^{\pi}_n)_{n\geq 1}$ for this dynamic formulation in discrete time. Next, note that there is a natural continuous-time analog, where at each stage a new vertex enters the system at rate proportional to the number of individuals in the network, and then connects using the dynamics to the existing vertex set as per the rules above. With a slight abuse of notation, we also use $(G^\pi_t)_{t\geq 0}$ for the continuous-time analog.  In this continuous-time formulation, if $N(t)$ denotes the number of vertices in the system at time $t$, then note that the dynamics can be formulated as follows:

    \begin{enumeratea}
    \item {\bf Multiplicative-coalescent-type aggregation:} Two distinct components $\cC_a(t), \cC_b(t) \subseteq G^\pi_t$ merge into a single component of size $|\cC_a(t)| + |\cC_b(t)| +1$ at rate proportional to $2\pi^2|\cC_a(t)||\cC_b(t)|/N(t)$,  thus essentially replicating a version of the multiplicative coalescent dynamics described previously. 
    \item {\bf Linear aggregation:} A component grows as $|\cC_a(t)| \leadsto |\cC_a(t)| +1 $ at rate $\pi^2 |\cC_a(t)|^2/N(t) +2\pi(1-\pi) |\cC_a(t)|$.
    \item {\bf Immigration:} A new vertex, with no edges, immigrates into the system at rate $(1-\pi)^2 N(t)$. Till a given time $T$, each such vertex either seeds a new component, where it is the oldest vertex, that grows without merging with other components containing an older vertex till time $T$, or eventually merges with such a component by this time.  
\end{enumeratea}

Note that the number of individuals in the continuous time formulation,  $N(t) = |G^\pi_t|$, has the same distribution as a rate-one Yule process, which is a pure birth process growing at rate of the current population size. Further, we contrast the evolution of this process with the standard multiplicative coalescent dynamics described for the \erdos random graph. Whilst there are a number of similarities, we  note the following differences:
\begin{enumeratea}
    \item The (percolated) \erdos random graph (as well as more general models of stochastic coagulation surveyed in \cite{Aldo99}) consists of a large fixed number $n$ of vertices and all the vertices have an \emph{exchangeable role} in the resulting graph. The dynamic formulation of percolation on evolving networks results in the number of vertices increasing over time, where the role of each vertex in the network evolution crucially depends on its \emph{time of arrival}. Thus the percolated graphs form a {\em stochastic process} of graph-valued random variables of {\em growing} size. 
    \item The competition between the above multiplicative-coalescent-type dynamics and linear aggregation, which results in the increase of existing component sizes via either mergers or aggregation, is balanced by the immigration (in continuous time, in an exponential
    time scale) of new ``singleton'' vertices that seed potentially new ``empires'', and compete for growth with the existing components, or are eventually consumed by a previously seeded component. 
    \item The inherent growth rate of typical component sizes emerges from a delicate balance between the above two competitive factors. Further, in this setting,  our goal is to understand the prevalence of {\bf long-range dependence} on the evolution: Is it true that components of early vertices have overwhelming advantage so that subsequently seeded components can never overtake these component sizes? A qualitative version of this result is described in Theorem \ref{thm-max-comp}(b).    
\end{enumeratea}

\subsection{Main results}
\label{sec:res}
In this section, we state our main results.
Write $\clusterpi{v}{n}$ for the connected component of vertex $v$ in $G^{\pi}_n$, and $|\clusterpi{v}{n}|$ for its size.  In continuous time, we use the same notation replacing $n$ with $t$, and, when possible, we suppress the dependence on $\pi$.  
Define 
\begin{equation}
\label{pic-UM-m=2}
    \pi_c = \frac{2 - \sqrt{2}}{4} = \frac{1}{2(2 + \sqrt{2})}\approx 0.1464.
\end{equation}
Known results and proof techniques for percolation on uniform \cite{BolRio05} (using techniques developed for the uniformly grown random graph in \cite{BolJanRio05}), and preferential attachment models \cite{HazHofRay23} imply that the size of the maximal component $|\clusterold_{\max}^\pi(n)|$ satisfies         
    \eqn{
    \label{super-PA}
    |\clusterold_{\max}^\pi(n)|/n \convp \theta(\pi),
    }
where $\convp$ denotes convergence in probability, and $\theta(\pi)$ is strictly positive for $\pi > \pi_c$ and zero for $\pi \leq \pi_c$. 
The regime $\pi < \pi_c$ is the main goal of this paper.

\medskip

\paragraph{\bf Scaling limit for largest subcritical connected components.} For $\pi < \pi_c$, define 
\begin{equation}
\label{eqn:alpha-pi}
    \alpha(\pi) :=  \frac{1}{2} \big( 1 - \sqrt{8\pi^2 - 8\pi + 1} \big).
\end{equation}
\chs{Our first result shows almost sure convergence of the component sizes, as well as a form of {\bf persistence} of the maximal component showing that whp, the identity of the maximal component is likely to be one of the components containing one of the early vertices. This phenomenon suggests promising randomized algorithms for network archaeology which is discussed in the next section. To carefully define this phenomenon, we need some notation to associating components by the oldest vertex in it.   } 

Define, for all $v\in[n]$,
\begin{equation}
\label{eqn:cc-less-def}
\clusterpi{\sss \geq v}{n} =
\left\{
\begin{array}{ll}
\clusterpi{v}{n} & \text{if $v$ is the oldest vertex in the component,}  \\
\emptyset & \text{otherwise}.
\end{array}
\right.
\end{equation}

\begin{theorem}[Convergence of fixed-vertex and maximal component sizes]
\label{thm-max-comp}
Consider percolation on the uniform attachment model with $m=2$. Let $\pi<\pi_c$. 
\begin{itemize}
\item[(a)] For any fixed $i\in \N$, $n^{-\alpha(\pi)}|\clusterpi{i}{\pi}| \convas \zeta_i$  as $n\to\infty$, where  $\zeta_i >0$ almost surely and $\convas$ denotes convergence almost surely.
\item[(b)] $n^{-\alpha(\pi)} |\clusterpi{\max}{n}| \convas \max_{i \ge 1}\zeta_i$ as $n\to\infty$. Further, the maximal component is `weakly persistent' in the sense that, for any $\vep>0$, there exists $K(\vep)<\infty$ such that 
	\begin{equation*}
    	\prob\Big(\clusterpi{\max}{n} = \clusterpi{\sss \geq i}{n} \text{ for some } i> K(\vep) \text{ for infinitely many } n\Big) < \vep.
	\end{equation*}
\end{itemize}
\end{theorem}

Figure \ref{fig:kde-two-plots} below shows simulations for the maximal component size rescaled by $n^{\alpha(\pi)}$ for two different choices of $\pi = 0.08$ and $\pi= 0.12$, showing the lack of concentration of the maximal component after rescaling by $n^{\alpha(\pi)}$.

\begin{rem}[General $m$-out random graphs]
    The same set of results are true for general number of out-edges $m\geq 2$ with 
    \[\pi_c(m) = \frac{1}{2(m+\sqrt{m(m-1)})},\]
    \[\alpha(\pi)=2m\pi(1-\pi)/\big( 1+\sqrt{1-4m\pi(1-\pi)} \big),\]
    This setting, as well as the more general preferential attachment models described in Section \ref{sec:model-gen}, are studied in \cite{BanBhaHazHofRay26} using the techniques developed in this paper. 
\end{rem}

\begin{rem}[Novel universality class]
    There are two distinctions when the above phenomena are compared with subcritical percolation on static models such as the configuration model \cites{Jans08b, Jans09c, Foun07}: 
    \begin{inparaenuma}
        \item For static models, typically the size of maximal subcritical components is of the same order as the {\em maximal degree} when the maximal degree is large. Thus, for power-law configuration models with degree exponent $\tau$, these can be expected to be of order $n^{1/(\tau-1)}$, while for random graphs having less heavy-tailed degrees, the maximal components have size of order $n^{o(1)}$. Since $\alpha(\pi)\in (0,\tfrac{1}{2}),$ however, for percolation on the above dynamic uniform attachment graph, the subcritical components have size that grows as a {\em fixed positive power} of $n$ {\em even when the degrees are at most of order $\log{n}$}. Further, the exponent $\alpha(\pi) \uparrow \tfrac{1}{2}$ as $\pi \nearrow \pi_c$.  We will see that this difference in behavior is intimately tied to the graph dynamics, which is such that {\em larger connected components are more likely to merge}, thus creating a kind of preferential attachment effect at the level of component sizes; 
        \item For static models with heavy-tailed degree distributions, given any $\eps>0,$ there exists $\tilde K(\eps)$ such that the maximal component contains one of the $\tilde K(\eps)$ largest degree vertices with probability at least $(1-\eps)$ as $n\to\infty$.  For the uniform attachment model considered in this paper, the vertex with the maximal degree is ``far away'' from the initial vertices (see \cite{BanBha21}). It should be provable that this `degree non-persistence' makes the maximal degree vertex unlikely to be in the maximal component; we defer further fine-scaled  topological properties of the maximal component to future work, see Section \ref{sec-disc}.
    \end{inparaenuma}  
\end{rem}
\medskip

\begin{figure}[htbp]
  \centering
  \begin{minipage}[b]{0.49\textwidth}
    \centering
    \includegraphics[width=\linewidth]{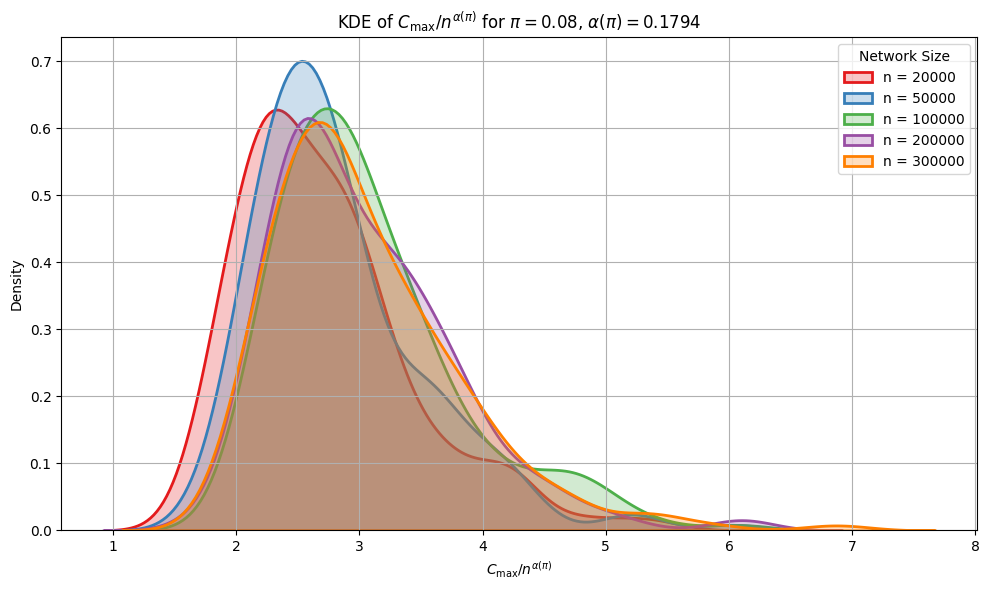}
    \\[0.5em]
    {\small $\pi = 0.08$, $\alpha(\pi) = 0.1794$}
  \end{minipage}
  \hfill
  \begin{minipage}[b]{0.49\textwidth}
    \centering
    \includegraphics[width=\linewidth]{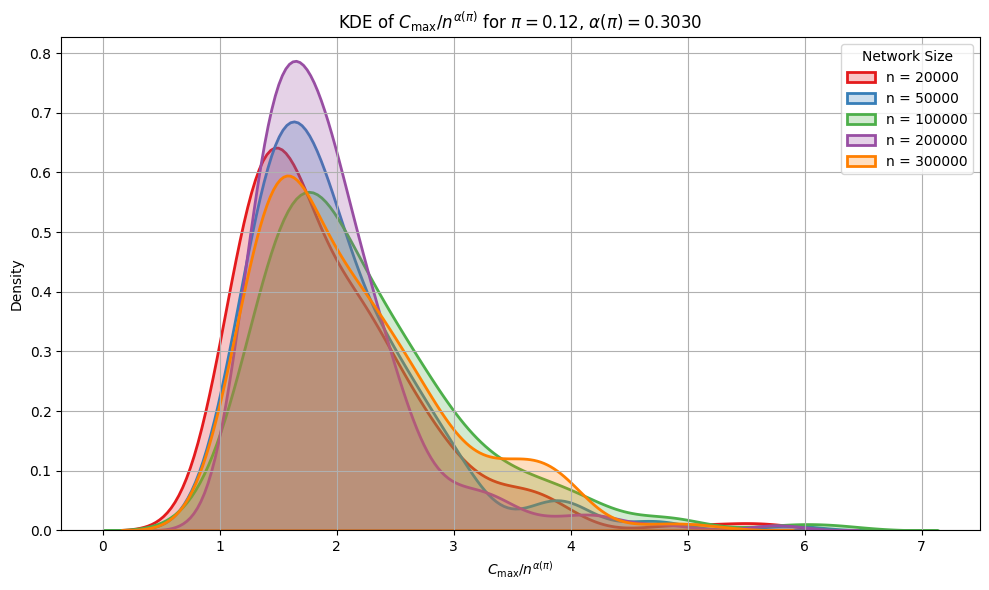}
    \\[0.5em]
    {\small $\pi = 0.12$, $\alpha(\pi) = 0.3030$}
  \end{minipage}
  \caption{Density plots of $|\clusterpi{\max}{n}|/n^{\alpha(\pi)}$ for two values of $\pi$. Each panel shows 200 trials per $n$, scaled by the predicted $n^{\alpha(\pi)}$, showing the non-constant nature of the limit.}
  \label{fig:kde-two-plots}
\end{figure}

\begin{figure}[htbp]
    \centering
    \begin{minipage}{0.48\textwidth}
        \centering
        \fbox{\includegraphics[width=\linewidth]{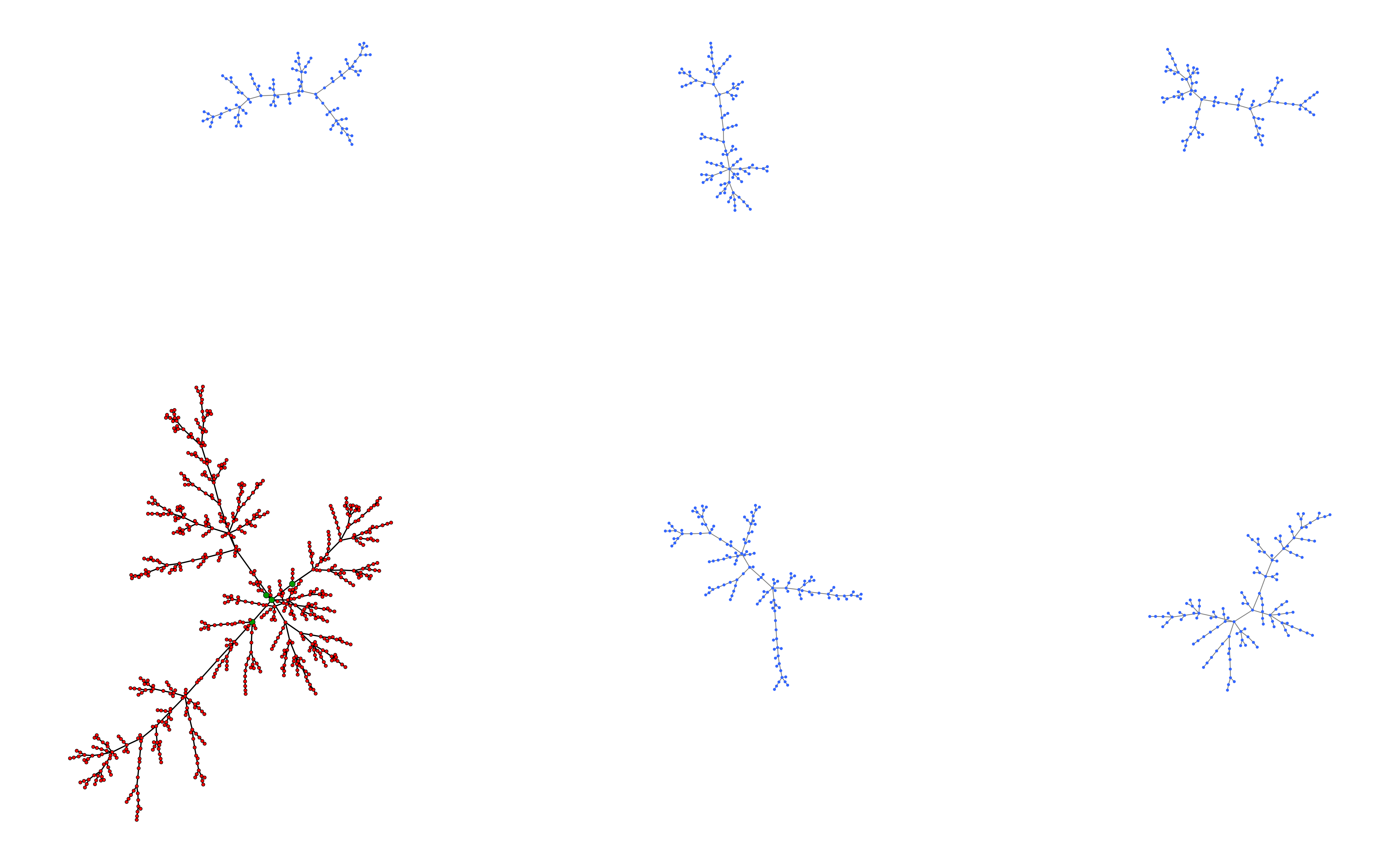}}
    \end{minipage}\hfill
    \begin{minipage}{0.48\textwidth}
        \centering
        \fbox{\includegraphics[width=\linewidth]{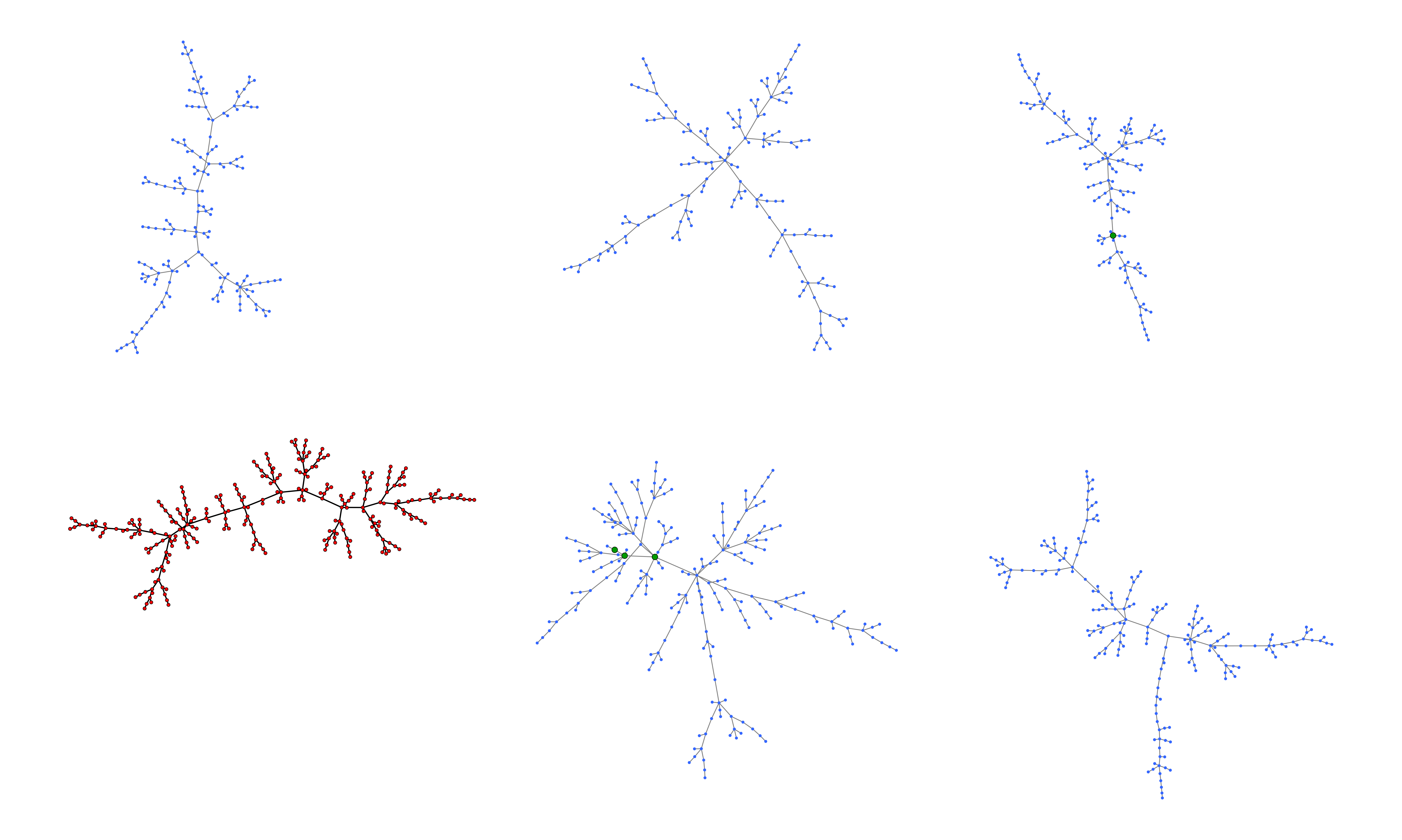}}
    \end{minipage}
    \caption{Both figures display the 6 largest components for $\pi=.14$ and the location of the first four vertices (in green) for $n=10^6$. In the left plot all of them are in the maximal component. In the right plot, 3 of the vertices appear in the second largest, and the fourth in the sixth largest, component.   }
    \label{fig:side_by_side}
\end{figure}

\paragraph{\bf Convergence of susceptibilities.} A crucial ingredient in our proof is the behavior of the {\em susceptibility}, which we define next.  For $n\geq 1$, write the collection of connected components in $G^{\pi}_n$ as 
    \eqn{
    \label{connected-component-vector}
    \vC(n) 
    = \set{\clusterold \subseteq G^{\pi}_n\colon  \clusterold \text{ connected component}}.
    }
Define the $k^{\rm th}$ susceptibility function in discrete time by
\begin{equation}
    \label{eqn:k-suscep}
    \susceptibilitypi{n}{k}=\frac{1}{n} \sum_{\clusterold\in\vC(n)} |\clusterold|^{k} =  \frac{1}{n} \sum_{v\in[n]} |\clusterpi{v}{n}|^{k-1} .
\end{equation}
From the above, $\susceptibilitypi{n}{2} = \E(|\clusterpi{V_n}{n}|\big| \vC(n))$ can be viewed as the {\em average component size} of a uniformly chosen vertex $V_n \in [n]$ in $G^{\pi}_n$, conditional on $\vC(n)$. Although this is a `local' quantity, our proofs will show that this quantity is the key to the asymptotics of `global' quantities, such as fixed-vertex and maximal component sizes. Thus, the second susceptibility $\susceptibilitypi{n}{2}$ will play a special role in the sequel. Our second main result describes its asymptotics for $\pi<\pi_c$:

\begin{theorem}[Convergence of susceptibility]
\label{thm-suscep-comp}
Consider percolation on the uniform attachment model with $m=2$. For any $\pi < \pi_c$, $\susceptibilitypi{n}{2}\convas \susceptibilitypi{\infty}{2}$, where
	\begin{equation}
    	\susceptibilitypi{\infty}{2} := \frac{(1 - 4\pi)-  \sqrt{8\pi^2 - 8\pi + 1}}{4\pi^2}.
	\label{eqn:s2-def}
	\end{equation}
\end{theorem}
\begin{rem}[Bounded critical susceptibility]
    Remarkably, the susceptibility $\susceptibilitypi{\infty}{2}$ in \eqref{eqn:s2-def} remains bounded {\em for all} $\pi<\pi_c$, with a finite limit as $\pi\nearrow \pi_c$. This is unlike related percolation settings, where the expected component size blows up as $\pi\nearrow \pi_c$ like $(\pi_c-\pi)^{-1}$ \cite{Jans09b}, and as such again is indicative of a different percolation universality class compared to those of other (static) random graph models. The $(\pi_c-\pi)^{-1}$ blow-up for the subcritical susceptibility in these static models can also be understood by noting that the local limit is a branching process (see, e.g., \cite{Hofs24}), for which the expected total progeny for percolation on it is $1/(1-\nu\pi)$, with $\nu$ the expected number of offspring, and for which $\pi_c=1/\nu$. For the uniform attachment model, however, the local limit is a multi-type branching process with continuum type space, for which the expected critical total progeny behaves rather differently.
\end{rem}

The scaling exponent $\alpha(\pi)$ of the maximal component in Theorem \ref{thm-max-comp} satisfies 
    \begin{equation}
    \label{eqn:alpha-pi-susc-rel}
    \alpha(\pi) = 2\pi^2 \susceptibilitypi{\infty}{2} +2 \pi.
    \end{equation}
The proof of Theorem \ref{thm-max-comp} will explain the origin of this relation.

\subsection{Contributions of this paper}
\label{sec:contr}
We briefly outline the main contributions of this paper, deferring an outline of the proof to Section \ref{sec-proof-outline}, and a detailed literature review to Section \ref{sec-disc}. 
\begin{enumeratea}
    \item {\bf Phase transitions for evolving networks.} Motivated by areas such as the early growth of the internet at the webpage level and duplication mutation models in protein interaction networks \cites{CalHopKleNewStr01,KimKraKahRed02,CouBau03}, \cite{DorGolMen08}*{Section III.F}, in the early 2000s there was a host of activity in the statistical physics community on the nature of the percolation phase transition for evolving networks (see Sec. \ref{sec:model-gen}), culminating in two major conjectures:
    \begin{enumeratei}
        \item {\bf Supercritical regime and the BKT-type phase transition.} It was conjectured that, owing to the interplay between the density of new connections and the rate of vertices entering the system,  the maximal component in these models should experience a Berezinskii-Kosterlitz-Thouless (BKT)-type infinite order phase transition as $\pi \searrow \pi_c$, with the maximal component scaling like $|\clusterpi{\max}{n}|/n \approx \exp(-C_{\tt model}/\sqrt{\pi-\pi_c})$. This conjecture was subsequently rigorously proven for many of the canonical evolving models including the CHKNS model \cites{Durr03,BolJanRio07}, the uniformly growing random graph model \cite{BolJanRio05}, the specific model in this paper \cite{BolRio05}, as well as a number of variants of the preferential attachment class of networks \cites{DerMor09,EckMorOrt18}. 
        \item {\bf Subcritical regime $\pi <\pi_c$.} For the subcritical regime where $\pi<\pi_c$, it was conjectured that this entire regime should be considered ``critical'' with the maximal component satisfying $|\clusterpi{\max}{n}| \approx n^{\alpha(\pi)}$ for a non-universal, $\pi-$ and model-dependent growth exponent $\alpha(\pi) < \tfrac{1}{2}$ \cite{KimKraKahRed02}. Till date, for the model considered in this paper, the best known results are that $|\clusterpi{\max}{n}|$ grows at most at rate $\sqrt{n\log{n}}$ (\cite{BolRio05}*{Theorem 2}), and, for a related model, that the expected component size of the first vertex $|\clusterpi{1}{n}|$ satisfies $\E[|\clusterpi{1}{n}|] = O(n^{\alpha({\pi})})$ \cite{BolJanRio05}. Proving this general conjecture has been challenging (see M\"orters and Schleicher \cite{MorSch25}, summarized in \eqref{subcritical-dynamic-IRG}, for recent progress ``on the log scale'' for a class of inhomogeneous random graphs believed to be in the same universality class as the uniform attachment model). This paper develops general techniques involving {\em stochastic approximation, local limit analysis, branching random walks and tree-graph inequalities} to prove the above conjectured phenomenon for the general class of evolving network models. 
    \end{enumeratei}
    \item {\bf Long-range dependence and network archaeology.} 
    One of the major motivations for studying evolving networks is the problem of reconstructing the temporal history of a network when only its final topology—namely, its adjacency matrix—is observed. This task, known as {\em network archaeology} \cite{NavKin11}, has recently inspired a burst of activity in the probability community, particularly in the setting where the underlying dynamics generate trees \cites{BubDevLug17,CraXu21,BanBha21,BanBha22,LugPer19}. In this regime, carefully designed centrality measures tailored to trees allow one to estimate the location of the root (vertex~1) within a fixed bounded radius with high accuracy.

Over the last two years, there has been substantial momentum toward extending these analyses beyond the tree case \cites{BriCalLug23,BriGirLugSul25}. In this direction, our \emph{weak persistence} result in Theorem~\ref{thm-max-comp}(b), which exhibits the inherent long-range dependence in such processes, naturally suggests the following randomized algorithm for estimating the root of a network from its adjacency matrix alone:

\begin{enumeratei}
\item Fix $\pi < \pi_c$. Independently retain each edge of the network with probability $\pi$ and delete it with probability $1-\pi$. By Theorem~\ref{thm-max-comp}(b), with high probability the root lies in one of the $K(\varepsilon)$ maximal components of the resulting subgraph. Hence the search for vertex~1 may be restricted to these components.
\item It should be relatively straightforward to show that, at a global scale, these maximal components are effectively tree-like, in the sense that cycles arise only during the very early stage of a component’s growth. Consequently, existing procedures for root estimation on trees should have a strong chance of identifying `confidence sets' for the early vertices in each component. The union of these candidate sets should then contain vertex~1 with high probability.
\end{enumeratei}

Given the length and scope of this paper, we defer a full development of this program to future work.
    
    \item \textbf{Stochastic approximations and the differential equation method.}
A standard workhorse in dynamic formulations within probabilistic combinatorics is the 
\emph{differential equation method} \cites{Worm95,Kurt70,Kurt71}, 
which allows one to show that macroscopic quantities---such as susceptibility functions---converge to 
deterministic, continuous limits governed by an ordinary differential equation (ODE). However, most 
classical applications begin with a system of large fixed size~$n$ and then track the evolution of 
associated functionals.

Transplanting these techniques directly to the setting of \emph{growing networks} is substantially more 
challenging: one must carefully disentangle the long-range dependence inherent in the evolution of 
network components, as well as the heavy-tailed nature of their asymptotics, before any macroscopic averaging 
can be meaningfully applied.

In this paper, we address these difficulties by combining \emph{local weak convergence} with 
\emph{branching-random-walk techniques} to first establish that various macroscopic functionals converge 
in probability to deterministic constants. We then apply refined \emph{stochastic-approximation} 
arguments to upgrade these results to almost sure convergence. As a consequence, we obtain almost-sure 
convergence of the appropriately rescaled component sizes.
\end{enumeratea}

\subsection{Proof outline}
\label{sec-proof-outline}
We next give an overview of the steps in the proof of Theorems \ref{thm-max-comp} and \ref{thm-suscep-comp}.
\smallskip

\paragraph{\bf Continuous-time embedding.} It will be convenient to use the continuous-time formulation $(G^\pi_t)_{t\geq 0}$. Let  $(|\clusterpi{1}{t}|)_{t\geq 0}$ denote the size process of the component containing vertex one. 
For continuous time $t\geq 0$, write the collection of components in $G^{\pi}_t$, similarly to \eqref{connected-component-vector}, as 
	\eqn{
	\label{clusters-UA-discrete}
	\vC(t) = \set{\clusterold \subseteq G^{\pi}_t\colon  \clusterold \text{ component}}.
	} 
Denote the number of vertices in the random graph at time $t$ by $N(t)$. As mentioned before, it is not hard to see that $t\mapsto N(t)$ is a rate-one {\em Yule process}, i.e., a pure-birth process with transitions $N(t)\leadsto N(t)+1$ at rate $N(t)$. Standard results, see, e.g., \cite{Norr98}, imply that $N(t)\e^{-t}\convas E$, with $E$ an exponential random variable with mean 1. Thus, $N(t)$ grows exponentially. Now, the evolution of the size of the component of a specific vertex, say vertex one, $|\clusterpi{1}{t}|$, is as follows:
\begin{itemize}
    \item[(a)] A new vertex enters the system with two edges and then merges another component $\clusterold\in \vC(t)$ with $\clusterpi{1}{t}$. Thus,  $\clusterold\neq \clusterpi{1}{t} $ merges with $\clusterpi{1}{t}$, leading to the transition $|\clusterpi{1}{t}|\leadsto |\clusterpi{1}{t}| + |\clusterold|+1$, at rate $2\pi^2|\clusterpi{1}{t}||\clusterold|/N(t).$
    \item[(b)] A new vertex enters the system with at least one edge, and all its edges connect to $\clusterpi{1}{t}$, leading to the transition $|\clusterpi{1}{t}| \leadsto |\clusterpi{1}{t}| + 1$, which occurs at rate $\pi^2 |\clusterpi{1}{t}|^2/N(t) + 2\pi(1-\pi)|\clusterpi{1}{t}|.$
    \item[(c)] A new vertex enters with no edges connected to it; this leads to no change in $|\clusterpi{1}{t}|$. 
\end{itemize}
For $k \ge 2$, define the $k^{\rm th}$ susceptibility in continuous time as
	\begin{equation}
	\label{susdef}
    	\csusceptibilitypi{t}{k} :=\frac{1}{N(t)}\sum_{\clusterold \in \vC(t)}|\clusterold|^{k}= \frac{1}{N(t)}\sum_{i\in[N(t)]}|\clusterpi{i}{t}|^{k-1}, \quad t \ge 0.
	\end{equation} 
Note that the discrete and continuous susceptibilities in \eqref{eqn:k-suscep} and \eqref{susdef}, respectively, are related through $\susceptibilitypi{N(t)}{k}=\csusceptibilitypi{t}{k}$. Denote the generator of the associated continuous-time Markov process by $\mathcal{L}$. Then
	\begin{align}
	\label{dynamics-connected-component-continuous}
    \mathcal{L} |\clusterpi{1}{t}| &= \sum_{\clusterold \in \vC(t)\colon \clusterold\neq \clusterpi{1}{t}}\frac{2\pi^2}{N(t)}|\clusterpi{1}{t}||\clusterold|(|\clusterold|+1) + \left(\pi^2 \frac{|\clusterpi{1}{t}|^2}{N(t)} + 2\pi(1-\pi)|\clusterpi{1}{t}|\right)\\
    &= 2\pi^2|\clusterpi{1}{t}| \left(\csusceptibilitypi{t}{2} - \frac{|\clusterpi{1}{t}|^2}{N(t)}\right) + 2\pi^2|\clusterpi{1}{t}|\left(1 - \frac{|\clusterpi{1}{t}|}{N(t)}\right)\nn\\
    &\qquad + \left(\pi^2 \frac{|\clusterpi{1}{t}|^2}{N(t)} + 2\pi(1-\pi)|\clusterpi{1}{t}|\right)\nn\\
    &\le \left(2\pi^2\csusceptibilitypi{t}{2} + 2\pi\right)|\clusterpi{1}{t}|.\nn
\end{align}
As a result, for $t\geq 0$, the process 
	\eqn{
	\label{super-martingale}
	M(t)=|\clusterpi{1}{t}| \exp\left(-\int_0^t [2\pi^2 \csusceptibilitypi{u}{2} + 2\pi]~ \dif u \right)
	}
is a non-negative super-martingale, and thus converges almost surely to a finite random variable $M(\infty)$. In order to gain insight into the asymptotics of $|\clusterpi{1}{t}|$, we thus need to understand the asymptotics of $t\mapsto \csusceptibilitypi{t}{2}$ as formulated in Theorem \ref{thm-suscep-comp}, as we explain next.
\medskip

\paragraph{\bf Convergence of components in Theorem \ref{thm-max-comp}: martingales.} Recall \eqref{super-martingale} and the key role played by the second susceptibility. If one can show $\susceptibilitypi{n}{2}\convas \susceptibilitypi{\infty}{2}$, equivalently for $\csusceptibilitypi{t}{k}=\susceptibilitypi{N(t)}{k}$, $\csusceptibilitypi{t}{k}\convas \susceptibilitypi{\infty}{2}$, then, for large $t$,
	\[
	\int_0^t [2\pi^2 \csusceptibilitypi{u}{2} + 2\pi]~ \dif u
	=t(2\pi^2\susceptibilitypi{\infty}{2}+2\pi)+o_{\pr}(t)=t \alpha(\pi) +o_{\pr}(t),
	\]
by \eqref{eqn:s2-def} and \eqref{eqn:alpha-pi}, and where we write $X_n=\op(Y_n)$ for random variables $X_n$ and $Y_n$ when $X_n/Y_n\convp 0$. As a result, $|\clusterpi{1}{t}|=\e^{t\alpha(\pi)+o_{\pr}(t)}$. Together with $|\clusterpi{1}{t}|$ in continuous time being equal to $|\clusterpi{1}{N(t)}|$ in discrete time, and the fact that $N(t)=\e^{t+O_{\pr}(1)}$, this explain why $|\clusterpi{1}{n}|=n^{\alpha(\pi)+o_{\pr}(1)}$. 
\smallskip

To upgrade this to almost sure convergence, i.e., to $|\clusterpi{1}{n}|n^{-\alpha(\pi)} \convas \zeta_1$, we strengthen the results to show that $n^{\gamma}(\susceptibilitypi{n}{2}-\susceptibilitypi{\infty}{2})\convas 0$ for some small $\gamma>0$. In turn, this implies that $\int_0^t (\csusceptibilitypi{u}{2}-\susceptibilitypi{\infty}{2})~ \dif u$ converges almost surely as well. To prove that $\zeta_1>0$ almost surely, instead, we will define a martingale $M^{\sss(-1)}(t)=|\clusterpi{1}{t}|^{-1} \exp\left(\int_0^t [2\pi^2 \csusceptibilitypi{u}{2} + 2\pi+R(u)]~ \dif u \right)$, for some appropriate error term $u\mapsto R(u)$. Since the non-negative martingale $M^{\sss(-1)}(t)$ converges almost surely, $\prob(M^{\sss(-1)}(\infty)=\infty)=0$, from which it follows after some work that also $\prob(M(\infty)=0)=0$.
\medskip

To prove the asymptotics of the maximal component and the weak persistence phenomenon, the crucial estimate is provided in Section \ref{sec-latecomplem}, where it is shown that the appropriately normalized `late' component sizes, namely, components where the oldest vertex has label $> M$ for some large enough $M$, cannot be large. This involves tracking the evolution of high powers of these component sizes which is performed via a semi-martingale decomposition and carefully controlling the drift and quadratic variation. `A priori' bounds on the maximal component size of the form $O(n^{\alpha(\pi) + \vep})$, for small $\vep>0$, play a crucial technical role in setting up appropriate stopping times required to understand the growth of late components. 
\medskip

\paragraph{\bf Convergence of susceptibilities in Theorem \ref{thm-suscep-comp}: stochastic approximations.}
The previous discussion explains the key role played by the second susceptibility. The evolution of the process $(\susceptibilitypi{n}{2})_{n\geq 1}$ in discrete time can be written as 
        \begin{equation}
        \label{SAscheme-rep0}
            \susceptibilitypi{n+1}{2} - \susceptibilitypi{n}{2} = \frac{1}{n+1}\left[F(\susceptibilitypi{n}{2}) + R_n + \xi_{n+1} \right], \qquad n\geq 1,
        \end{equation}
where the function $F$ equals $F(s) = 2\pi^2 s^2+ (4\pi-1)s+1$, the error term $R_n$ satisfies $|R_n| \leq K \susceptibilitypi{n}{4}/n$ for a positive constant $K$, and the martingale differences  $(\xi_{n+1})_{n\geq 1}$ given by $\xi_{n+1}=(n+1)\big(\susceptibilitypi{n+1}{2} - \expec[\susceptibilitypi{n+1}{2}\,|\,G^{\pi}_n]\big)$ satisfy $\expec[\xi_{n+1}^2\,|\,G^{\pi}_n]\leq K (\susceptibilitypi{n}{3})^2$.
\smallskip

The stochastic recursion in \eqref{SAscheme-rep0} follows by an analysis similar to the one in \eqref{dynamics-connected-component-continuous}, applied to the sum of squares of component sizes rather than the component size of a single vertex. For this, we add vertices one by one, and inspect the evolution of the sum of squares of the component sizes. When two edges are added, they merge components $\clusterold$ and $\clusterold'$ to become a component of size $|\clusterold|+|\clusterold'|+1$ with probability $2|\clusterold||\clusterold'|/n^2$. Summing out over all pairs of components $\clusterold, \clusterold'$ explains why the evolution is {\em quadratic}. The precise shape of $s\mapsto F(s)$ follows by carefully analyzing all arising terms.
\smallskip

To analyze the asymptotics of the stochastic evolution of $\susceptibilitypi{n}{2}$, we note that we can write
	\[
	F(s)=b(s-\lambda_1)(s-\lambda_2),
	\]
where $b=2\pi^2$, while
    \[
    \lambda_1=\frac{1 - 4\pi-  \sqrt{8\pi^2 - 8\pi + 1}}{4\pi^2}=\susceptibilitypi{\infty}{2},\qquad \text{and}
    \qquad \lambda_2=\frac{1 - 4\pi+  \sqrt{8\pi^2 - 8\pi + 1}}{4\pi^2}>\lambda_1.
    \]
The values $\lambda_1$ and $\lambda_2$ correspond to the {\em fixed points} of the dynamics in \eqref{SAscheme-rep0}. Since $F'(s)=2bs-b(\lambda_1+\lambda_2)$, we have that $F'(\lambda_1)<0$ and $F'(\lambda_2)>0$, so that $\lambda_1$ is the stable fixed point, while $\lambda_2$ is unstable. Intuitively, we obtain two scenarios. Either the dynamics drives $\susceptibilitypi{n}{2}$ to the stable fixed point $\lambda_1=\susceptibilitypi{\infty}{2},$ which explains Theorem \ref{thm-suscep-comp}, or the dynamics lands in the `unstable' region ($s> \lambda_2$) and the value of $\susceptibilitypi{n}{2}$ diverges. We show that the second scenario cannot occur.
\smallskip

The proof follows by a careful analysis of the evolution, using as a starting point that $\susceptibilitypi{n}{2}\convp \susceptibilitypi{\infty}{2}$ by a local convergence argument, combined with a uniform integrability argument that is adapted from \cite{Hofs24}*{Chapter 8}. This argument identifies $\susceptibilitypi{\infty}{2}$ also as the expected component size of the root of percolation on the uniform attachment version of the P\'olya point tree, which is the local limit of our uniform attachment model. Since $\susceptibilitypi{n}{2}\convp \susceptibilitypi{\infty}{2}$, for every $\eta>0$, $n\mapsto \susceptibilitypi{n}{2}$ must enter the $\eta$-neighborhood of $\susceptibilitypi{\infty}{2}$ infinitely often. A path-wise analysis using the stochastic approximation scheme \eqref{SAscheme-rep0} shows that, each time this happens, the evolution has a {\em strictly positive probability} of never leaving this $\eta$-neighborhood, so eventually it has to succeed. Since $\eta>0$ is arbitrary, this completes the proof of Theorem \ref{thm-suscep-comp}.

An essential ingredient in the proof is a preliminary estimate that shows that $|\clusterpi{\max}{n}|/\sqrt{n}\convas 0$. Such a bound is crucial, since it shows that the process $n\mapsto \susceptibilitypi{n}{2}$ makes small steps, so that it cannot jump far when it is close to $\susceptibilitypi{\infty}{2}$.

\medskip
\paragraph{\bf Organization of this paper.} After discussing our results and related literature in Section \ref{sec-disc}, the subsequent sections implement each step of the above program, as follows. In Section \ref{sec:local-limit}, we identify the local limit in probability of our percolated uniform attachment model. In Section \ref{sec-weak-bounds-components}, we derive weak upper bounds on component sizes that imply that the second susceptibility makes small steps. In Section \ref{sec-conv-second-susc-stoch-approx}, we prove the convergence of the second susceptibility in Theorem \ref{thm-suscep-comp}, by relying on a careful stochastic approximation analysis. In Section \ref{sec:proof-first-comp}, we complete the proof of the asymptotics for the first component in Theorem \ref{thm-max-comp}(a). In Section \ref{sec-max-component}, we identify the asymptotics of the maximal component and prove Theorem \ref{thm-max-comp}(b).

\subsection{Discussion}
\label{sec-disc}
In this section, we discuss related literature and place our results in the context of the broader area of evolving networks. 
\medskip

\paragraph{\bf New universality class} Our main results pertain to the subcritical phase of percolation on uniform attachment models with $m=2$ out-edges. Our results show that percolation on such graphs is in an entirely different universality class compared to various other canonical models, such as rank-1 inhomogeneous random graphs and configuration models. Such random graphs are in the \erdos{} universality class unless the inhomogeneity, present due to the variability of the vertex degrees or weights, is very strong. Typically in such settings, the susceptibility blows up as $\pi\nearrow \pi_c$, while for our model, it remains finite {\em even when approaching the critical value}. Further, for rank-1 models, the barely subcritical connected components have sizes that can be much larger than $\sqrt{n}$, and the scale of the strictly subcritical components, arising for $\pi<\pi_c,$ is {\em independent of $\pi$,} unlike in our model, where these sizes scale with $n$ as $n^{\alpha(\pi)}$, with $\pi\mapsto \alpha(\pi)$ increasing from $\alpha(0)=0$ to $\alpha(\pi_c)=\tfrac{1}{2}$. 
\medskip

\paragraph{\bf Related literature}

\begin{enumeratea}
    \item {\bf Conjectured universality:} As described in Section \ref{sec:model-gen}, there is a host of models formulated in both the uniform attachment family \cites{BolJanRio05, Durr03,Rior05}, as well as more general degree-driven attachment schemes.  We expect the subcritical results in this paper to cover this broad family of models, albeit with substantial technical work (ongoing \cite{BanBhaHazHofRay26}). We present several tools—e.g., Proposition \ref{prop-as-up-down}—in a form more general than needed here, which avoids repeating delicate stochastic-approximation arguments. Nonetheless, extending the analysis to preferential attachment is highly challenging, since the attachment depends on the {\em total degree of components} rather than component size, so our stochastic-approximation and martingale arguments must track component weights instead of sizes. Converting weights back to sizes entails studying multiple susceptibilities, including sums of squared weights and mixed weight and size related susceptibilities.
    \item {\bf Related work in the subcritical regime:} M\"orters and Schleicher \cite{MorSch25} investigate a class of dynamic inhomogeneous random graphs, in which an edge is independently present between vertices $u$ and $v$ with probability $p_{uv}=\beta (u\wedge v)^{\gamma-1}(u\vee v)^{-\gamma}$ for some $\gamma\in [0,1]$. Such models are called {\em inhomogeneous random graphs of preferential attachment type}, since also for preferential attachment models as in \eqref{eqn:linear-att-def} with $a=1$ and $\delta>-m$, the probability that an edge is present between vertices $u$ and $v$ is of the order  $p_{uv}$ for $\gamma =(m+\delta)/(2m+\delta)$ and an appropriate $\beta>0$ (see \cite[Chapter 8]{Hofs24}). Previous work \cite{DerMor13}, shows that $\beta_c=(\tfrac{1}{4}-\tfrac{\gamma}{2})\wedge 0$ for this model (i.e., for $\beta > \beta_c$ there exists a giant component). Thus, $\beta_c>0$ when $\gamma<\tfrac{1}{2}$, which in preferential attachment models would correspond to $\delta>0$.  M\"orters and Schleicher \cite{MorSch25} show that the maximal component $\clusterold_{\max}(n)$ satisfies 
    \eqn{
    \label{subcritical-dynamic-IRG}
    \log{(|\clusterold_{\max}(n)|)}/\log{n}\convas \tfrac{1}{2}-\sqrt{(\tfrac{1}{2}-\gamma)^2-\beta(1-2\gamma)}.
    }
This result can be contrasted with the result in this paper with the exponent of $n$ similar to $\alpha(\pi)$ in \eqref{eqn:alpha-pi}, and with $\beta$ in the model of M\"orters and Schleicher \cite{MorSch25} playing the role of $\pi$. In particular, the quantities in the square roots in \eqref{subcritical-dynamic-IRG} and \eqref{eqn:alpha-pi} equal zero precisely when the model is {\em critical}.  The proof techniques in \cite{MorSch25} are completely different, using connections to branching random walks, while  the main thrust of this paper is directly leveraging the intrinsic dynamics of the underlying model. We refer to \cite{BolJanRio07}*{Section 16} for an extensive discussion of such infinite-order phase transitions in the context of inhomogeneous random graphs with independent edges.

    \item {\bf BKT-type phase transitions and the supercritical regime:} Compared to the subcritical regime, there is significantly more work for these models in the supercritical regime. For the specific model in this paper, it was proved by Bollob\'as and Riordan \cite{BolRio05} (see Riordan \cite{Rior05} for the sharpest results) that the giant is small, in that the proportion of vertices in it is approximately $\e^{-\Theta(1)/\sqrt{\pi-\pi_c}}$, where $\pi$ denotes the percolation parameter, thus this model exhibits an {\em infinite order} phase transition. For specific case of percolation on the linear preferential attachment with affine term $\delta=0$, similar results were proven by  Riordan \cite{Rior05}. 
Barely supercritical regime of the Bernoulli preferential attachment model was studied by Eckhoff, M\"orters and Ortgiese \cite{EckMorOrt18}, following the establishment of a phase transition by Dereich and M\"orters in \cites{DerMor09, DerMor13}.
See also \cites{Durr03, BolJanRio05} for a related model, which is sometimes called the CHKNS model after \cite{CalHopKleNewStr01}.
\end{enumeratea}

\smallskip

\paragraph{\bf Work in progress} In addition to the developing the universality program described above for understanding the subcritical regime for evolving network models, we are currently exploring 
the following related questions: 
\begin{enumeratei}
    \item {\bf  Power-law exponent of the component size of a typical vertex:}  Related literature, e.g. \cite{Durr03}*{Theorem 7} suggests that the probability that the connected component of a uniform vertex has size at least $k$ decays as $k^{-1/\alpha(\pi)}$. Such results provide additional evidence that such connected components are always {\em critical}, with sizes depending sensitively on the percolation parameter $\pi$. This is in stark contrast to classical percolation, e.g.\ on $\bZ^d$, or static random graphs, where the connected component size distribution only has power-law tails at, or very near, criticality (see, e.g., \cite[Chapter 9]{Grim99}).
    \item {\bf Structure of maximal components and network archaeology:} For $\pi<\pi_c$, one could expect the connected components to have a {\em bounded} number of cycles. Indeed, for small $n$, the probability that a vertex enters with 2 percolated edges that {\em both} connect to the component of vertex 1 is strictly positive. However, as $n$ grows large, this probability decays as $(|\clusterpi{1}{n}|/n)^2=n^{2\alpha(\pi)-2},$ where the exponent $2\alpha(\pi)-2$ is strictly smaller than $-1$. Thus, an almost-surely {\em finite} number of vertices connect to $\clusterpi{1}{n}$ with two edges, which indicates that the number of cycles is almost surely bounded. This raises the question what the scaling limit is of such subcritical components. Further, as described in Section \ref{sec:contr}, this property directly lead to randomized algorithms for estimating the location of the initial vertex set given the topology only of the final network,  using classical tree based algorithms. 
    \item {\bf Near-critical behavior:} For the inhomogeneous random graph model \cite{MorSch25} described in \eqref{subcritical-dynamic-IRG}, work in progress by Jorritsma, Maillard and M\"orters shows that this model has an intricate near-critical behavior when $\pi = \pi_c + \beta/(\log{n})^2$ for $\beta \in \bR$. This would also be the final frontier for the class of models considered in this paper. 
\end{enumeratei}

\smallskip

\color{black}

\medskip

\paragraph{\bf Notation.} Throughout this paper, we abbreviate independent and identically distributed to iid, and left- and right-hand side to lhs and rhs, respectively. We further write $E\sim {\rm Exp}(\lambda)$ to indicate that $E$ has an exponential distribution with parameter $\lambda$, $X\sim{\rm Bin}(n,p)$ to that $X$ has a binomial distribution with parameters $n\geq 1$ and $p\in (0,1)$, and $I\sim {\rm Ber}(p)$ that $I$ has a Bernoulli distribution with parameter $p$. We further write $x\vee y=\max\{x,y\}$ and $x\wedge y=\min\{x,y\}$.

\section{Local limits and the second susceptibility}
\label{sec:local-limit}

This section describes the local weak limit of the percolation component of a randomly selected vertex in the uniform attachment model and shows that for $\pi < \pi_c$, the size of this component is uniformly integrable sequence.  The main result,  Proposition \ref{prop:expec-s2-convg} below, further identifies $\susceptibilitypi{\infty}{2}$, defined in \eqref{eqn:s2-def}, as the expected component size of the root in the $\pi$-percolated local limit. We first need a definition:

\begin{defn}[{\bf Multi-type branching random walk killed at a random barrier}]
\label{def:mbrw}
    Fix type space $\cS = \{\rO, \rY\}\times \bR$, where the labels $\rO, \rY$ are mnemonics for ``old'' and ``young'', respectively, and the second coordinate represents the {\em locations} of the particles. Consider the following multi-type branching random walk on $\bR$, started with one individual $\varnothing$ at the origin. Write $\pr_{\srO}$ (respectively $\pr_{\srY}$) for the probability measure where the root has label $\rO$ (respectively, $\rY$) and location at the origin. Similarly, write $\E_{\srO}, \E_{\srY}$ for the corresponding expectation operations. Let $\cA$ denote an exponential random variable with mean 1 independent of the dynamics of the branching random walk defined as follows:
\begin{enumeratei}
    \item Vertices of both types give birth to type $\rY$ vertices with locations prescribed by a Poisson process with rate $2\pi$ translated from the position of the parent. More precisely,  the child of a vertex at location $a$, born at time $t$ in the Poisson point process, is positioned at location $a+t$.
    \item A type $\rO$ vertex at location $a$ gives birth to $\text{Bin}(2, \pi)$ vertices of type $\rO$. A type $\rY$ vertex gives birth to a $\text{Ber}(\pi)$ number of vertices of type $\rO$. Conditionally on the location $a$ of the vertex, each of these label $\rO$ vertices is at location $a-\hat E$, where $\hat E$ is an exponential random variable with mean 1, independent of all other randomness.
    \item Any particle whose location exceeds $\cA$ is killed. 
\end{enumeratei}
Write $\MBRW$ for this process.  Write $\cT_{\varnothing}$ for the genealogical tree of the above process viewed as a marked tree, where the type of each vertex $v\in \MBRW$ is represented by the corresponding element in $\cS= \{\rO, \rY\}\times \bR$, i.e., its location and whether it is young or old. 
\end{defn}

\begin{rem}[Moving the random barrier to zero]
\label{rem-barrier-zero}
    One can equivalently construct the above branching random walk via translating the location of all particles by $-\cA$, so that the root starts at location $-\cA$ and a killing barrier at zero (see \cites{DerMor09,DerMor13} for a related model). The above formulation is slightly more convenient for the proof of Proposition \ref{prop:expec-s2-convg}(b) below. 
\end{rem}

\begin{rem}[Two-type killed branching random walk representation of general local limits]
\label{rem-two-type-killed-BRW}
\normalfont
For the general preferential attachment model with $a>0$ and $m\geq 2$ in \eqref{eqn:linear-att-def}, a similar description of the local limit follows. This is due to the fact that the older children in the tree (of which can be $m$ or $m-1$)
are located at $U^\alpha$ times the birth time of their parent in the tree, independently, where $\alpha=(2m+\delta)/(m+\delta)$. After taking a logarithm, these particles will be $\alpha \log(U)$ away from the age of their branching process parents. Finally, also the Poisson process of younger children is multiplicative, and the logarithmic transformation turns that into an addition. For the offspring operator, this is explained in \cite[Section 4.1]{HofZhu25}. A similar analysis can be done for the offspring distribution; see the discussion around \eqref{for:pointgraph:poisson-logarithmic} in Appendix \ref{sec-local-limit-UA}, where the logarithmic transformation is worked out in detail. Call the logarithm of the birth time of a particle its {\em location}. Then, in Appendix \ref{sec-local-limit-UA}, we prove that, after this transformation, the local limit of the general preferential and uniform attachment models are two-type killed branching random walks, where particles are killed when their location is above zero.
\hfill\ensymboldefinition
\end{rem}

We will sometimes refer to $\rO$ and $\rY$ as the {\bf relative types} of the vertices for reasons that will become clear below.   Let $\bT_{\cS}$ denote the set of all rooted trees with vertices marked by $\cS$. Consider the component of a randomly selected vertex $\clusterpi{o_n}{n},$ and view this as an element of $\bT_{\cS}$ via the following {\bf exploration} procedure:
\begin{enumeratei}
    \item The root $o_n$ has type $(\rO, 0)$. 
    \item Represent any neighbors $v\in [o_n-1]$ of $o_n$ in $\clusterpi{o_n}{n}$ as $(\rO, \log{v}- \log{o_n})$. Represent any neighbors $v\in[o_n+1,n]$ of $o_n$ in $\clusterpi{o_n}{n}$ by $(\rY, \log{v}- \log{o_n})$. Thus, the ``location'' of every vertex is $\log{v}- \log{o_n}$. 
    \item Explore the component $\clusterpi{o_n}{n}$ in a breadth-first manner. At each step, the vertices in the latest generation of the explored component are queried in increasing order of their locations and yet unexplored subsequent vertices, arising as their neighbors, are revealed. For any subsequent vertex $v_s$, its location is given by $\log{v_s}- \log{o_n}$. Its \emph{relative type} is $\rY$ if the edge connecting it to the currently explored component is incident on a vertex $v_e$ at an earlier location (that is, $\log{v_s}- \log{v_e}>0$), otherwise it is $\rO$.
\end{enumeratei}
We will call the above the {\em genealogical tree} of $\clusterpi{o_n}{n}$. To reduce terminology, we continue to use $\clusterpi{o_n}{n}$ for this object. 
\begin{prop}[Local convergence of connected components]
\label{prop:expec-s2-convg}
Let $\clusterpi{\sss\varnothing}{\infty}$ be the distribution of the genealogical tree of $\cT_{\varnothing}$ under the distribution $\pr_{\srO}$, so that  the root has label $\rO$. The following properties hold: 
\begin{enumeratea}
    \item The genealogical tree $\clusterpi{o_n}{n}$ of the component of a randomly selected vertex $o_n$ converges in the local weak convergence sense in probability to $\clusterpi{\sss\varnothing}{\infty}$. 
    \item For $\pi< \pi_c$, the expected size of the limit component size satisfies $\expec[|\clusterpi{\sss\varnothing}{\infty}|]=\susceptibilitypi{\infty}{2}$, where $\susceptibilitypi{\infty}{2}$ is defined in \eqref{eqn:s2-def}.
    \item $\limsup_{n\to\infty} \expec[\susceptibilitypi{n}{2}] \leq \expec[|\clusterpi{\sss\varnothing}{\infty}|]$. In particular, 
 $\expec[\susceptibilitypi{n}{2}] \to \susceptibilitypi{\infty}{2}$ as $n\to\infty$.  
\end{enumeratea}
\end{prop}

\begin{proof} We prove the different parts one by one.
    \begin{subproof}{a} This part (even for general $m\geq 2$ for the uniform attachment),  is proven in Appendix \ref{sec-local-limit-UA}. An alternative proof is given by Riordan in \cite{Rior05}. For a general overview of various notions of local weak convergence, see \cite{Hofs24}*{Chapter 2}.  One implication of this result is the following Lemma. See Lemma \cite{Hofs24}*{Chapter 5, Lemma 5.25} for a proof in the general setting of not just uniform attachment but preferential attachment. Define the constants,  $c_{\srY \srO}=1,$ while $c_{st}=2$ for all other choices of $s,t\in \{\rO,\rY\}.$ Next define the kernel $\kappa: (0,1]\times \cS\to \bR_+$ via, 
 \eqn{
    \label{def-kappa-UA}
        \kappa((x,s),(y,t))=\frac{c_{st}}{x\vee y}~,
    }
    \begin{lem}[Expectation of percolated connected component local limit]
    \label{lem:size-identity}
        For $\pi\in (0,1)$, the size of the local weak limit of the percolation cluster $|\clusterpi{\sss\varnothing}{\infty}|$ satisfies 
\[\E[|\clusterpi{\sss\varnothing}{\infty}|] = 1+ \sum\limits_{\ell=1}^{\infty} \pi^\ell \int\limits_0^1 \cdots \int\limits_0^1
        \sum\limits_{t_0, \ldots, t_\ell \in\{\srO,\srY\}}\prod\limits_{i=1}^\ell\kappa \Big( \big(x_{i-1},t_{i-1} \big),\big(x_{i},t_i\big)  \Big)\prod\limits_{i=0}^\ell \,dx_i ~.\]        
    \end{lem}
    The proof of this lemma can be found in the appendix. The proof of (c) provides an inkling of the origin of the kernel $\kappa$ in \eqref{def-kappa-UA}. 
\end{subproof}

\begin{subproof}{b}
Recall that $\cT_{\varnothing}$ denotes the geneological tree of killed multi-type branching random walk $\MBRW$ as in Definition \ref{def:mbrw}. Write 
$$x_{\srO} := \mathbb{E}_{\srO}[|\cT_{\varnothing}|] =  \mathbb{E}[|\clusterpi{\sss\varnothing}{\infty}|], \qquad x_{\srY} := \mathbb{E}_{\srY}[|\cT_{\varnothing}|] .$$ 
We will use the branching structure in the evolution of the process to obtain recursive relationships between $x_{\srO}, x_{\srY}$ whose solution completes the proof.  Let $\fP \sim \text{Bin}(2, \pi)$ denote the number of children of label $\rO$ in the tree.
Write these as $\{V_1, \dots, V_{\fP}\}$. Each such $V_i$ has age $-E_i$ compared to their parent in the tree, where $(E_i)_{i\in[\fP]}$ are iid exponential random variables with mean 1.

For each $i\in[\fP]$, let $S_i(-\infty, 0]$ denote the set of vertices emerging from $V_i$ of label $\rO$ and $\rY$ whose location is non-positive. Then,
\[
\mathbb{E} \left[ |S_i(-\infty, 0]| \mid \fP \right] = x_{\srO} \quad \forall i\in[\fP].
\]
Recall that $\cA \sim {\mathrm Exp}(1)$ denotes the killing barrier.  Further, we say that $u$ is a {\em descendant} of $v$ in the tree when the path from the root to $u$ passes through $v$.  We use the above to derive a recursive relation based on the observation that
$\cT_{\varnothing}$ comprises of  
$\{\varnothing\} \cup \left(\bigcup_{i=1}^{\fP} S_i(-\infty, 0] \right),$
along with all vertices with location $ \le \cA$ that are descendants of the children in the tree of every vertex in $\{\varnothing\} \cup \left(\bigcup_{i=1}^{\fP} S_i(-\infty, 0] \right)$ whose location is greater than $0$. These children  must all have label $\rY$, because all the $\rO$ descendants are already in $\bigcup_{i=1}^{\fP} S_i(-\infty, 0]$. Note that
\begin{equation}
    \label{eqn:rt1}
    \mathbb{E} \left[ \left| \{\varnothing\} \cup \left(\bigcup_{i=1}^{\fP} S_i(-\infty, 0] \right) \right| \right] = 
    1+\expec[\fP]\expec[S_1(-\infty, 0]\mid \fP]=1 + 2\pi x_{\srO}.
\end{equation}

Further,  for any vertex $V \in \{\varnothing\} \cup \left(\bigcup_{i=1}^{\fP} S_i(-\infty, 0] \right)$, all its $\rY$ children in the tree are born at epochs of a Poisson process (PP) with rate $2\pi$. Let 
    \eqn{
    \label{intensity-Y-children-UA}
    t_1, t_2, \ldots \sim PP(2\pi)
    }
be the birth times of its children of positive age (which must have label $\rY$). By the memoryless property of the exponential distribution, for any $i \in \mathbb{N}$, $x > 0$,

\[
\pr(\cA - t_i \geq x \mid  t_i \leq \cA) = \e^{-x}.
\]

Let $v_i$ denote the vertex born at time $t_i$. We say that $u$ is a {\em positive $\fP$ descendant} of $v_i$ if $u$ is a descendant of $v_i$, and the path connecting $u$ to $v_i$ does not pass through any vertex in $\{\varnothing\} \cup \left(\bigcup_{i=1}^{\fP} S_i(-\infty, 0] \right)$. Define
\begin{equation}
    T_i =\left\{
    u \in \cT_{\varnothing}\colon  u \text{ is a positive $\fP$ descendant of } v_i 
    \right\}.
\end{equation}
Hence, on the event $\{ t_i \leq \mathcal{A} \}$, the conditional expectation of $|T_i|$ given the history up till time $t_i$ (all vertices born before this time along with their connectivity structure) equals $x_{\srY}$, i.e., $\expec[|T_i|\mid t_i\leq \mathcal{A}]=x_{\srY}$. This follows from the memoryless property of $\mathcal{A}$: Conditionally on the event $\mathcal{A} \ge t_i$, the random variable $\mathcal{A} - t_i$, which is the `evolution time' of $T_i$ from the birth of its ancestor $v_i$ at time $t_i$ to $\mathcal{A}$, is an exponential random variable with mean one.

From standard facts about exponential random variables, $\prob(\mathcal{A} \ge t_1) = 2\pi/(2\pi + 1)$ and, for any $i \ge 2$,
$$
\prob(\mathcal{A} \ge t_i \mid \mathcal{A} \ge t_{i-1}) = \frac{2\pi}{2\pi + 1}.
$$
Combining the above observations, we obtain
\eqn{
\label{expec-particles-Y}
\mathbb{E} \left[ \sum_{i=1}^{\infty} \indic{t_i \leq \mathcal{A}} |T_i| \right]
= \sum_{i=1}^{\infty} \prob(t_i \leq \mathcal{A}) x_{\srY}
= \sum_{i=1}^{\infty} \left(\frac{2\pi}{2\pi+1}\right)^i x_{\srY} = 2\pi x_{\srY}.
}
As \eqref{expec-particles-Y} holds for any vertex $V \in \{\varnothing\} \cup \left(\bigcup_{i=1}^{\fP} S_i(-\infty, 0] \right)$, the contribution to the size of $\cT_{\varnothing}$ by $V$ and its positive $\fP$ descendants is $1 + 2\pi x_{\srY}$.

Combining this with \eqref{eqn:rt1}, we obtain
\begin{equation}
    \label{eqn:rt-2}
    x_{\srO} = \mathbb{E}[|\clusterpi{\sss\varnothing}{\infty}|] = (1 + 2\pi x_{\srO}) (1 + 2\pi x_{\srY}). 
\end{equation}
By a similar argument, if the root label is $\rY$, then it has $\fP\sim \text{Ber}(\pi)$ children of label $\rO$. Now, using the same description, with $S_i(-\infty, 0]$ defined above and this new $\fP$, for the descendants of the root and its label $\rO$ children, we have
\eqn{
x_{\srY} = (1 + \pi x_{\srO}) (1 + 2\pi x_{\srY}). 
}
Combining the above two identities gives 
\[
\frac{x_{\srO}}{x_{\srY}} = \frac{1 + 2\pi x_{\srO}}{1 + \pi x_{\srO}}, \]
which implies that
    \[x_{\srY} = \frac{(1 + \pi x_{\srO}) x_{\srO}}{1 + 2\pi x_{\srO}}.
    \]
Substituting this into \eqref{eqn:rt-2} gives 
\[
x_{\srO} = (1 + 2\pi x_{\srO}) \left( 1 + \frac{2\pi x_{\srO} (1 + \pi x_{\srO})}{1 + 2\pi x_{\srO}} \right).
\]

Solving for $x_{\srO}$ gives
\[
x_{\srO} = \frac{1 - 4\pi \pm \sqrt{8\pi^2 - 8\pi + 1}}{4\pi^2}.
\]
Since $\pi\mapsto x_{\srO}=x_{\srO}(\pi)$ is continuous for $\pi < {\pi}_c$ and satisfies $\lim_{\pi \searrow 0}x_{\srO}(\pi) = 1$, the only feasible solution is
\[
\mathbb{E} [|\clusterpi{\sss\varnothing}{\infty}|]  = \frac{1 - 4\pi - \sqrt{8\pi^2 - 8\pi + 1}}{4\pi^2},
\]
which agrees with $\susceptibilitypi{\infty}{2}$ in \eqref{eqn:s2-def}.
\end{subproof}

\begin{subproof}{c}
By definition,
    \eqan{\label{for:s2:UB-1}
    \expec\big[ \susceptibilitypi{n}{2} \big] = \expec [|\clusterpi{o_n}{n}|],
    }
    where $o_n$ is a uniformly chosen vertex in $G^{\pi}_n$.
    Note that whenever a vertex \(u\) is in \( \clusterpi{o_n}{n} \), there exists a path from \(o_n\) to \(u\) in \(G^{\pi}_n\), and vice-versa. By a union bound,
    \eqn{\label{for:s2:UB-2}
        \expec\big[ \susceptibilitypi{n}{2} \big] \leq  \frac{1}{n}\sum\limits_{u,v\in[n]}\sum_{{\bf p}_\ell(v,u)} \prob ({\bf p}_\ell(v,u) \subseteq G^{\pi}_n),
    }
    where the sum over ${\bf p}_\ell(v,u)$ is over all {\em edge-labeled} self-avoiding paths connecting $v$ and $u$, with the (distinct) labels on the edges indicating the edge label that made the connection, which is an element in $\{1,2\}$.
    Below, we will consistently write ${\bf p}_\ell(v,u)$ for an edge-labeled self-avoiding path, and ${\bf p}(v,u)$ for an unlabeled self-avoiding path. We write that $e\in {\bf p}(v,u)$ when the path ${\bf p}(v,u)$ contains the edge $e$. For an edge $e=\{u_{i-1}, u_i\}$ in a path \( {\bf p}(v,u)=\{ v=u_0, u_1,\cdots,u_{\sss d(u,v)}=u \} \), where $d(u,v)$ denotes the path length, we write $\bar{e}=u_{i}, \, \underline{e}=u_{i-1}$. 
    We note that, by the independence of the edge connections,
        \eqn{\label{for:s2:UB-2a}
        \prob ({\bf p}_\ell(v,u) \subseteq G^{\pi}_n)
        =\prod_{e\in {\bf p}_\ell(v,u)} \frac{\pi}{(\bar{e}\vee \underline{e})},
        }
    which is {\em independent} of the actual edge labeling. Thus,
        \eqn{\label{for:s2:UB-2b}
        \expec\big[ \susceptibilitypi{n}{2} \big] \leq  \frac{1}{n}\sum\limits_{u,v\in[n]}\sum_{{\bf p}(v,u)} N({\bf p}(v,u))\prod_{e\in {\bf p}(v,u)} \frac{\pi}{(\bar{e}\vee \underline{e})},
    }
where, for an unlabeled path ${\bf p}(v,u)$, we let $N({\bf p}(v,u))$ denote the total number of edge-labeled paths consistent with ${\bf p}(v,u)$.
  \smallskip
    
We next compute $N({\bf p}(v,u))$. Given a path \( {\bf p}(v,u)=\{ v=u_0, u_1,\cdots,u_{\sss d(u,v)}=u \} \), define the labels of the nodes in \({\bf p}(v,u)\) as
        \[
            \lab_{\bf{p}}(u_i)=\begin{cases}
            \rO \quad \text{if} \quad u_i<u_{i-1};\\
            \rY \quad \text{if} \quad u_i>u_{i-1},
            \end{cases}
        \]
where, by convention, \(\lab_{\bf p}(u_0)=\rO\). Recall the constants defined above Lemma \ref{lem:size-identity}, $c_{\srY \srO}=1,$ while $c_{st}=2$ for all other choices of $s,t\in \{\rO,\rY\}.$ Then,
    \eqn{
    \label{label-counting-path}
    N({\bf p}(v,u))=\prod_{e\in {\bf p}(v,u)} c_{\lab_{\bf p}(\underline{e}),\lab_{\bf p}(\bar{e})}.
    }
Indeed, if $u_{\ell-1}$ is older than $ u_\ell$ (that is, $\lab_{\bf{p}}(u_\ell) = \rY$), an outgoing edge of $u_{\ell}$ has already been revealed in the segment $\{u_0,\dots, u_\ell\}$ and, if $\lab_{\bf{p}}(u_{\ell+1}) = \rO$, we have only one choice for the edge $(u_\ell, u_{\ell+1})$. For all other cases, we have $m$ choices for this edge. This gives rise to the factors involving $c_{st}$. 
We arrive at  
    \eqan{\label{for:s2:UB-3a}
        \expec\big[ \susceptibilitypi{n}{2} \big] 
        &\leq 1+ \frac{1}{n}\sum\limits_{\ell=1}^{\infty} \sum\limits_{u_0,u_1,\cdots,u_{\ell}\in[n]}
        \prod\limits_{i=1}^{\ell} \pi \frac{c_{\lab_{\bf p}(u_{i-1}),\lab_{\bf p}(u_{i})}}{(u_{i-1}\vee u_i)}.
        }
\smallskip
We can further rewrite this as
        \eqan{
        \label{for:s2:UB-3}
        \expec\big[ \susceptibilitypi{n}{2} \big] &\leq 1+ \frac{1}{n}\sum\limits_{\ell=1}^{\infty} \sum\limits_{u_0,u_1,\cdots,u_{\ell}\in[n]} \prod\limits_{i=1}^{\ell} \tfrac{\pi}{n} \kappa \Big( \big(\tfrac{u_{i-1}}{n},\lab_{\bf p}(u_{i-1})\big), \big(\tfrac{u_{i}}{n},\lab_{\bf p}(u_{i})\big) \Big)~,
    }
    where \(\kappa\) is the kernel of the mean-offspring operator for the local limit as in Lemma \ref{lem:size-identity}.
    Upper bounding the sums by integrals, using that, for $a,b>0$,
        \eqn{
        \frac{1}{\lceil a\vee b\rceil}\leq \frac{1}{a\vee b},
        }
    we can upper bound
        \(\expec\big[ \susceptibilitypi{n}{2} \big]\) 
    by
    \eqn{\label{for:s2:UB-4}
        \expec\big[ \susceptibilitypi{n}{2} \big] \leq 1+ \sum\limits_{\ell=1}^{\infty} \pi^\ell \int\limits_0^1 \cdots \int\limits_0^1
        \sum\limits_{t_0, \ldots, t_\ell \in\{\srO,\srY\}}\prod\limits_{i=1}^\ell\kappa \Big( \big(x_{i-1},t_{i-1} \big),\big(x_{i},t_i\big)  \Big)\prod\limits_{i=0}^\ell \,dx_i ~,
    }
where, by convention, $t_0=\rO$. Now using Lemma \ref{lem:size-identity} shows that
\eqn{\label{for:s2:UB-5}
        \expec\big[ \susceptibilitypi{n}{2} \big] \leq \expec \big[ |\clusterpi{\varnothing}{\infty}| \big]~.
    }
and thus completes the first assertion of part(c) namely the upper bound.

    To prove the convergence of expectations, note that by part (a), $|\clusterpi{o_n}{n}|$ converges weakly to $|\clusterpi{\sss\varnothing}{\infty}|$. Thus, by Fatou's Lemma,
    $$
    \liminf_{n \to \infty}\expec\big[|\clusterpi{o_n}{n}|\big] \ge \expec \big[ |\clusterpi{\sss \varnothing}{\infty}| \big]~.
    $$
    The assertion 
    $$
    \lim_{n \to \infty}\expec\big[ \susceptibilitypi{n}{2} \big] = \lim_{n \to \infty}\expec\big[|\clusterpi{o_n}{n}|\big] = \expec \big[ |\clusterpi{\sss \varnothing}{\infty}| \big]
    $$
    now follows by combining this observation with the upper bound.
\end{subproof}

This completes the proof of  Proposition \ref{prop:expec-s2-convg}.
\end{proof}

\section{Weak upper bounds on component sizes}
\label{sec-weak-bounds-components}
In this section, we obtain upper bounds on the sizes of connected components, as well as the maximal component upto an arbitrarily small but positive correction to the true scaling exponent. These bounds are later bootstrapped to remove the correction and play a crucial role in the proofs of those results. We start in Section \ref{sec-weak-bound-expected-component} by proving a bound on the expected size of the component of a vertex $v\in[n]$. This bound is crucially used in Section 
\ref{sec-weak-almost sure-bound-component} to prove a weak almost sure upper bound on the maximal component.

\subsection{Weak upper bound on component size of a given vertex}
\label{sec-weak-bound-expected-component}

For $v \in [n]$, recall  that $\clusterpi{v}{n}$ denotes the component containing $v$ in $\Gpi{n}$.

\begin{prop}[Weak upper bound on expected component size]
\label{prop:expec-upper-bound}
    For any $\beta \in \left(\alpha(\pi)  
    ,\tfrac{1}{2} \right)$, there exists $C_{\sss \beta} > 0$ such that, for all $v\in [n]$,
\[
 \mathbb{E} \left[ | \clusterpi{v}{n} | \right] \leq C_{\sss \beta} \left(\frac{n}{v}\right)^\beta.
\]
\end{prop}

To prove Proposition \ref{prop:expec-upper-bound}, we write the expected cluster 
size of \(v\) as
\eqn{\label{eq:cluster-size:path-expansion}
    \expec\big[ | \clusterpi{v}{n} | \big] = \sum\limits_{u\in[n]}\prob\big(\{u\leftrightsquigarrow v\}\subseteq G_n^\pi\big),
}
where \(\{u\leftrightsquigarrow v\}\) refers to the event that \(u\) and \(v\) are connected. To obtain an upper bound, we use the union bound for the right hand side of \eqref{eq:cluster-size:path-expansion} as
\eqn{\label{eq:cluster-size:union-bound}
    \expec\big[ | \clusterpi{v}{n} | \big] \leq \sum\limits_{u\in[n]}\sum\limits_{{\bf p}(v,u)}\prob\big( {\bf p}(v,u)\subseteq G_n^\pi \big),
}
where the sum over \({\bf p}(v,u)\) is over all self-avoiding paths between \(v\) and \(u\), i.e., paths that visit all vertices at most once. To upper bound the rhs of \eqref{eq:cluster-size:union-bound}, we
make use of \emph{labeled paths}, a generalization of which is \emph{(vertex)-labeled trees}, defined as follows:

\begin{defn}[(Vertex-)labeled tree]\label{defn:labeled-tree}
    For any finite rooted tree \(\Tree\) with vertices in \([N]\) for some \(N\in\N\), we define the \emph{labeled tree} \(\Tree_\ell\) by assigning the vertices of the tree labels in \(\{\rO,\rY\}\). For any vertex \(u \in \Tree\), its label in \(\Tree_\ell\) is given by
    $$
        \lab_{\mathrm{ \textsf{T}}}(u) = \begin{cases}
            &\rO\quad \text{if}~u~\text{is the root of}~\Tree;\\
            &\rY\quad \text{if}~\mathfrak{p}(u)<u;\\
            &\rO\quad \text{if}~\mathfrak{p}(u)>u,
        \end{cases}
    $$
    where \(\mathfrak{p}(u)\) is the parent of \(u\) in \(\Tree\).
\end{defn}

\begin{proof}[Proof of Proposition~\ref{prop:expec-upper-bound}]
Let $v \in [n]$. We will prove a slightly stronger statement, namely, that
        \eqan{
        \label{base-case:1}
        \sum\limits_{u\in[n]}\sum\limits_{{\bf p}(v,u)}\prob\big( {\bf p}(v,u)\subseteq G_n^\pi \big)\big(\frac{n}{u}\big)^\beta
        \leq C_\beta \big(\frac{n}{v}\big)^\beta.
        }
     For a self-avoiding `path' $\boldsymbol{\omega} = (\omega_0 = v, \omega_1, \dots, \omega_d) \in [n]^{d+1}$ of length $d$, rooted at $\omega_0$, we will denote the corresponding (vertex-)labeled path by
    $$
    \boldsymbol{\omega}_\ell = \left((\omega_0,\lab_{\boldsymbol{\omega}}(\omega_0)),(\omega_1,\lab_{\boldsymbol{\omega}}(\omega_1)),\ldots,(\omega_d,\lab_{\boldsymbol{\omega}}(\omega_d))\right),
    $$
    where $\lab_{\boldsymbol{\omega}}$ is as defined in Definition \ref{defn:labeled-tree}.
    
    We will denote by ${\bf P}_d(v,u)$ the set of all vertex-labeled paths from \(v\) to \(u\) of length \(d\).
    Note that $d=0$ for \(u=v\). Therefore,
    \eqan{\label{base-case:1a}
        \sum\limits_{u\in[n]}\sum\limits_{{\bf p}(v,u)}\prob\big( {\bf p}(v,u)\subseteq G_n^\pi \big)\big(\frac{n}{u}\big)^\beta = \sum\limits_{u\in[n]} \sum\limits_{d=0}^\infty \sum\limits_{\boldsymbol{\omega}_\ell \in {\bf P}_d(v,u)} \prob\big( \boldsymbol{\omega}_\ell \subseteq  G_n^\pi \big)\big(\frac{n}{u}\big)^\beta~.
    }
    For \(u\neq v\) and \(d=0\), we set the probability on the rhs to be \(0\). 

    By Fubini's Theorem, for any fixed \(d\geq 0\), we can thus write
    \eqan{\label{base-case:2}
        \sum\limits_{u\in[n]} \sum\limits_{\boldsymbol{\omega}_\ell \in {\bf P}_d(v,u)} \prob\big( \boldsymbol{\omega}_\ell \subseteq  G_n^\pi \big)\big(\frac{n}{u}\big)^\beta = \pi^d\sum\limits_{\omega_1,\cdots,\omega_{d}\in[n]} 
        \big(\frac{n}{\omega_d}\big)^\beta
        \prod\limits_{i=1}^d \frac{c_{\sss \lab_{\boldsymbol{\omega}}(\omega_{i-1})\lab_{\boldsymbol{\omega}}(\omega_i)}}{(\omega_{i-1}\vee \omega_{i})},
    }
    where $\boldsymbol{\omega}_\ell$ denotes the vertex-labeled path corresponding to $\boldsymbol{\omega} = (\omega_0 = v, \omega_1, \dots, \omega_d) \in [n]^{d+1}$, and as before Lemma \ref{lem:size-identity}, \(c_{\srY\srO} =1\), while \(c_{\rm st}=2\) for any other choice of \(s,t\in\{\rO,\rY\}\). 
    Let \[M_\beta=\begin{bmatrix}
        \tfrac{2}{1-\beta} & \tfrac{2}{\beta}\\
        \tfrac{1}{1-\beta} & \tfrac{2}{\beta}
    \end{bmatrix},\]
    where the rows and columns are ordered as $(\rO,\rY)$, and let
    $$
    \lambda_\beta=\frac{1 +\sqrt{1-2\beta(1-\beta)}}{\beta(1-\beta)}
    $$
    denote the largest eigenvalue of \(M_\beta\), and \(\boldsymbol{\nu}_{\sss \beta}=(\nu_{\srO},\nu_{\srY})\) the eigenvector corresponding to \(\lambda_\beta\).
    Writing $\varphi_\beta(u)=u^{-\beta}$, we bound the rhs of \eqref{base-case:2} as
    \eqan{\label{base-case:3}
        &\sum\limits_{\omega_1,\cdots,\omega_{d}\in[n]} \big(\frac{n}{\omega_d}\big)^\beta\prod\limits_{i=1}^d \frac{c_{\sss \lab_{\boldsymbol{\omega}}(\omega_{i-1})\lab_{\boldsymbol{\omega}}(\omega_i)}}{(\omega_{i-1}\vee \omega_{i})} \nn\\
        &\hspace{3cm}\le c_\beta \big(\frac{n}{v}\big)^\beta\sum\limits_{\omega_1,\cdots,\omega_{d}\in[n]} \frac{\nu_{\lab_{\boldsymbol{\omega}}(\omega_d)}\varphi_{\beta}(\omega_d)}{\nu_{\lab_{\boldsymbol{\omega}}(\omega_0)}\varphi_\beta(\omega_0)}\prod\limits_{i=1}^d \frac{c_{\sss \lab_{\boldsymbol{\omega}}(\omega_{i-1})\lab_{\boldsymbol{\omega}}(\omega_i)}}{(\omega_{i-1}\vee \omega_{i})},
    }
    for some constant \(c_\beta>0\), depending only on the choice of \(\beta\).
    Now, observe that for any $j \in [d]$, the label $\lab_{\boldsymbol{\omega}}(\omega_j)$ in the path $\boldsymbol{\omega}$ is determined by $\omega_0, \omega_1, \dots, \omega_{j-1}$ and whether $\omega_j < \omega_{j-1}$ or not. Thus, for given $\omega_0, \omega_1, \dots, \omega_{j-1}$, using an integral bound and the eigenfunction property of \(M_\beta\),
    \eqan{\label{base-case:4}
        &\sum\limits_{\omega_j \in[n]} \frac{\nu_{\lab_{\boldsymbol{\omega}}(\omega_j)}\varphi_{\beta}(\omega_j)}{\nu_{\lab_{\boldsymbol{\omega}}(\omega_{j-1})}\varphi_\beta(\omega_{j-1})} \,\frac{c_{\sss \lab_{\boldsymbol{\omega}}(\omega_{j-1})\lab_{\boldsymbol{\omega}}(\omega_j)}}{\omega_{j-1}\vee \omega_{j}}\nn\\
        &\hspace{2cm}= \frac{1}{\nu_{\lab_{\boldsymbol{\omega}}(\omega_{j-1})}\varphi_\beta(\omega_{j-1})}\left(\sum\limits_{z=1}^{\omega_{j-1}-1}\frac{c_{\sss \lab_{\boldsymbol{\omega}}(\omega_{j-1})\srO}\nu_{\srO}}{\omega_{j-1}z^{\beta}}+\sum\limits_{z=\omega_{j-1}+1}^n\frac{c_{\sss \lab_{\boldsymbol{\omega}}(\omega_{j-1})\srY}\nu_{\srY}}{z^{1+\beta}}\right)\nn\\
        &\hspace{2cm}\leq \frac{1}{\nu_{\lab_{\boldsymbol{\omega}}(\omega_{j-1})}}\left(\left(M_\beta\right)_{\sss \lab_{\boldsymbol{\omega}}(\omega_{j-1})\srO}\nu_{\srO}
        +\left(M_\beta\right)_{\sss \lab_{\boldsymbol{\omega}}(\omega_{j-1})\srY}\nu_{\srY}\right)\nn\\
        &\hspace{2cm}= \lambda_\beta.
    }
    By \eqref{base-case:3} and \eqref{base-case:4}, we bound the lhs of \eqref{base-case:2} as
    \eqan{\label{base-case:5}
       \sum\limits_{u\in[n]} \sum\limits_{\boldsymbol{\omega}_\ell \in {\bf P}_d(v,u)} \prob\big( \boldsymbol{\omega}_\ell \subseteq  G_n^\pi \big)\big(\frac{n}{u}\big)^\beta \leq c_\beta (\pi\lambda_\beta)^d\big(\frac{n}{v}\big)^\beta.
    }
    For the rhs of \eqref{base-case:5} to be summable, we need \(\pi\lambda_\beta<1\) for all \(\pi< \pi_c.\) Now solving for \(\beta\) in \(\pi\lambda_\beta=1\) and choosing the smallest solution, we obtain that \(\pi\lambda_\beta<1\) for \(\beta\in(\alpha(\pi),\tfrac{1}{2})\) and all \(\pi<\pi_c\).
    
    Thus, for $\pi < \pi_c$ and \(\beta\in(\alpha(\pi),\tfrac{1}{2})\), summing over \(d\) and choosing \(C_\beta=\tfrac{c_\beta}{1-\pi\lambda_\beta}\), we bound the lhs of \eqref{base-case:1} as
    \eqn{\label{base-case:6}
        \sum\limits_{u\in[n]}\sum\limits_{{\bf p}(v,u)}\prob\big( {\bf p}(v,u)\subseteq G_n^\pi \big)\big(\frac{n}{u}\big)^\beta \leq C_\beta \big(\frac{n}{v}\big)^\beta.
    }
    Hence, \eqref{eq:cluster-size:union-bound} and \eqref{base-case:6} complete the proof of Proposition~\ref{prop:expec-upper-bound}.
\end{proof}

\subsection{Weak almost sure upper bounds on the size of the maximal component}
\label{sec-weak-almost sure-bound-component}
The goal of this section is to prove the following proposition: 

\begin{prop}[Weak upper bound on the maximal component]
\label{prop:max-as-up-bound}
    For all $\vep>0$, the maximal component satisfies that, almost surely,
    $$\limsup\limits_{n \rightarrow \infty}|\clusterpi{\max}{n}|/n^{\alpha(\pi) + \vep} \leq 1~.$$
\end{prop}

The proof of Proposition \ref{prop:max-as-up-bound} is broken into two steps, bounding the connected component size of the \emph{late} vertices, followed by bounding the connected component size of the \emph{early} vertices. We start with the early vertices. Recall that $\clusterpi{i}{n}$ denotes the component of vertex $i$. For \(N\in[n]\), define
\eqn{\label{def:early-cluster}
    \clusterpi{[N]}{n} = \bigcup\limits_{u\in[N]} \clusterpi{u}{n}~.
}
The following lemma provides an upper bound for the size of the union of component sizes containing the early vertices:

\begin{lem}[Almost sure upper bound on early components]
\label{lem-almost-sure-upper-bound-early-UA}
    For all $\vep >0$ and \( a(\vep)< \vep/8\), almost surely, for all sufficiently large $n$, 
    \eqn{
    |\clusterpi{[n^{a(\vep)}]}{n}| \leq n^{\alpha(\pi)+\vep}.
    }
\end{lem}
We next show that the sizes of the other components are significantly smaller than the ones containing the early vertices. Recall the definition of $\clusterpi{\sss \geq v}{n}$ from \eqref{eqn:cc-less-def}. 

\begin{lem}[Bound on the component sizes of late vertices]\label{lem-almost-sure-upper-bound-late-UA}
    Given any $a >0$, there exists $\vep = \vep(a)>0$ such that, almost surely, for all sufficiently large $n$,
    \[\max_{n^a \leq i \leq  n} |\clusterpi{\sss \geq i}{n}| \leq n^{\alpha(\pi) - \vep}.\]
\end{lem}
We first prove Proposition~\ref{prop:max-as-up-bound} subject to Lemmas~\ref{lem-almost-sure-upper-bound-early-UA} and \ref{lem-almost-sure-upper-bound-late-UA}:

\begin{proof}[Proof of Proposition~\ref{prop:max-as-up-bound} subject to Lemmas~\ref{lem-almost-sure-upper-bound-early-UA} and \ref{lem-almost-sure-upper-bound-late-UA}]
    Fixing any \(\vep>0\), by Lemma~\ref{lem-almost-sure-upper-bound-early-UA}, we obtain \(a(\vep)>0\) such that, almost surely, for all sufficiently large \(n\),
    \eqn{\label{for:prop-as-UB-1}
        |\clusterpi{[n^{a(\vep)}]}{n}| \leq n^{\alpha(\pi)+\vep}~.
    }
    On the other hand, by Lemma~\ref{lem-almost-sure-upper-bound-late-UA}, there exists \(\vep^\prime>0\), such that, almost surely for sufficiently large \(n\),
    \eqn{\label{for:prop-as-UB-2}
        \max_{n^{a(\vep)} \leq v \leq  n} |\clusterpi{\sss \geq v}{n}| \leq n^{\alpha(\pi) - \vep^\prime}~.
    }
    Observing that
    \[
        |\clusterpi{\rm max}{n}| \leq \max\left\{ |\clusterpi{[n^{a(\vep)}]}{n}|~,\max_{n^{a(\vep)} \leq v \leq  n} |\clusterpi{\sss \geq v}{n}| \right\}~,
    \]
    the proof of Proposition \ref{prop:max-as-up-bound} follows immediately from \eqref{for:prop-as-UB-1} and \eqref{for:prop-as-UB-2}.
\end{proof}
Next, we prove Lemmas~\ref{lem-almost-sure-upper-bound-early-UA} and \ref{lem-almost-sure-upper-bound-late-UA}. We start by proving Lemma~\ref{lem-almost-sure-upper-bound-early-UA}:
\begin{proof}[Proof of Lemma~\ref{lem-almost-sure-upper-bound-early-UA}]
    We prove the assertion first for the subsequence of times $n = n_k = \e^k$ for $k\geq 1$, and show that, for all $\vep >0$, there exists \(a(\vep)>0,\) such that almost surely, for sufficiently large $k$,
    \begin{equation}
    | \clusterpi{[n_k]}{n_k}| \leq C n_k^{\alpha(\pi)+\vep}~,
    \label{eqn:942}
    \end{equation}
    for some finite constant \(C\). By Proposition~\ref{prop:expec-upper-bound}, for all \(\beta\in(\alpha(\pi),\tfrac{1}{2})\), 
    \eqn{\label{for:lem:as-UB-1}
        \expec \left[| \clusterpi{[n_k^a]}{n_k} | \right] \leq C_{\sss \beta} \sum\limits_{u\in[n_k^a]}\big( \tfrac{n_k}{u} \big)^{\beta} \leq C_{\sss \beta} n_k^{\beta+a(1-\beta)}~. 
    }
    By Markov's inequality,
    \eqn{\label{for:lem:as-UB-2}
        \prob \left( | \clusterpi{[n_k^a]}{n_k} | > n_k^{\alpha(\pi)+\vep} \right) \leq C_{\sss \beta} n_k^{\beta - \alpha(\pi)-\vep+a(1-\beta)}~.
    }
    Choosing \(\beta=\alpha(\pi)+\vep/2\), and \(a(\vep)<\vep/4\), the exponent of \(n_k\) in the rhs of \eqref{for:lem:as-UB-2} can be shown to be negative, and thus, summing over \(k\in\mathbb{N}\), 
    \eqn{\label{for:lem:as-UB-3}
        \sum\limits_{k=1}^{\infty} \prob \left( | \clusterpi{[n_k^a]}{n_k} | > n_k^{\alpha(\pi)+\vep} \right) \leq C_{\sss \beta} \sum\limits_{k=1}^{\infty} n_k^{\beta - \alpha(\pi)-\vep+a(1-\beta)} < \infty~.
    }
    By the Borel-Cantelli lemma, \eqref{eqn:942} holds true. Now, it remains to prove the same for other values of \(n\) as well. Note that, for all \(n\in (n_{k-1},n_k),\)
    \eqn{\label{for:lem:as-UB-4}
        |\clusterpi{[n^a]}{n}|\leq|\clusterpi{[n_k^a]}{n}|\leq |\clusterpi{[n_k^a]}{n_k}|~. 
    }
    Since \(n\in (n_{k-1},n_k)\), by \eqref{eqn:942} and \eqref{for:lem:as-UB-4}, almost surely, for \(n\) sufficiently large, and \(a(\vep)<\vep/8\),
    \eqn{\label{for:lem:as-UB-5}
        |\clusterpi{[n^{a(\vep)}]}{n}| \leq C_{\sss \beta} \e^{\alpha(\pi)+\vep/2} n^{\alpha(\pi)+\vep/2}.
    }
    For \(n\) sufficiently large, we bound \(C_{\sss \beta} \e^{\alpha(\pi)+\vep/2}\) by \(n^{\vep/2}\), to complete the proof of Lemma~\ref{lem-almost-sure-upper-bound-early-UA}.
\end{proof}
Finally, it remains to prove Lemma~\ref{lem-almost-sure-upper-bound-late-UA}, for which we use Markov's inequality on the higher moments of \(|\clusterpi{\sss \geq v}{n}|\) for each $v$, followed by a use of the Borel-Cantelli lemma to obtain the required almost sure result. In the process, we need to bound all moments of \(|\clusterpi{\sss \geq v}{n}|\):
\begin{lem}[Tree-graph bounds on component sizes]
\label{lem-higher-moment-bound}
    For any \(k\geq 1\), and \(\beta\in(\alpha(\pi),\tfrac{1}{2}),\) there exists  $C_k=C_k(\beta) < \infty$ such that, for all $v\in[n],$
    \[
        \expec \Big[ \big|\clusterpi{\sss \geq v}{n}\big|^k \Big] \leq C_k \big( \frac{n}{v} \big)^{k\beta}~.
    \]
\end{lem}
We first prove Lemma~\ref{lem-almost-sure-upper-bound-late-UA} subject to Lemma~\ref{lem-higher-moment-bound}, and then prove Lemma~\ref{lem-higher-moment-bound}:
\begin{proof}[Proof of Lemma~\ref{lem-almost-sure-upper-bound-late-UA} subject to Lemma~\ref{lem-higher-moment-bound}]
    By Markov's inequality, for all \(k\),
    \eqn{\label{for:ASUBL-1}
        \prob\Big( \big|\clusterpi{\sss \geq v}{n}\big| > n^{\alpha(\pi) - \vep} \Big) \leq n^{-k(\alpha(\pi) - \vep)} \expec \Big[ \big|\clusterpi{\sss \geq v}{n}\big|^k \Big]~.
    }
    By Lemma~\ref{lem-higher-moment-bound}, the rhs of \eqref{for:ASUBL-1} can be upper bounded as
    \eqn{\label{for:ASUBL-2}
        \prob\Big( \big|\clusterpi{\sss \geq v}{n}\big| > n^{\alpha(\pi) - \vep} \Big) \leq C_k n^{-k(\alpha(\pi) - \vep)} \big( \frac{n}{v} \big)^{k\beta} = C_k n^{k(\beta-\alpha(\pi) +\vep)} v^{-k\beta}~.
    }
    By the union bound, the maximum component for $v\in [n^a,n]$ can be bounded as
    \eqan{\label{for:ASUBL-3}
        \prob\Big( \max\limits_{n^a\leq v \leq n} \big|\clusterpi{\sss \geq v}{n}\big| > n^{\alpha(\pi) - \vep} \Big) &\leq C_k n^{k(\beta-\alpha(\pi)+\vep)} \sum\limits_{v=n^a}^n v^{-k\beta}\nn\\
        &\leq C_k^\prime n^{k(\beta-\alpha(\pi)+\vep)+a(1-k\beta)}~,
    }
    for some finite constant \(C_k^\prime\). We now choose \(k\) so large that, for  \(\beta=\alpha(\pi)+\vep\), the exponent of \(n\) in the rhs of \eqref{for:ASUBL-3} is less than \(-1\), i.e.,
    \[
        k(\beta-\alpha(\pi)+\vep)+a(1-k\beta)<-1~.
    \]
    Therefore, choosing \(\vep=\vep(a)<a\alpha(\pi)/2,~\beta=\alpha(\pi)+\vep\) and  \(k>\frac{a+1}{a\alpha(\pi) - 2\vep}\), the rhs of \eqref{for:ASUBL-3} is summable. Hence, by the Borel-Cantelli lemma, the almost sure bound holds true for all sufficiently large $n$, completing the proof of Lemma~\ref{lem-almost-sure-upper-bound-late-UA}.
\end{proof}
Lastly, we prove Lemma~\ref{lem-higher-moment-bound}:

\begin{proof}[Proof of Lemma~\ref{lem-higher-moment-bound}]
Our proof is inspired by the {\em tree-graph inequalities} in percolation theory \cite{AizNew84}, which we adapt to this setting. For $j\geq 1$, we let $(j)_k=j(j-1)\cdots (j-k+1)$ be the falling factorial. We then note that, for any non-negative integer-valued random variable,
        \eqn{
        \expec\big[X^k\big]
        \leq k^k+\frac{(k+1)^k}{k\,!}\expec\big[(X-1)_k\big],
        }
since
    \[
    \frac{j^k}{(j-1)_k}=\prod_{i=1}^{k} \frac{j}{j-i}
    \leq \prod_{i=1}^{k} \frac{k+1}{k+1-i}  \leq \frac{(k+1)^k}{k\,!},
    \]
for all $j\geq k+1$. Thus, to bound $\expec\big[X^k\big]$, it suffices to bound $\expec\big[(X-1)_k\big]$.

For any \(v\in[n]\), we rewrite the expectation of \(\big(\big|\clusterpi{\sss \geq v}{n}\big|-1\big)_k\) as
    \eqn{\label{for:lem-higher-moment-1}
        \expec\Big[\big(\big|\clusterpi{\sss \geq v}{n}\big|-1\big)_k\Big] = \expec\Big[ \sum\limits_{(u_1,\ldots,u_k)\subseteq(v,n]^k} \prod\limits_{i\in[k]}\indic{u_i\in\clusterpi{\sss\geq v}{n}} \Big],
    }
where the sum over $(u_1,\ldots,u_k)\subseteq(v,n]$ is over {\em distinct} elements in $(v,n]$. 
\smallskip

    Recall the notion of an edge-labeled path below \eqref{for:s2:UB-2}, where the edge-labels correspond to which edge of the later vertex made the connection to the older vertex. We extend this to {\em edge-labeled trees,} where again each edge carries a label in $\{1,2\}$ indicating which edge made the connection. Further, we call a tree {\em self-avoiding} when every vertex occurs in it at most once.
    \smallskip
    
    For any $k$-tuple \(\mathbf{u}_k := (u_1,\ldots,u_k)\in(v,n]^k\) of distinct indices, the event \(\{ u_1,\ldots,u_k \}\in\clusterpi{\sss\geq v}{n}\) ensures that there exists an edge-labeled self-avoiding tree 
    \(\Tree^v_\ell(\mathbf{u}_k) \subseteq G_n^\pi\) spanned by the vertices \(\{v,u_1,\ldots,u_k\}\), and this tree consists solely of vertices in $(v,n]$. We call \(\{u_1,\ldots,u_k\}\) the {\em terminal vertices} of the tree \(\Tree_\ell(\mathbf{u}_k)\).
    Let us denote by \(\cT^v_\ell(\mathbf{u}_k)\) the set of all such edge-labeled self-avoiding trees. 
    By the union bound, we upper bound the expectation of the \(k^{\rm th}\) falling factorial of the percolated connected component size as
    \eqan{\label{for:lem-higher-moment-2}
        \expec\Big[\big(\big| \clusterpi{\sss \geq v}{n} \big|-1\big)_k\Big] \leq \sum\limits_{\mathbf{u}_k\subseteq(v,n]^k} \,\sum\limits_{\Tree^v_{\ell}(\mathbf{u}_k) \in \cT^v_\ell(\mathbf{u}_k)} \prob \big( \Tree^v_{\ell}(\mathbf{u}_k) \subseteq  G_n^\pi\big).
    }
   To prove the lemma, we use induction on \(k\), where the induction hypothesis is that the upper bound 
    \begin{equation}\label{for:lem-higher-moment-3}
        \sum\limits_{\mathbf{u}_k\subseteq(v,n]^k} \,\sum\limits_{\Tree^v_{\ell}(\mathbf{u}_k) \in \cT^v_\ell(\mathbf{u}_k)} \prob \big( \Tree^v_{\ell}(\mathbf{u}_k) \subseteq  G_n^\pi\big)\, \prod_{i=1}^k \big(\frac{n}{u_i}\big)^{\beta}\leq C_\beta(k) \big(\frac{n}{v}\big)^{\beta k},\qquad\text{for all}~v\in[n],
    \end{equation}
    holds, for some finite constant $C_\beta(k)$, for all $k\geq 1$. Since $n\geq u_i>v$ for all $i\in[k]$, this obviously proves the claim.
    \smallskip
    
    To start the induction, we prove the base case for \(k=1\), for which the trees reduce to {\em paths}. Hence, the lhs of \eqref{for:lem-higher-moment-3} is upper bounded by
    $$
    \sum\limits_{u\in[n]}\sum\limits_{{\bf p}(v,u)}\prob\big( {\bf p}(v,u)\subseteq G_n^\pi \big)\big(\frac{n}{u}\big)^\beta.
    $$
    Therefore, \eqref{base-case:1} proves the base case for the induction with $C_\beta(1) = C_\beta$. Next, we move on to prove the induction step. Let us assume that \eqref{for:lem-higher-moment-3} is true for all \(k<r\), for some $r \ge 2$, and aim to prove it for \(k=\rm r\).
    \smallskip

    We start by proving a property of edge-labeled self-avoiding trees that will be convenient to apply the induction hypothesis. Fix such a tree $\Tree_{\ell}(\mathbf{u}_r) \in \cT_\ell(\mathbf{u}_r)$ with terminal vertices $\{u_1,\ldots,u_r\}$. For $i\neq j\in [r]$, let $z_{ij}$ be the last vertex that the paths from $v$ to $u_i$ and $u_j$ have in common.
    We claim that there exist $i\neq j\in [r]$ such that the edges in the two paths from $z_{ij}$ to $u_i$ and $u_j$ are not used by any path from $v$ to $u_l$ for any $l\in [r]\setminus \{i,j\}$. Note that one of these two paths could be empty if $v,u_i,u_j$ lie on a common path with $v$ at one end. This claim can easily be seen by induction, where the statement is trivial when $r=2$ (since then $[r]\setminus \{i,j\}=\varnothing$). To advance the induction, assume the statement is correct for $r-1$, for some $r \ge 3$, and let us prove it for $r$. By the induction hypothesis, there exist $i\neq j\in [r-1]$ such that the paths from $z_{ij}$ to $u_i$ and $u_j$ are not used by any path from $v$ to $u_l$ for any $l\in [r-1]\setminus \{i,j\}$. Consider the path from $u_r$ to $v$. If this path does not contain $z_{ij},$ then $u_i$ and $u_j$ remain on having the desired property. If this path does contain $z_{ij}$, then it must intersect with the path from $z_{ij}$ to $u_i$, or from $z_{ij}$ to $u_j$, or the paths from $z_{ij}$ to $u_i, u_j$ and $u_r$ are all disjoint, in which case $u_i,u_j$ still have the desired property. Thus, by symmetry, assume that it intersects the path from $z_{ij}$ to $u_i$. Then, obviously, $u_i$ and $u_r$ have the desired property.
    
    \smallskip
    
    For any given coordinate-wise distinct $\mathbf{u}_r \in (v,n]^r$ and $z \in (v,n]$, denote by $r_{ij}(\mathbf{u}_r) := (u_l : l \neq i,j) \in (v,n]^{r-2}$, and recall that $\cT^v_\ell(r_{ij}(\mathbf{u}_r), z)$ denotes the set of edge-labeled self-avoiding trees with terminal vertices $\{\{u_l : l \neq i,j\}, z\}$, which consists only of vertices in $(v,n]$. For any $\Tree^v_\ell(\mathbf{u}_r) \in \cT^v_\ell(\mathbf{u}_r)$, let $u_i$, $u_j$, $z=z_{ij}$ satisfy the above property for $i \neq j \in [r]$. Then we will write $\Tree^v_\ell(r_{ij}(\mathbf{u}_r), z) \sim \Tree^v_\ell(\mathbf{u}_r)$ for some $\Tree_\ell^v(r_{ij}(\mathbf{u}_r), z) \in \cT_\ell^v(r_{ij}(\mathbf{u}_r), z)$ to denote that $\Tree^v_\ell(r_{ij}(\mathbf{u}_r), z)$ is obtained by removing the paths from $z=z_{ij}$ to $u_i$ and $u_j$ in $\Tree^v_\ell(\mathbf{u}_r)$.
    
 With this notation, we can write the sum on the lhs of \eqref{for:lem-higher-moment-3} with $k=r$ as
    \begin{multline}\label{mult:genbd}
     \sum\limits_{\mathbf{u}_r\subseteq(v,n]^r} \,\sum\limits_{\Tree^v_{\ell}(\mathbf{u}_r) \in \cT^v_\ell(\mathbf{u}_r)} \prob \big( \Tree^v_{\ell}(\mathbf{u}_r) \subseteq  G_n^\pi\big)\, \prod_{i=1}^r \big(\frac{n}{u_i}\big)^{\beta}\\
     = \sum\limits_{\mathbf{u}_r\subseteq(v,n]^r}\sum_{z\in (v,n]} \ \sum_{i\neq j\in[r]} \ \sum\limits_{\Tree^v_{\ell}(r_{ij}(\mathbf{u}_r), z) \in \cT^v_{\ell}(r_{ij}(\mathbf{u}_r), z)} \prob \big( \Tree^v_{\ell}(r_{ij}(\mathbf{u}_r), z) \subseteq G_n^\pi \big)\prod_{s\in [r]\setminus \{i,j\}} \big(\frac{n}{u_s}\big)^{\beta} \\
        \times\sum\limits_{\stackrel{\Tree^v_\ell(\mathbf{u}_r) \in \cT^v_\ell(\mathbf{u}_r),}{\Tree^v_\ell(r_{ij}(\mathbf{u}_r), z) \sim \Tree^v_\ell(\mathbf{u}_r)}}
        \prob \big( \Tree^v_\ell(\mathbf{u}_r)\subseteq G_n^\pi\mid \Tree^v_\ell(r_{ij}(\mathbf{u}_r), z)\subseteq G_n^\pi \big)\big(\frac{n}{u_i}\big)^{\beta}\big(\frac{n}{u_j}\big)^{\beta},
     \end{multline}
where the outer sums are implicitly assumed to be over distinct indices in $(v,n]^r$. The index $z$ could either be distinct from the elements in $\{u_i,u_j,v\}$, or $z \in \{u_i,u_j,v\}$.
    Let us denote the contribution to the sum on the rhs of \eqref{mult:genbd} when $z \notin \{u_i,u_j,v\}$ by $S_1$, and the remaining contribution where $z\in \{u_i,u_j,v\}$ by $S_2$.

    We begin by bounding
    \begin{multline}\label{eq:s1bd}
        S_1 := \sum\limits_{\mathbf{u}_r\subseteq(v,n]^r} \sum_{i\neq j\in[r]} \ \sum_{\stackrel{z\in (v,n]}{z \notin \{u_i,u_j,v\}}} \ \sum\limits_{\Tree^v_{\ell}(r_{ij}(\mathbf{u}_r), z) \in \cT^v_{\ell}(r_{ij}(\mathbf{u}_r), z)} \prob \big( \Tree^v_{\ell}(r_{ij}(\mathbf{u}_r), z) \subseteq G_n^\pi \big)\prod_{s\in [r]\setminus \{i,j\}} \big(\frac{n}{u_s}\big)^{\beta} \\
        \times\sum\limits_{\stackrel{\Tree^v_\ell(\mathbf{u}_r) \in \cT^v_\ell(\mathbf{u}_r),}{\Tree^v_\ell(r_{ij}(\mathbf{u}_r), z) \sim \Tree^v_\ell(\mathbf{u}_r)}}
        \prob \big( \Tree^v_\ell(\mathbf{u}_r)\subseteq G_n^\pi\mid \Tree^v_\ell(r_{ij}(\mathbf{u}_r), z)\subseteq G_n^\pi \big)\big(\frac{n}{u_i}\big)^{\beta}\big(\frac{n}{u_j}\big)^{\beta}. 
    \end{multline}
    Let $p_1(z, u_i)$ and $p_2(z, u_j)$ denote the paths from $z$ to $u_i$ and $u_j$.
    Conditionally on \(\{ \Tree^v_\ell(r_{ij}(\mathbf{u}_r), z)\subseteq G_n^\pi \}\), the probability of observing \(\Tree^v_\ell(\mathbf{u}_r)\) in \(G_n^\pi\) is at most the probability of observing the paths $p_1(z, u_i)$ and $p_2(z, u_j)$ in \(G_n^{\pi}\), which avoids edges in \(\Tree^v_\ell(r_{ij}(\mathbf{u}_r),z)\), under the same conditioning event. By independence, which is due to the self-avoiding constraint of all the vertices in the tree $\Tree^v_\ell(\mathbf{u}_r)$,
    \eqan{\label{for:induction-hmb:2}
        &\sum\limits_{u_i, u_j\in(v,n]}\sum\limits_{\stackrel{\Tree^v_\ell(\mathbf{u}_r) \in \cT^v_\ell(\mathbf{u}_r),}{\Tree^v_\ell(r_{ij}(\mathbf{u}_r), z) \sim \Tree^v_\ell(\mathbf{u}_r)}}
        \prob \big( \Tree^v_\ell(\mathbf{u}_r)\subseteq G_n^\pi\mid \Tree^v_\ell(r_{ij}(\mathbf{u}_r), z)\subseteq G_n^\pi \big)\big(\frac{n}{u_i}\big)^{\beta}\big(\frac{n}{u_j}\big)^{\beta}\\
        &\hspace{1cm}\le \sum\limits_{u_i, u_j\in(v,n]} \sum\limits_{p_1(z, u_i)}\prob \big( p_1(z, u_i)\subseteq G_n^{\pi}\big)\big(\frac{n}{u_i}\big)^{\beta}\sum\limits_{p_2(z, u_j)} \prob \big( p_2(z, u_j)\subseteq G_n^{\pi}\big)\big(\frac{n}{u_j}\big)^{\beta}\nn.
    }
    Indeed, conditional on \(\{ \Tree^v_\ell(r_{ij}(\mathbf{u}_r), z)\subseteq G_n^\pi \}\), one may simply independently `resample' the edges used in \(\Tree^v_\ell(r_{ij}(\mathbf{u}_r),z)\) and retain the edges not used in this tree, to obtain a graph independent of the conditioning event. This graph has the same law as $G_n^\pi$ and the existence of the path $p_1(z, u_i)$ in $G_n^\pi$ which avoids edges in \(\Tree^v_\ell(r_{ij}(\mathbf{u}_r),z)\) implies the existence of the same path in the new graph. This gives the above upper bound.
    
    By \eqref{base-case:1}, setting $C_\beta(1)= C_\beta$, for any  starting point \(z\) in \((v,n]\), the rhs of \eqref{for:induction-hmb:2} can be upper bounded by
    \eqan{\label{for:induction-hmb:5}
     C_{\beta}(1)^2\big( \frac{n}{z} \big)^{2\beta}\leq C_{\beta}(1)^2\big( \frac{n}{v} \big)^{\beta}\big( \frac{n}{z} \big)^{\beta},
    }
    where we have used that $z\geq v$, and we note that the factor $(n/z)^{\beta}$ produces the missing vertex factor corresponding to $z$ for the tree $\Tree^v_\ell(r_{ij}(\mathbf{u}_r),z)$ in the first line of \eqref{eq:s1bd}.
    \smallskip

    Substituting the upper bounds thus obtained in \eqref{for:induction-hmb:2} and \eqref{for:induction-hmb:5} into \eqref{eq:s1bd}, we arrive at
    \eqan{
        S_1&\leq C_\beta(1)^2 \big( \frac{n}{v} \big)^\beta \sum\limits_{\widetilde{\mathbf{u}}_{r-1}\subseteq(v,n]^{r-1}} \sum_{i\neq j\in[r]} \ \sum\limits_{\Tree^v_{\ell}(\widetilde{\mathbf{u}}_{r-1}) \in \cT^v_{\ell}(\widetilde{\mathbf{u}}_{r-1})} \prob \big( \Tree^v_{\ell}(\widetilde{\mathbf{u}}_{r-1}) \subseteq G_n^\pi \big)\prod_{s=1}^{r-1} \big(\frac{n}{\widetilde{u}_s}\big)^{\beta},
        \nn
    }
    where we combined the role of $(r_{ij}(\mathbf{u}_r), z)$ into $\widetilde{\mathbf{u}}_{r-1} \subseteq (v,n]^{r-1}$.
    Hence, by the induction hypothesis, the above yields
    \eqan{\label{eq:s1final}
         S_1\leq C_\beta(1)^2 r(r-1) C_\beta(r-1)\big(\frac{n}{v}\big)^{{\rm r}\beta},
    }
    where the factor $r(r-1)$ arises from the possible choices of $i\neq j\in[r].$
    \smallskip

    We continue with the contribution $S_2$ where $z\in \{u_i,u_j,v\}$, which is similar and arguably simpler. When $z=v$, the paths from $z$ to $u_i$ and $u_j$ are disjoint from the edge-labeled self-avoiding tree with terminal vertices $\{u_l : l \neq i,j\}$, so that this contribution, which we denote by $S_{21},$ can be bounded by
        \eqan{
        \label{S1-bd}
        S_{21}
        &\leq \sum_{i\neq j\in[r]} \ \sum\limits_{\widetilde{\mathbf{u}}_{r-2}\subseteq(v,n]^{r-2}} \ \sum\limits_{\Tree^v_{\ell}(\widetilde{\mathbf{u}}_{r-2}) \in \cT^v_{\ell}(\widetilde{\mathbf{u}}_{r-2})} \prob \big( \Tree^v_{\ell}(\widetilde{\mathbf{u}}_{r-2}) \subseteq G_n^\pi \big)\prod_{s=1}^{r-2} \big(\frac{n}{\widetilde{u}_s}\big)^{\beta}\\
        &\qquad \times
        \sum\limits_{u_i, u_j\in(v,n]} \sum\limits_{p_1(v, u_i)}\sum\limits_{p_1(v, u_j)}\prob \big( p_1(v, u_i)\subseteq G_n^{\pi}\big)\big(\frac{n}{u_i}\big)^{\beta}\sum\limits_{p_2(v, u_j)} \prob \big( p_2(v, u_j)\subseteq G_n^{\pi}\big)\big(\frac{n}{u_j}\big)^{\beta},\nn
        }
 whose contribution can be bounded, by the induction hypothesis, by
    \eqn{
    \label{S21-bd}
    S_{21}\leq C_\beta(1)^2 r(r-1) C_\beta(r-2)\big(\frac{n}{v}\big)^{{\rm r}\beta}.
    }
    
For the contribution to $S_2$ where $z\in \{u_i,u_j\},$ which we denote by $S_{22},$ by symmetry and without loss of generality, we may assume that $z=u_i$. Then, the path from $u_i$ to $u_j$, denote by $p(u_i, u_j)$, is disjoint from the self-avoiding edge-labeled tree $\Tree_{\ell}(r_j(\mathbf{u}_r))$ with terminal vertices comprised of the indices in $r_j(\mathbf{u}_r) := (u_l)_{l\in [r]\setminus \{j\}}$. Then
        \eqan{
        \label{S22-bd}
        S_{22}
        &\leq 2\sum\limits_{\widetilde{\mathbf{u}}_{r-1}\subseteq(v,n]^{r-1}} \ \sum_{ i \neq j \in[r]} \ \sum\limits_{\Tree^v_{\ell}(\widetilde{\mathbf{u}}_{r-1}) \in \cT^v_{\ell}(\widetilde{\mathbf{u}}_{r-1})} \prob \big( \Tree^v_{\ell}(\widetilde{\mathbf{u}}_{r-1}) \subseteq G_n^\pi \big)\prod_{s=1}^{r-1} \big(\frac{n}{\widetilde{u}_s}\big)^{\beta}\\
        &\qquad \times
        \sum\limits_{u_j\in(v,n]} \sum\limits_{p(u_i, u_j)}\prob \big( p(u_i, u_j)\subseteq G_n^{\pi}\big)\big(\frac{n}{u_j}\big)^{\beta}.\nn
        }
By \eqref{base-case:1}, for fixed $u_i$, the final sum is bounded by
$C_{\beta}(1)(n/u_i)^{\beta}\leq C_{\beta}(1)(n/v)^{\beta},$
while, by the induction hypothesis, the first line is bounded by $r(r-1) C_\beta(r-1)(n/v)^{{\rm (r-1)}\beta}$.
\smallskip

In total, by \eqref{eq:s1final}, \eqref{S1-bd}, \eqref{S21-bd} and \eqref{S22-bd}, the lhs of \eqref{mult:genbd} is bounded by
    \eqan{
    S_1+S_2&=S_1+S_{21}+S_{22}\\
    &\leq C_\beta(1)^2 r(r-1) C_\beta(r-1)\big(\frac{n}{v}\big)^{{\rm r}\beta}+C_\beta(1)^2 r(r-1) C_\beta(r-2)\big(\frac{n}{v}\big)^{{\rm r}\beta}\nn\\
    &\qquad +2C_{\beta}(1)r(r-1)C_\beta(r-1)\big(\frac{n}{v} \big)^{{\rm r}\beta},\nn\\
    &= C_\beta(r)\big(\frac{n}{v} \big)^{{\rm r}\beta},\nn
    }
when $C_\beta(r)=C_\beta(1)r(r-1) [C_\beta(1)C_\beta(r-1)+C_\beta(1)C_\beta(r-2)
+2C_\beta(r-1)].$
\smallskip

This completes the proof of the induction hypothesis, and thus of Lemma~\ref{lem-higher-moment-bound}.
\end{proof}


\section{Convergence of the second susceptibility and stochastic approximation}
\label{sec-conv-second-susc-stoch-approx}
In this section, we analyze the asymptotics of the second susceptibility and its truncated versions.  This culminates in the proof of Theorem \ref{thm-suscep-comp}. We also obtain a rate of convergence in Proposition \ref{prop-conrates2}, which will be a crucial tool in the subsequent sections. In Section \ref{sec-size-truncated-susceptibilities}, we analyze size-truncated susceptibilities. In Section \ref{sec-convergence-probability-susceptibitlity}, we start by proving convergence in probability of the susceptibility. We strengthen this to convergence almost surely in Section \ref{sec-convergence-as-susceptibitlity}, which is the main technical achievement of this section.

\subsection{Analysis of size-truncated second susceptibility}
\label{sec-size-truncated-susceptibilities}

For $\ell \geq 1$, let  $X_\ell^\pi(n)$ denote the number of components of size $\ell$ in $G^{\pi}_n$. Let $x_{\ell}^\pi(n) = X_{\ell}^\pi(n)/n$ denote the proportion of size $\ell$ components. For fixed $L\geq 1$, define the $L^{\rm th}$ truncated second susceptibility via 
\begin{equation}
    \label{eqn:truncated-sus}
    \susceptibilitypi{n}{\sss 2,L} = \frac{1}{n} \sum_{v=1}^{n}|\clusterpi{v}{n}|\ind\set{|\clusterpi{v}{n}|\leq L} = \frac{1}{n} \sum_{\ell = 1}^{L} \ell^2 x_{\ell}^\pi(n). 
\end{equation}
The following proposition describes the almost-sure convergence of these quantities, where we recall the size of the local limit of the connected component of a uniform vertex $|\clusterpi{\sss\varnothing}{\infty}|$ described in Proposition \ref{prop:expec-s2-convg}:

\begin{prop}[Convergence of size-truncated susceptibility]
\label{prop:trunc-suscep-as}
    For every fixed $L$, 
    \begin{equation}
    \label{eqn:6407}
    \susceptibilitypi{n}{\sss 2,L} \convas \susceptibilitypi{\infty}{\sss 2,L} := \E\left[|\clusterpi{\sss\varnothing}{\infty}|\ind\set{|\clusterpi{\sss\varnothing}{\infty}|\leq L}\right].\end{equation}  
    Further, with $\susceptibilitypi{\infty}{2}$ as in \eqref{eqn:s2-def},  
    \begin{equation}
    \label{eqn:sum-assert}
    \lim_{L \to \infty}  \E\left[|\clusterpi{\sss\varnothing}{\infty}|\ind\set{|\clusterpi{\sss\varnothing}{\infty}|\leq L}\right] = \susceptibilitypi{\infty}{2}.
    \end{equation}
\end{prop}

\begin{proof}
We start by introducing some necessary notation. Let $\dU = (\fU_i)_{i\geq 2}$ be an independent collection of random variables. Here, each variable $\fU_i$ is again made up of three independent components $\fU_i = (\fB_i, \fV_i^{\sss(1)}, \fV_i^{\sss(2)})$ with $\fB_i\sim \Bin(2, \pi)$ and, for $j=1,2$, $\fV_i^{\sss(j)}$ are discrete uniform random variables on the index set $[i-1]:=\set{1,2,\ldots, i-1}$. 

For later use, let $\dU^\prime = (\fU_i^\prime)_{i\geq 2}$ be an independent copy of $\dU$.
 \smallskip
 
Fix $\ell \geq 1$ and $n\geq 2$, and note that the number of components of size $\ell$ at time $n$ can be constructed as a deterministic function of the first $n-1$ random variables of $\dU$, namely $X_{\ell}^\pi(n) = g_{n;\ell}(\fU_2, \ldots, \fU_n)$ for an appropriate choice of $g_{n;\ell}$. In addition, note that a perturbation of any single coordinate $\fU_i$ of this function can be viewed as adding or deleting at most two edges used in the formation of the graph up to time $n$. This can at most create or destroy at most two components of size $\ell$, and, hence, 
 \[\big|g_{n;\ell}(\fU_2, \ldots, \fU_{i-1},\fU_{i},\fU_{i+1}, \ldots, \fU_{n} ) -g_{n;\ell}(\fU_2, \ldots, \fU_{i-1},\fU_{i}^\prime,\fU_{i+1}, \ldots, \fU_{n} ) \big|\leq 2. \]
Thus, by the Azuma-Hoeffding inequality, for any fixed $\ell \leq L$, the density of components and the truncated susceptibility satisfy
\begin{equation}
    \label{eqn:442}
    \bigg|\frac{X_{\ell}^\pi(n)}{n} - \E\Big[\frac{X_{\ell}^\pi(n)}{n}\Big]\bigg|\convas 0, \qquad \text{so that} \qquad  |\susceptibilitypi{n}{\sss 2,L}- \E[\susceptibilitypi{n}{\sss 2,L}]|\convas 0. 
\end{equation}
 Local convergence of the component of a randomly selected vertex in Proposition \ref{prop:expec-s2-convg} implies that 
\begin{equation}
    \label{eqn:6357}
    \E[\susceptibilitypi{n}{\sss 2,L})] \to \E\left[|\clusterpi{\sss\varnothing}{\infty}|\ind\set{|\clusterpi{\sss\varnothing}{\infty}|\leq L}\right], \qquad \text{as } n\to\infty.
\end{equation}
Combining \eqref{eqn:442} and \eqref{eqn:6357} proves \eqref{eqn:6407}. To complete the proof of Proposition \ref{prop:trunc-suscep-as},  using the monotone convergence theorem we have $\E\left[|\clusterpi{\sss\varnothing}{\infty}|\ind\set{|\clusterpi{\sss\varnothing}{\infty}|\leq L}\right] \nearrow \E\left[|\clusterpi{\sss\varnothing}{\infty}|\right]$
as $L\nearrow \infty$. Proposition \ref{prop:expec-s2-convg}(b) implies that $\E\left[|\clusterpi{\sss\varnothing}{\infty}|\right] = \susceptibilitypi{\infty}{2} $.
\end{proof}

\subsection{Convergence in probability of the second susceptibility}
\label{sec-convergence-probability-susceptibitlity}
In this section, we start by showing convergence in probability of the second susceptibility, which we strengthen to almost sure convergence in the next section:
\begin{prop}[Convergence in probability of the second susceptibility]
\label{prop:s2-prob-convg}
    The second susceptibility satisfies  $\susceptibilitypi{n}{2} \convp \susceptibilitypi{\infty}{2}$, where $\susceptibilitypi{\infty}{2}$ is defined as in \eqref{eqn:s2-def}. 
\end{prop}
\begin{proof}
    By Proposition~\ref{prop:trunc-suscep-as}, for all \(L\in\mathbb{N}\), 
    \begin{equation}
        \label{eqn:828}
        \susceptibilitypi{n}{ 2} \geq \susceptibilitypi{n}{\sss 2,L} \convas \susceptibilitypi{\infty}{\sss 2,L}\uparrow \expec\big[|\clusterpi{\sss\varnothing}{\infty}|\big]=\susceptibilitypi{\infty}{2},
    \end{equation}
    where the last equality follows from Proposition~\ref{prop:trunc-suscep-as}. Moreover, as $\susceptibilitypi{\cdot}{\sss 2,L}$ is bounded by $L^2$, the convergence of $\susceptibilitypi{n}{\sss 2,L}$ also takes place in $\bL^1$. Thus, for any $\vep> 0$, 
    \begin{align*}
        \prob \big( \big| \susceptibilitypi{n}{2} - \susceptibilitypi{\infty}{2} \big| \geq \vep \big) &\leq \prob \big( \big| \susceptibilitypi{n}{2} - \susceptibilitypi{n}{\sss 2,L}\big| \geq \vep/2 \big) + \prob\big(\big|\susceptibilitypi{n}{\sss 2,L} - \susceptibilitypi{\infty}{2} \big| > \vep/2 \big)\\
        &\leq \frac{2}{\vep}\expec \big[ \susceptibilitypi{n}{2} - \susceptibilitypi{n}{\sss 2,L} \big] + \frac{2}{\vep}\expec \big[ \big| \susceptibilitypi{n}{\sss 2,L} - \susceptibilitypi{\infty}{\sss 2,L} \big|\big] + \frac{2}{\vep}\big|\susceptibilitypi{\infty}{\sss 2,L} - \susceptibilitypi{\infty}{2}\big|. 
    \end{align*}
    By \eqref{eqn:828}, the middle term vanishes as $n\rightarrow \infty$. By Proposition~\ref{prop:expec-s2-convg}(c), $\lim_{n \to \infty}\expec\big[\susceptibilitypi{n}{2}\big] =  \susceptibilitypi{\infty}{2}$. Using this along with \eqref{eqn:828} gives
    \begin{align*}
        \lim_{n \to \infty}\expec \big[ \susceptibilitypi{n}{2} - \susceptibilitypi{n}{\sss 2,L} \big] = \susceptibilitypi{\infty}{2} - \susceptibilitypi{\infty}{\sss 2,L}.
    \end{align*}
    Thus, for any $L$, 
    \[\limsup_{n\to \infty} \prob \big( \big|\susceptibilitypi{n}{2} - \susceptibilitypi{\infty}{2}\big| \geq \vep \big) \leq \frac{4}{\vep} \big[ \susceptibilitypi{\infty}{2} - \susceptibilitypi{\infty}{\sss 2,L} \big]. \]
    Now letting $L\uparrow \infty$ and again using \eqref{eqn:828} completes the proof. 
\end{proof}

\subsection{Almost sure convergence of the second susceptibility}
\label{sec-convergence-as-susceptibitlity}
    We continue with the proof of Theorem \ref{thm-suscep-comp}. We prove this result by using Proposition \ref{prop:s2-prob-convg} in conjunction with a stochastic approximation argument applied to the evolution of $\susceptibilitypi{\cdot}{2}$, combined with the almost sure weak upper bound on the maximal component size in Proposition \ref{prop:max-as-up-bound}. Let us start with the stochastic approximation scheme. 
    
    As before, we write  $(G^{\pi}_n)_{n\geq 1}$ for the natural filtration of the process. The next lemma makes the stochastic approximation explicit:

    \begin{lem}[Stochastic approximation for the second susceptibility]
    \label{lem:s2-evol}
        The evolution of the process $(\susceptibilitypi{n}{2})_{n\geq 1}$ can be written as 
        \begin{equation}\label{SAscheme}
            \susceptibilitypi{n+1}{2} - \susceptibilitypi{n}{2} = \frac{1}{n+1}\left[F(\susceptibilitypi{n}{2}) + R_n + \xi_{n+1} \right], \qquad n\geq 1,
        \end{equation}
        with the following description of the terms on the right-hand side:
        \begin{enumeratea}
            \item $F(s) = 2\pi^2 s^2+ (4\pi-1)s+1$, and the error term $R_n$ satisfies $|R_n| \leq K \susceptibilitypi{n}{4}/n, \, n \ge 1,$ for a positive constant $K$.
            \item The martingale differences  $(\xi_{n+1})_{n\geq 1}$ satisfy $\E[\xi_{n+1}^2\,|\,G^{\pi}_n] \leq K (\susceptibilitypi{n}{3})^2$. 
        \end{enumeratea}
        Here, writing $\Delta \susceptibilitypi{n}{2} = \susceptibilitypi{n+1}{2} - \susceptibilitypi{n}{2}$, the above objects have the interpretation 
        \[  
        (n+1)\E[\Delta \susceptibilitypi{n}{2} \,| \,G^{\pi}_n]= F(s_2(n)) + R_n, \qquad \xi_{n+1} = (n+1)\big(\Delta \susceptibilitypi{n}{2} - \E[\Delta \susceptibilitypi{n}{2}\,|\,G^{\pi}_n]\big).
        \]
    \end{lem}

    \begin{proof}
    Observe that
    \begin{align*}
        &\E\left[\Delta \susceptibilitypi{n}{2} \,| \,G^{\pi}_n \right]\\
        &\quad = \E\left[\susceptibilitypi{n+1}{2} - \frac{n}{n+1}\susceptibilitypi{n}{2} \,\Big|\,G^{\pi}_n \right] - \frac{\susceptibilitypi{n}{2}}{n+1}\\
        &\quad= \frac{1}{n+1}\sum_{i < j} \left((|\clusterpi{\sss \ge i}{n}| + |\clusterpi{\sss \ge j}{n}| + 1)^2 - |\clusterpi{\sss \ge i}{n}|^2 - |\clusterpi{\sss \ge j}{n}|^2\right) \frac{2\pi^2|\clusterpi{\sss \ge i}{n}||\clusterpi{\sss \ge j}{n}|}{n^2}\\
        &\qquad + \frac{1}{n+1}\sum_i\left((|\clusterpi{\sss \ge i}{n}|+1)^2 - |\clusterpi{\sss \ge i}{n}|^2\right)\left(\frac{\pi^2|\clusterpi{\sss \ge i}{n}|^2}{n^2} + \frac{2\pi(1-\pi)|\clusterpi{\sss \ge i}{n}|}{n}\right)\\
        &\qquad + \frac{(1-\pi)^2}{n+1} - \frac{\susceptibilitypi{n}{2}}{n+1}.
    \end{align*}
    Simplifying the right-hand side, we obtain
    \begin{align*}
        \E\left[\susceptibilitypi{n+1}{2} - \susceptibilitypi{n}{2} \mid G^{\pi}_n \right] = \frac{1}{n+1}\left[F(\susceptibilitypi{n}{2}) + R_n\right],
    \end{align*}
    where $F$ is as defined in the lemma, and
    \begin{align*}
        R_n = -\frac{2\pi^2}{n}\susceptibilitypi{n}{4} - \frac{2\pi^2}{n}\susceptibilitypi{n}{3}.
    \end{align*}
    Writing $\xi_{n+1} = (n+1)\big(\Delta \susceptibilitypi{n}{2} - \E[\Delta \susceptibilitypi{n}{2}\mid G^{\pi}_n]\big)$, and using that ${\mathrm Var}(X)\leq \expec[(X-a)^2]$ for every random variable $X$ and $a\in \mathbb{R}$, we bound
        \begin{align*}
        &\E[\xi_{n+1}^2\mid G^{\pi}_n]
        \le 2(n+1)^2 \E\left[\left(\susceptibilitypi{n+1}{2} - \frac{n}{n+1}\susceptibilitypi{n}{2}\right)^2 \,\Big|\, G^{\pi}_n \right]\\
        &\quad = \sum_{i < j} \left((|\clusterpi{\sss \ge i}{n}| + |\clusterpi{\sss \ge j}{n}| + 1)^2 - |\clusterpi{\sss \ge i}{n}|^2 - |\clusterpi{\sss \ge j}{n}|^2\right)^2 \frac{4\pi^2|\clusterpi{\sss \ge i}{n}||\clusterpi{\sss \ge j}{n}|}{n^2}\\
        &\qquad + \sum_i\left((|\clusterpi{\sss \ge i}{n}|+1)^2 - |\clusterpi{\sss \ge i}{n}|^2\right)^2\left(\frac{2\pi^2|\clusterpi{\sss \ge i}{n}|^2}{n^2} + \frac{4\pi(1-\pi)|\clusterpi{\sss \ge i}{n}|}{n}\right)+ 2(1-\pi)^2\\
        & \le \sum_{i < j} \left(16|\clusterpi{\sss \ge i}{n}|^2|\clusterpi{\sss \ge j}{n}|^2 + 16|\clusterpi{\sss \ge i}{n}|^2 + 16|\clusterpi{\sss \ge j}{n}|^2 + 4\right) \frac{4\pi^2|\clusterpi{\sss \ge i}{n}||\clusterpi{\sss \ge j}{n}|}{n^2}\\
        &\qquad + \sum_i\left(8|\clusterpi{\sss \ge i}{n}|^2 +2\right)\left(\frac{2\pi^2|\clusterpi{\sss \ge i}{n}|^2}{n^2} + \frac{4\pi(1-\pi)|\clusterpi{\sss \ge i}{n}|}{n}\right)+ 2(1-\pi)^2\\
        &\le C_1 (\susceptibilitypi{n}{3})^2 + C_2\frac{\susceptibilitypi{n}{4}}{n} + 2(1-\pi)^2,
        \end{align*}
        for universal positive constants $C_1, C_2$, where we have used $\susceptibilitypi{n}{k} \ge 1$ for all $n \ge 1, k\ge 1$. The bound on $\E[\xi_{n+1}^2\mid G^{\pi}_n]$ follows from the above, upon noting that $\susceptibilitypi{n}{4}/n \le \susceptibilitypi{n}{3} \le (\susceptibilitypi{n}{3})^2$ for all $n \ge 1$.
    \end{proof}

By the form of the function $F$, $F(\susceptibilitypi{\infty}{2})=0$ and $\susceptibilitypi{\infty}{2}$ is the only {\em stable} fixed point of the associated discrete dynamical system (see Figure \ref{fig:F}). This suggests that $\susceptibilitypi{n}{2}$ should indeed converge almost surely to $\susceptibilitypi{\infty}{2}$. Leveraging this intuition is one of the highlights of modern stochastic approximation theory, see e.g. \cite{Pema07}.  Proving this in our context however, requires careful analysis owing to the unbounded nature of the jumps and highlights the importance of the apriori a.s. upper bounds on the size of the maximal component (Prop. \ref{prop:max-as-up-bound}).
\begin{figure}
    \centering
    \includegraphics[width=0.7\linewidth]{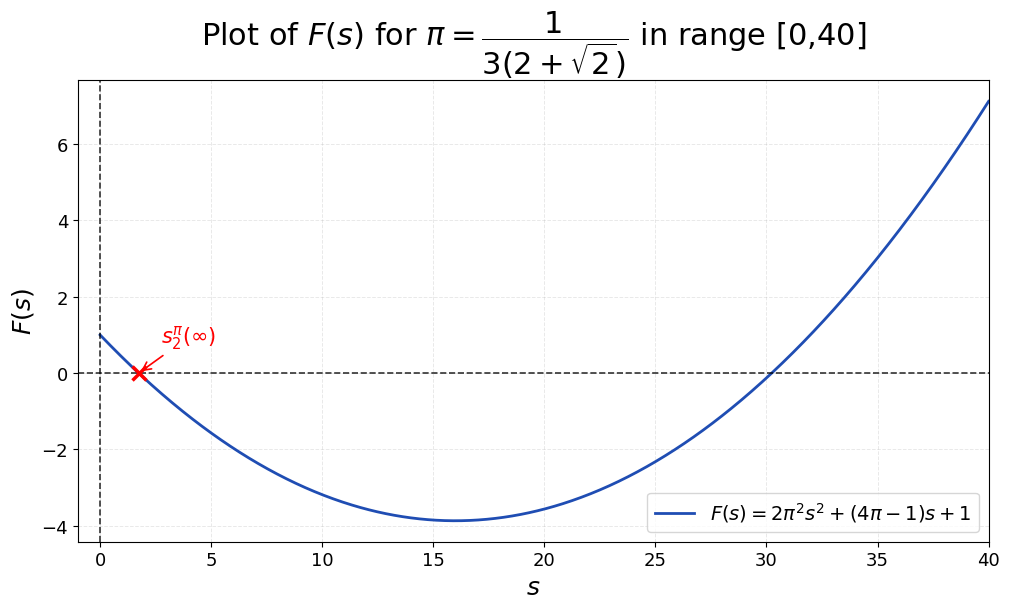}
    \caption{Example of the function $F$ for a specific choice of $\pi< \pi_c$. }
    \label{fig:F}
\end{figure}

\medskip

The following proposition will complete the proof of Theorem \ref{thm-suscep-comp}. We state this result more generally, so that we can apply the technique also to related models, such as the preferential attachment model \cite{BanBhaHazHofRay26}:

\begin{prop}[Convergence of solutions to stochastic approximation]
\label{prop-as-up-down}
Let $(\susceptibility{n}{2})_{n\geq 1}$ be a process, whose evolution can be written as 
        \begin{equation}
        \label{SAscheme-rep}
            \susceptibility{n+1}{2} - \susceptibility{n}{2} = \frac{1}{n+1}\left[F(\susceptibility{n}{2}) + R_n + \xi_{n+1} \right], \qquad n\geq 1,
        \end{equation}
        where the function $F(s)$ equals $F(s) = b(s-\lambda_1)(s-\lambda_2)$ with $b>0, 0<\lambda_1<\lambda_2$, and $|R_n| \leq Kn^{-\vep}\susceptibility{n}{2}^p$ and  $\E[\xi_{n+1}^2\,|\,G^{\pi}_n] \leq K n^{1-\vep}\susceptibility{n}{2}^p$ eventually almost surely for some $K,\vep,p>0$. Suppose further that $\susceptibility{n}{2}\convp \lambda_1$. Then, in fact, $\susceptibility{n}{2}\convas \lambda_1$.
\end{prop}

To see why Proposition \ref{prop-as-up-down} is useful, we first  show how it can be used to prove  Theorem \ref{thm-suscep-comp}:
\begin{proof}[Proof of Theorem \ref{thm-suscep-comp} subject to Proposition \ref{prop-as-up-down}]
We apply Proposition \ref{prop-as-up-down} to the uniform attachment model with $m=2$, for which we take $\susceptibility{n}{2}=\susceptibilitypi{n}{2}$. We check all the conditions:
\begin{itemize}
\item[(a)] Recall the quadratic function $F$ in Lemma \ref{lem:s2-evol} and write this as $F(s) = b(s-\lambda_1)(s-\lambda_2)$, where $\susceptibilitypi{\infty}{2} = \lambda_1< \lambda_2 $ are the two roots and $b=2\pi^2>0$. Thus, the function $F$ satisfies the assumptions in Proposition \ref{prop-as-up-down}. 
\item[(b)] By Lemma \ref{lem:s2-evol}, the second susceptibility of the uniform attachment model satisfies 
\eqref{SAscheme-rep} with $|R_n| \leq K \susceptibilitypi{n}{4}/n$  and $\E[\xi_{n+1}^2\,|\,G^{\pi}_n] \leq K (\susceptibilitypi{n}{3})^2$ for a positive constant $K$. Since, by Proposition \ref{prop:max-as-up-bound}, $\limsup\limits_{n \rightarrow \infty}|\clusterpi{\max}{n}|/n^{\alpha(\pi) + \vep} \leq 1$ eventually almost surely for every $\vep>0$, we also get that
    \eqn{
    |R_n| \leq K \susceptibilitypi{n}{4}/n\leq \frac{K}{n}|\clusterpi{\max}{n}|^2 \susceptibilitypi{n}{2}\leq  n^{2\alpha(\pi)+2\vep-1}\susceptibilitypi{n}{2},
    }
which, since $\alpha(\pi)<\tfrac{1}{2},$ can be made smaller than $K n^{-\vep}\susceptibilitypi{n}{2}$ when taking $\vep>0$ sufficiently small, so that $p=1$. Similarly,
    \eqn{
    \E[\xi_{n+1}^2\,|\,G^{\pi}_n] \leq K (\susceptibilitypi{n}{3})^2\leq |\clusterpi{\max}{n}|^2 \susceptibilitypi{n}{2}^2\leq K n^{1-\vep}\susceptibilitypi{n}{2}^2,
    }
so that $p=2$. Thus, since also $\susceptibilitypi{n}{2}\geq 1$, the required bounds hold for some $K,\vep>0$ and $p=2$.
\item[(c)] Finally, $\susceptibilitypi{n}{2}\convp \lambda_1$ by Proposition \ref{prop:s2-prob-convg}.
\end{itemize}
We conclude that all the conditions in Proposition \ref{prop-as-up-down} apply to $\susceptibilitypi{n}{2}$ in the uniform attachment model with $m=2$, so that Theorem \ref{thm-suscep-comp} follows from Proposition \ref{prop-as-up-down}.
\end{proof}
\medskip

In the remainder of this section, we work with $\susceptibility{n}{2}$ under the conditions in Proposition \ref{prop-as-up-down}, which we prove next:

\begin{proof}[Proof of Proposition \ref{prop-as-up-down}]
    In the following, $C, C_1, C_2, C_3$ will denote positive constants depending only on $\lambda_1,\lambda_2$. We will show that, for any $\eta> 0$,
    \begin{enumeratea}
        \item  $\susceptibility{n}{2} \leq \lambda_1+\eta$ eventually almost surely;
        \item $\susceptibility{n}{2}\geq \lambda_1 -\eta $ eventually almost surely.
    \end{enumeratea}
Together, these show that $\susceptibility{n}{2}\convas \lambda_1$.
\medskip

We start by showing that $\susceptibility{n}{2} \leq \lambda_1+\eta$ eventually almost surely.
We fix any $\eta \in (0, \min((\lambda_1 - \lambda_2),1)$ and $\vep$ as in the statement of Proposition \ref{prop-as-up-down}. For $M \in \mathbb{N}$ and $K'=K(\lambda_1+\eta)^p$, define the stopping times
\begin{align}
N_M &:=\inf\set{n\geq M\colon |\susceptibility{n}{2} - \lambda_1| \leq \eta/4 },\label{stopping-time-bound-1}\\
T_M &:= \inf\set{n > N_M\colon |R_n| > K'n^{-\vep}, \E[\xi_{n+1}^2\,|\,G^{\pi}_n] > K' n^{1-\vep}, \text{ or } \susceptibility{n}{2} \geq \lambda_1+\eta }.
\label{stopping-time-bound-2}
 \end{align}
Note that, as $s_2(n) \convp \lambda_1$, it has an almost surely convergent subsequence and thus $N_M < \infty$ almost surely for any $M \in \mathbb{N}$.

For $M \in \mathbb{N}$, define
\begin{equation}\label{errdef}
    \mathcal{E}^{\sss(M)}_n := \sum_{j=N_M}^{n-1} \frac{1}{j+1}\left(R_j + \xi_{j+1}\right), \quad n > N_M.
\end{equation}
Define the process $\mathcal{M}_{N_M}=0$ and 
$$
\mathcal{M}_n := \sum_{j=N_M+1}^{T_M \wedge n}\frac{\xi_j}{j}, \quad n \ge N_M+1.
$$ 
This is a (stopped) martingale with respect to the filtration generated by $(G^\pi_n)_{n \ge N_M}$ with quadratic variation given by
$$
\langle \mathcal{M}\rangle_n = \sum_{j=N_M+1}^{T_M \wedge n}\frac{\mathbb{E}\big[\xi_{j}^2 \mid G^\pi_{j-1}\big]}{j^2}, \quad n \ge N_M+1.
$$
Thus, by Doob's $L^2$-maximal inequality (applied conditionally given $G^\pi_{N_M}$),  for any $\delta>0$,
\begin{align}\label{stopping-time-bound-3}
    &\prob\left(\sup_{N_M < n \le T_M} \left|\sum_{j=N_M}^{n-1}\frac{\xi_{j+1}}{j+1}\right| > \delta ~\Big|~ G^\pi_{N_M}\right) \le C \delta^{-2} \mathbb{E}\left[\mathcal{M}_{T_M}^2\right] = C \delta^{-2} \mathbb{E}\left[\langle \mathcal{M}\rangle_{T_M}~\Big|~ G^\pi_{N_M}\right]\notag\\
    &\quad = \mathbb{E}\left[\sum_{j=N_M+1}^{T_M}\frac{\mathbb{E}\big[\xi_{j}^2 \mid G^\pi_{j-1}\big]}{j^2}~\Big|~ G^\pi_{N_M}\right] \le K'\mathbb{E}\left[\sum_{j=N_M+1}^{T_M} \frac{1}{j^{1 + \vep}}~\Big|~ G^\pi_{N_M}\right] \le K' \sum_{j=M+1}^\infty \frac{1}{j^{1 + \vep}},
\end{align}
where we have used the definition of $T_M$ in the second inequality, and the fact that $N_M\geq M$ in the last inequality.

From \eqref{stopping-time-bound-1}, \eqref{stopping-time-bound-2} and \eqref{stopping-time-bound-3}, we obtain that there exists an $M_0=M_0(\eta) \in \mathbb{N}$ such that, for all $M \ge M_0$,
\begin{align}\label{doobl2}
    \prob\left(\sup_{N_M < n \le T_M} \left|\mathcal{E}^{\sss(M)}_n\right| > \eta/8~\Big|~ G^\pi_{N_M}\right) &\le \prob\left(C_1 \sum_{j=M}^{\infty}\frac{1}{j^{1+ \vep}} + \sup_{N_M < n \le T_M} \left|\sum_{j=N_M}^{n}\frac{\xi_{j+1}}{j+1}\right| > \eta/8~\Big|~ G^\pi_{N_M}\right)\notag\\
    & \le \prob\left(\sup_{N_M < n \le T_M} \left|\sum_{j=N_M}^{n-1}\frac{\xi_{j+1}}{j+1}\right| > \eta/16~\Big|~ G^\pi_{N_M}\right)\notag\\
    &\le \frac{C_2}{\eta^2}\sum_{j=M+1}^{\infty}\frac{1}{j^{1+ \vep}}\le \frac{C_3}{\eta^2M^{\vep}}.
\end{align}

The above bound shows that, in the stochastic approximation scheme in \eqref{SAscheme}, the `fluctuations' due to the error and noise are quantifiably small. Hence, the `drift term' given by $F$ pulls the process $\susceptibility{\cdot}{2}$ towards $\lambda_1$ and keeps $\susceptibility{\cdot}{2}$ in the interval $(\lambda_1 - \eta, \lambda_1 + \eta)$. 
\smallskip

To make this precise, without loss of generality, we assume $\susceptibility{N_M}{2} > \lambda_1$ and define the sequence of stopping times $(\tau_l)_{l\geq 0}$ by $\tau_0 := N_M$ and, for $l \ge 0$,
\begin{align}
    \tau_{2l+1} &:= \inf\{ n > \tau_{2l} \colon \susceptibility{n}{2} \le \lambda_1\} \wedge T_M,
    \label{stopping-times-def}\\
    \tau_{2l+2} &:= \inf\{ n > \tau_{2l+1} \colon \susceptibility{n}{2} > \lambda_1\} \wedge T_M.
\end{align}
Also define the event
$$
\mathscr{E}_M := \left\lbrace \sup_{N_M < n \le T_M} \left|\mathcal{E}^{\sss(M)}_n\right| \le \eta/8 \right\rbrace.
$$
On the event $\mathscr{E}_M$, for any $l \ge 0$, $\susceptibility{\tau_{2l}}{2} \le \lambda_1 + \eta/4$. Moreover, as $F(\susceptibility{n}{2}) \le 0$ for $n \in [\tau_{2l}, \tau_{2l+1}-1]$, we obtain from the recursive decomposition in Lemma \ref{lem:s2-evol} that, uniformly for $n \in [\tau_{2l}, \tau_{2l+1}-1]$,
\begin{equation*}
\susceptibility{n+1}{2} -\lambda_1 \leq \susceptibility{\tau_{2l}}{2} -\lambda_1 + \sum_{j=\tau_{2l}}^n \frac{R_j}{j+1} + \sum_{j=\tau_{2l}}^n \frac{\xi_{j+1}}{j+1} \le \frac{3\eta}{8}.
\end{equation*}
By a similar argument, uniformly for $n \in [\tau_{2l+1}, \tau_{2l+2}-1]$,
\begin{equation*}
    \susceptibility{n+1}{2} -\lambda_1  \ge -\frac{3\eta}{8}.
\end{equation*}
In particular, on the event $\mathscr{E}_M$, $|\susceptibility{n+1}{2} -\lambda_1| \le 3\eta/4$ for all $n \in [N_M, T_M-1]$. Moreover, this implies that, on the event $\mathscr{E}_M$, the only way in which $T_M$ is finite is if $|R_n| > K'n^{-\vep}$ or $\E[\xi_{n+1}^2\,|\,G^{\pi}_n] > K' n^{1-\vep}$ for some finite $n$ where $\susceptibility{n}{2} \leq \lambda_1+ 3\eta/4$; in particular, $|R_n| > Kn^{-\vep}s_2(n)^p$ or $\E[\xi_{n+1}^2\,|\,G^{\pi}_n] > K n^{1-\vep}s_2(n)^p$. Thus,
\begin{align*}
    &\prob\left(|\susceptibility{n+1}{2} -\lambda_1| > \eta \text{ for some } n \ge N_M\right)\\
    &\qquad \le \prob\left(\mathscr{E}_M^c\right)+ \prob\left(|R_n| > Kn^{-\vep}s_2(n)^p, \text{ or }\E[\xi_{n+1}^2\,|\,G^{\pi}_n] > K n^{1-\vep}s_2(n)^p \text{ for some } n \ge M\right).
\end{align*}
The above bound tends to zero as $M \to \infty$ by \eqref{doobl2} and the assumptions in Proposition \ref{prop-as-up-down}, which completes the proof.
\end{proof}

The following proposition gives a rate of convergence of $\susceptibility{n}{2}$ to $\lambda_1$, which will be useful later in obtaining precise asymptotics for the root component size:

\begin{prop}[Polynomial convergence in stochastic approximation]
\label{prop-conrates2}
Under the assumptions in Proposition \ref{prop-as-up-down}, there exists $\gamma >0$ such that, as $n \to \infty$,
    \begin{equation*}
        n^{\gamma}|\susceptibility{n}{2} - \lambda_1| \convas 0.
    \end{equation*}
\end{prop}

\begin{proof}
Recall the representation $F(s) = b(s-\lambda_1)(s-\lambda_2)$, where $F$ is introduced in Proposition \ref{prop-as-up-down}. Also recall the stopping times $N_M, T_M$ with $\eta = (\lambda_2 - \lambda_1)/2$, and $(\tau_l)_{l \ge 0}$ defined in \eqref{stopping-times-def} in the proof of Proposition \ref{prop-as-up-down}, as well as the stochastic approximation scheme, for $n \ge N_M$,
\begin{equation*}
    \susceptibility{n+1}{2} - \lambda_1 = (\susceptibility{n}{2} - \lambda_1)\left(1 + \frac{b(\susceptibility{n}{2} - \lambda_2)}{n+1}\right) + \mathcal{E}^{\sss(M)}_{n+1} - \mathcal{E}^{\sss(M)}_n,
\end{equation*}
where $\mathcal{E}^{\sss(M)}_n$ is defined in \eqref{errdef} with $\mathcal{E}^{\sss(M)}_{N_M}=0$.
We will first prove that, almost surely, there exist a $\delta>0$ and a sequence $n_k \to \infty$ such that, for all $k$ satisfying $N_M \le n_k < n_k^2 \le T_M$,
\begin{equation}\label{subseq}
n_k^{\delta}|\susceptibility{n^2_k}{2} - \lambda_1| \le C,
\end{equation}
for some deterministic constant $C>0$.

For this, for any $n,m \in [\tau_{2l}, \tau_{2l+1}-1]$ with $n>m$, recursively applying the stochastic approximation scheme, we obtain
\begin{align*}
\susceptibility{n+1}{2} - \lambda_1 &\le (\susceptibility{n}{2} - \lambda_1)\left(1 - \frac{b(\lambda_2 - \lambda_1)}{2(n+1)}\right) + \mathcal{E}^{\sss(M)}_{n+1} - \mathcal{E}^{\sss(M)}_n\\
&\le (\susceptibility{m}{2} - \lambda_1)\prod_{j=m+1}^{n+1}\left(1 - \frac{b(\lambda_2 - \lambda_1)}{2j}\right) + \tilde{\mathcal{E}}^{\sss(M)}_{n+1} - \tilde{\mathcal{E}}^{\sss(M)}_{m},
\end{align*}
where, for $\ell \in \mathbb{N}$,
$$
\tilde{\mathcal{E}}^{\sss(M)}_{\ell} := \sum_{j=1}^{\ell}\left(\prod_{k=j+1}^{n+1}\left(1 - \frac{b(\lambda_2 - \lambda_1)}{2k}\right)\right)\left(\mathcal{E}^{\sss(M)}_{j+1} - \mathcal{E}^{\sss(M)}_j\right),
$$
$\prod_{k=n+2}^{n+1} = 1$, and $\tilde{\mathcal{E}}^{\sss(M)}_{0}=0$. Writing $\theta = b(\lambda_2 - \lambda_1)/2$, and applying the standard exponential inequality $1-x\leq \e^{-x}$, we obtain
\begin{align*}
\susceptibility{n+1}{2} - \lambda_1 &\le (\susceptibility{m}{2} - \lambda_1)\exp\left(- \theta\sum_{j=m+1}^{n+1}\frac{1}{j}\right) + \tilde{\mathcal{E}}^{\sss(M)}_{n+1} - \tilde{\mathcal{E}}^{\sss(M)}_{m}\\
&\le \left(\frac{m+1}{n+2}\right)^\theta (s_2(m) - \lambda_1) + \tilde{\mathcal{E}}^{\sss(M)}_{n+1} - \tilde{\mathcal{E}}^{\sss(M)}_{m}.
\end{align*}
By a symmetric argument, we also obtain, for $n, m \in [\tau_{2l+1}, \tau_{2l+2}-1]$ with $n > m$,
\begin{equation*}
    \susceptibility{n+1}{2} - \lambda_1 \ge \left(\frac{m+1}{n+2}\right)^\theta (\susceptibility{m}{2} - \lambda_1) + \tilde{\mathcal{E}}^{\sss(M)}_{n+1} - \tilde{\mathcal{E}}^{\sss(M)}_{m}.
\end{equation*}
Thus, we obtain positive constants $C_1, C_2$ such that, for any $n \in \mathbb{N}$, on the event $\{N_M \le n < n^2 < T_M\}$,
\begin{equation}\label{contraction}
|\susceptibility{n^2}{2} - \lambda_1| \le \frac{C_1}{n^{\theta}}|\susceptibility{n}{2} - \lambda_1| + \left|\tilde{\mathcal{E}}^{\sss(M)}_{n^2} - \tilde{\mathcal{E}}^{\sss(M)}_{n}\right|.
\end{equation}
Exactly as in \eqref{doobl2}, by  Doob's $L^2$-maximal inequality, with $\mathcal{E}^{\sss(M)}_{\cdot}$ replaced by $\tilde{\mathcal{E}}^{\sss(M)}_{\cdot}$, there exists $C_3>0$ such that, on the event $\{N_M \le n\}$,
\begin{equation*}
\prob\left(\sup_{n \le \ell \le T_M}\left|\tilde{\mathcal{E}}^{\sss(M)}_\ell - \tilde{\mathcal{E}}^{\sss(M)}_n\right| > n^{-\vep/2} \, \Big| \, G^\pi_{n}\right) \le \frac{C_3}{n^\vep}.
\end{equation*}
From this, almost surely, we can obtain a sequence $n_k \rightarrow \infty$ such that, for all $k \in \mathbb{N}$ with $N_M \le n_k \le T_M$,
\begin{equation*}
    \left|\tilde{\mathcal{E}}^{\sss(M)}_\ell - \tilde{\mathcal{E}}^{\sss(M)}_{n_k}\right| \le n_k^{-\vep/2} \quad \text{ for all } \ell \in [n_k, T_M].
\end{equation*}
The inequality in \eqref{subseq} follows from the above and \eqref{contraction} with $\delta = \theta \wedge (\vep/2)$.
\smallskip

Next, we extend the bound \eqref{subseq} to all $n$. To do so, for any $\gamma \in (0,1 \wedge (\delta/2) \wedge b(\lambda_2 - \lambda_1))$, we define
$$
h(n) := n^\gamma (\susceptibility{n}{2} - \lambda_1), \quad n \in \mathbb{N}.
$$
Note that, for a deterministic bounded sequence $(\vep_n)_{n\geq 1}$,
\begin{align*}
    h(n+1) - h(n) &= (n+1)^\gamma(\susceptibility{n+1}{2} - \susceptibility{n}{2}) + ((n+1)^\gamma - n^\gamma)(\susceptibility{n}{2} - \lambda_1)\\
    &= \frac{1}{(n+1)^{1-\gamma}}\left(F(\susceptibility{n}{2}) + R_n + \xi_{n+1}\right) + \left(\frac{\gamma}{(n+1)^{1-\gamma}} + \frac{\vep_n}{(n+1)^{2-\gamma}}\right)(\susceptibility{n}{2} - \lambda_1)\\
    & = \frac{1}{(n+1)^{1-\gamma}}\tilde{F}(\susceptibility{n}{2}) + \frac{\tilde{R}_n + \tilde{\xi}_{n+1}}{n+1},
\end{align*}
where
\begin{align*}
    \tilde{F}(s) &:= F(s) + \gamma (s - \lambda_1),\\
    \tilde{R}_n &:= (n+1)^{\gamma}R_n + \frac{\vep_n}{(n+1)^{1-\gamma}}(\susceptibility{n}{2} - \lambda_1),\\
    \tilde{\xi}_{n+1} &:= (n+1)^{\gamma}\xi_{n+1}.
\end{align*}
Observe that $s \mapsto \tilde{F}(s)$ has zeros at $\lambda_1, \tilde{\lambda}_2$ for some $\tilde{\lambda}_2 \in (\lambda_1, \lambda_2)$. Fix $\tilde{\eta} \in (0, \min((\lambda_1 - \tilde{\lambda}_2),1)$ and, similarly as in the proof of Proposition \ref{prop-as-up-down}, for $M \in \mathbb{N}$, $\tilde{\vep}, \gamma >0$ sufficiently small, and $\tilde{K}>0$ sufficiently large, define
 \begin{align*}
 \tilde{N}_M &:=\inf\set{n\geq M\colon  |h(n)| \leq \tilde{\eta}/4 },\\
\tilde{T}_M &:= \inf\set{n > \tilde{N}_M\colon |\tilde{R}_n| > \tilde{K} n^{-\tilde{\vep}} \ \text{ or } \ \E[\tilde{\xi}_{n+1}^2\,|\,G^{\pi}_n] > \tilde{K} n^{1-\tilde{\vep}}}.
 \end{align*}
By \eqref{subseq}, for any $M \in \mathbb{N}$, on the event $\{T_M = \infty\}$, $\tilde{N}_M$ is almost surely finite. 
Now, arguing exactly as in Proposition \ref{prop-as-up-down}, we conclude that $h(n) \convas 0$, and Proposition \ref{prop-conrates2} follows.
\end{proof}

\section{Asymptotics for the component size of the first vertex}
\label{sec:proof-first-comp}
The goal of this section is to prove Theorem \ref{thm-max-comp}(a). We will only prove the case $i=1$ (first component). The proof for general $i \ge 1$ is very similar and is omitted.

We start in Section \ref{as} with the almost sure convergence which follows from a positive supermartingale argument along with Proposition \ref{prop-conrates2}. Proving the positivity of the limit is surprisingly delicate and is undertaken in Section \ref{sec-almost-sure-positive-limit}. In Section \ref{sec-weak-lower-bound-component} we prove a preliminary weaker lower bound on $|\clusterpi{1}{n}|$ that we will rely on in Section \ref{sec-almost-sure-positive-limit}.

\subsection{Almost sure convergence of the component size of the first vertex}
\label{as}

In this section, we will prove the following proposition, which gives the almost sure convergence claimed in Theorem \ref{thm-max-comp}(a) for $i=1$: 

\begin{prop}[Almost sure convergence of percolated component of vertex 1]\label{prop:ubc1}
    There exists a finite random variable $\zeta_1$ such that
    $$
    n^{-\alpha(\pi)}|\clusterpi{1}{n}| \convas \zeta_1.
    $$
\end{prop}

To prove Proposition \ref{prop:ubc1}, it will be convenient to use the continuous-time construction $(G^\pi_t)_{t\geq 0}$ defined in Section \ref{sec:int}. As before,  $(|\clusterpi{1}{t}|)_{t\geq 0}$ denotes the size process of the component containing vertex one. Recall the evolution of $|\clusterpi{1}{t}|$ described in \eqref{dynamics-connected-component-continuous}.

For $k \ge 2$, recall the definition of the $k^{\rm th}$ susceptibility in continuous time in \eqref{susdef}. 
Using the above description the following is easy to check: 

\begin{lem}[Super-martingale for component-size process]
\label{lem:super-mg}
    In continuous time, the process
    \[ M(t) = |\clusterpi{1}{t}| \exp\left(-\int_0^t [2\pi^2S_2^\pi(u) + 2\pi]~ du \right), \qquad t\geq 0, \]
    is a non-negative super-martingale, and thus converges almost surely to a finite random variable $M(\infty)$.
\end{lem}
\begin{proof} See \eqref{dynamics-connected-component-continuous} and \eqref{super-martingale}.
\end{proof}

\begin{proof}[Proof of Proposition \ref{prop:ubc1}]
Recall that, in continuous time, the size process $N(\cdot)$ is a rate-one Yule process. For $n\in \mathbb{N}$, define the stopping time $T_n := \inf\set{t\geq 0\colon N(t) = n}$. Standard results about the Yule process (see, e.g., \cite[Chapter 2]{Norr98}) imply that 
\begin{equation}
\label{eqn:tn-as}
    T_n -\log{n} \convas W_\infty,
\end{equation}
where $W_\infty \sim \log{\cE}$ with $\cE\sim {\mathrm Exp}(1)$. By Lemma \ref{lem:super-mg}, as $n\rightarrow \infty$,
$$
M(T_n) = |\clusterpi{1}{n}| \exp\left(-\int_0^{T_n} [2\pi^2S_2^\pi(u) + 2\pi]~ du \right) \convas M(\infty).
$$
Observe that, writing $w(u) := 2\pi^2\csusceptibilitypi{u}{2} + 2\pi$,
\begin{align*}
    &\exp\left(-\alpha(\pi)\log n\right)\,\exp\left(\int_0^{T_n}w(u)du\right) = \exp\left(-\int_0^{\log n}(\alpha(\pi) - w(u))du\right)\,\exp\left(\int_{\log n}^{T_n}w(u)du\right)\\
    &\qquad =  \exp\left(-\int_0^{\log n}(\alpha(\pi) - w(u))du\right)\,\exp\left((T_n - \log n)\alpha(\pi)\right)\,\exp\left(\int_{\log n}^{T_n}(w(u) - \alpha(\pi))du\right).
\end{align*}
By \eqref{eqn:tn-as} and the fact that $w(u) \convas \alpha(\pi)$ as $u \to \infty$, we have $\exp\left(\int_{\log n}^{T_n}(w(u) - \alpha(\pi))du\right) \convas 1$ as $n \to \infty$. Again, by \eqref{eqn:tn-as}, $\exp\left((T_n - \log n)\alpha(\pi)\right) \convas \exp\left(\alpha(\pi)W_{\infty}\right)$  as $n \to \infty$. Hence, to show the almost sure convergence of the above, it suffices to show, almost surely,
\begin{equation}\label{ctS}
    \int_0^{\infty}|\alpha(\pi) - w(u)|du < \infty.
\end{equation}
Note that
$$
\int_0^{\infty}|\alpha(\pi) - w(u)|du = \sum_{n=1}^{\infty}(T_{n+1} - T_n)|\alpha(\pi) - w(T_n)| = 2\pi^2 \sum_{n=1}^{\infty}(T_{n+1} - T_n)|\susceptibilitypi{\infty}{2} - \susceptibilitypi{n}{2}|.
$$
From Proposition \ref{prop-conrates2}, we know that there exists $\gamma>0$ such that, almost surely for all $n \in \mathbb{N}$, $|\susceptibilitypi{n}{2} - \susceptibilitypi{\infty}{2}| \le Cn^{-\gamma}$, for some $C < \infty$. Moreover, since $n(T_{n+1} - T_n)$ are iid Exponential random variables with mean one, a standard application of the Borel-Cantelli Lemma implies that $\limsup_{n \to \infty}n(T_{n+1} - T_n)/\log n < \infty$ almost surely. From these two observations, we conclude that the above sum converges almost surely, and, hence, 
$
\exp\left(-\alpha(\pi)\log n\right)\,\exp\left(\int_0^{T_n}w(u)du\right)
$
converges almost surely. This implies that
$$
n^{-\alpha(\pi)}|\clusterpi{1}{n}| = \exp\left(-\alpha(\pi)\log n\right)\,\exp\left(\int_0^{T_n}w(u)du\right)M(T_n)
$$
converges almost surely as $n \to \infty$, which completes the proof of Proposition \ref{prop:ubc1}.
\end{proof}

\subsection{Almost sure weak lower bound on component size of the first vertex}
\label{sec-weak-lower-bound-component}

The main goal of this and the next section is to prove that the limiting random variable in Proposition \ref{prop:ubc1} is almost surely {\em positive}. In this section, we provide the preliminaries for this proof: In Proposition \ref{complb} we provide an almost sure weak {\em lower bound} on the size of the first component, which has an arbitrarily small correction to the scaling exponent.
Lemma \ref{s3bd} proves an integrability condition for the third susceptibility. In this section, $C,C',C_1,C_2,\dots$ will denote universal positive constants. The following result gives the weak lower bound on $|\clusterpi{1}{n}|$ we are after:

\begin{prop}[Weak lower bound on component of vertex 1]\label{complb}
    For any $\delta > 0$, the first component satisfies
\begin{equation}
\label{eqn:806}
    \mathbb{P} \left( \liminf_{n \to \infty} \frac{|\clusterpi{1}{n}|}{n^{(1 - \delta) \alpha(\pi)}} = \infty \right) = 1.
    \end{equation}
\end{prop}

\begin{proof} First note that, by Proposition \ref{prop:trunc-suscep-as}, 
\begin{equation}
\label{eqn:as-convg-trunc-root}
    \mathbb{E} \Big[ |\clusterpi{o_n}{n}| \indic{|\clusterpi{o_n}{n}| \leq K} \mid G^{\pi}_n\Big] = \susceptibilitypi{n}{\sss 2,K} \convas \mathbb{E} \left[ |\clusterpi{\sss\varnothing}{\infty}| \indic{|\clusterpi{\sss\varnothing}{\infty}| \le K}\right].
\end{equation}
where $\clusterpi{\sss\varnothing}{\infty}$ is the percolated component size of the root in the local limit.
Recall the enumeration of components of $G^{\pi}_n$, in descending order of their oldest vertex member, given by $\clusterpi{\sss \geq 1}{n}, \dots, \clusterpi{\sss \ge n}{n}$, with $\clusterpi{\sss \geq 1}{n} = \clusterpi{1}{n}$. Let $\mathscr{F}_n$ denote the natural filtration of the graph process $(G^{\pi}_n)_{n\geq1}$.
\smallskip

Take any $K \geq 1$. Define a coupled process $\mathscr{C}^{\sss (K,\pi)}_{\sss \ge 1}(n)=\mathscr{C}^{\sss (K,\pi)}_{1}(n), \mathscr{C}^{\sss (K,\pi)}_{\sss \ge 2}(n), \dots, \mathscr{C}^{\sss (K,\pi)}_{\sss \ge n}(n),$ adapted to $\mathscr{F}_n$, as follows:

\begin{itemize}

\item[1.] Let $\alpha_K := \inf \{ n \geq 1: |\clusterpi{1}{n}| \geq K + 1 \}$. The two processes evolve synchronously for $n \leq \alpha_K$.

\item[2.] For $n > \alpha_K$, if the $(n+1)^{\rm th}$ vertex attempts to connect to $\mathscr{C}^{\sss (K,\pi)}_{1}(n)$ and a $\mathscr{C}^{\sss (K,\pi)}_{\sss \ge j}(n)$ for $j \geq 2$ with $|\mathscr{C}^{\sss (K,\pi)}_{\sss \ge j}(n)| > K$, then
   \begin{equation*}
       \mathscr{C}^{\sss (K,\pi)}_{1}(n+1) = \mathscr{C}^{\sss (K,\pi)}_{1}(n) \cup \{n+1\} \quad \text{(with associated edge)},
   \end{equation*}
and $\mathscr{C}^{\sss (K,\pi)}_{\sss \ge j}(n+1) = \mathscr{C}^{\sss (K,\pi)}_{\sss \ge j}(n)$ for all $j \geq 2$. Otherwise, the dynamics are the same as those of $\{\clusterpi{1}{n}, \dots, \clusterpi{\sss \ge n}{n}\}$.
\end{itemize}
\smallskip

It can be verified that $|\mathscr{C}^{\sss (K,\pi)}_{1}(n)| \leq |\clusterpi{1}{n}|$ almost surely, for all $n \geq 0$. Moreover, for all $n \geq 0$,
\begin{equation}
    \sum_{j=1}^{n} |\mathscr{C}^{\sss (K,\pi)}_{\sss \ge j}(n)|^2 \leq \sum_{j=1}^{n} |\clusterpi{\sss \ge j}{n}|^2,
\end{equation}
and
\begin{equation}
    \susceptibilitypi{n}{\sss 2,K} := \frac{1}{n}\sum_{j=1}^{n} |\clusterpi{\sss \ge j}{n}|^2 \indic{|\clusterpi{\sss \ge j}{n}| \leq K} = \frac{1}{n}\sum_{j=1}^{n} |\mathscr{C}^{\sss (K,\pi)}_{\sss \ge j}(n)|^2 \indic{|\mathscr{C}^{\sss (K,\pi)}_{\sss \ge j}(n)| \leq K},
\end{equation}
where the second equality holds since the dynamics of connected components of sizes at most $K$ are {\em unaffected} by our change in dynamics.

Observing that, on the event $\{n \geq \alpha_K\} \in \mathscr{F}_n$,
we can expand
\begin{align*}
    &\mathbb{E} \left[|\mathscr{C}^{\sss (K,\pi)}_{1}(n+1)| - |\mathscr{C}^{\sss (K,\pi)}_{1}(n)|\mid \mathscr{F}_n \right]\\
    &\quad = \sum_{j=2}^{n} \frac{2\pi^2}{n^2} \left( |\mathscr{C}^{\sss (K,\pi)}_{1}(n)||\mathscr{C}^{\sss (K,\pi)}_{\sss \ge j}(n)| (|\mathscr{C}^{\sss (K,\pi)}_{\sss \ge j}(n)|\indic{|\mathscr{C}^{\sss (K,\pi)}_{\sss \ge j}(n)| \leq K} + 1) \right)\\
     &\qquad + \frac{\pi^2 |\mathscr{C}^{\sss (K,\pi)}_{1}(n)|^2}{n^2} + \frac{2\pi (1-\pi)}{n} |\mathscr{C}^{\sss (K,\pi)}_{1}(n)|\\
     &\quad = \sum_{j=1}^{n} \frac{2\pi^2}{n^2} \left( |\mathscr{C}^{\sss (K,\pi)}_{1}(n)||\mathscr{C}^{\sss (K,\pi)}_{\sss \ge j}(n)| (|\mathscr{C}^{\sss (K,\pi)}_{\sss \ge j}(n)|\indic{|\mathscr{C}^{\sss (K,\pi)}_{\sss \ge j}(n)| \leq K} + 1) \right)\\
     &\qquad - \frac{2\pi^2}{n^2} \left( |\mathscr{C}^{\sss (K,\pi)}_{1}(n)|^2(|\mathscr{C}^{\sss (K,\pi)}_{1}(n)|\indic{|\mathscr{C}^{\sss (K,\pi)}_{1}(n)| \leq K} + 1)\right)
      + \frac{\pi^2 |\mathscr{C}^{\sss (K,\pi)}_{1}(n)|^2}{n^2} + \frac{2\pi (1-\pi)}{n} |\mathscr{C}^{\sss (K,\pi)}_{1}(n)|\\
     &\quad = \frac{2\pi^2}{n} \susceptibilitypi{n}{\sss 2,K} |\mathscr{C}^{\sss (K,\pi)}_{1}(n)| + \frac{2\pi^2}{n} |\mathscr{C}^{\sss (K,\pi)}_{1}(n)|- \frac{2\pi^2}{n^2} \left( |\mathscr{C}^{\sss (K,\pi)}_{1}(n)|^2(|\mathscr{C}^{\sss (K,\pi)}_{1}(n)|\indic{|\mathscr{C}^{\sss (K,\pi)}_{1}(n)| \leq K} + 1)\right)\\
     &\qquad + \frac{\pi^2 |\mathscr{C}^{\sss (K,\pi)}_{1}(n)|^2}{n^2} + \frac{2\pi (1-\pi)}{n} |\mathscr{C}^{\sss (K,\pi)}_{1}(n)|.
\end{align*}

Simplifying further, we obtain
\begin{multline}
\label{eqn:S}
\mathbb{E} \left[|\mathscr{C}^{\sss (K,\pi)}_{1}(n+1)| - |\mathscr{C}^{\sss (K,\pi)}_{1}(n)|\mid \mathscr{F}_n \right]\\
    =\left( 2\pi^2 \susceptibilitypi{n}{\sss 2,K} + 2\pi \right) \frac{|\mathscr{C}^{\sss (K,\pi)}_{1}(n)|}{n} - \left( \frac{2\pi^2}{n^2} |\mathscr{C}^{\sss (K,\pi)}_{1}(n)|^3 \indic{|\mathscr{C}^{\sss (K,\pi)}_{1}(n)| \leq K} + \frac{\pi^2 |\mathscr{C}^{\sss (K,\pi)}_{1}(n)|^2}{n^2} \right).
\end{multline}

In the following, we work conditionally on $\mathscr{F}_N$, where $N \geq \alpha_K \vee 2 + 1$ is any stopping time.
\smallskip

Define
\begin{equation}
    X_n := \log |\mathscr{C}^{\sss (K,\pi)}_{1}(n)|.
\end{equation}

Consider the martingale
\begin{equation*}
    \Psi_n := X_n - \sum_{j=N}^{n} \mathbb{E} [X_j - X_{j-1} \mid \mathscr{F}_{j-1}], \quad n \geq N.
\end{equation*}

Note that
\begin{equation*}
    |X_n - X_{n-1}| = \log \left( \frac{|\mathscr{C}^{\sss (K,\pi)}_{1}(n)|}{|\mathscr{C}^{\sss (K,\pi)}_{1}(n-1)|} \right) \leq \frac{|\mathscr{C}^{\sss (K,\pi)}_{1}(n)| - |\mathscr{C}^{\sss (K,\pi)}_{1}(n-1)|}{|\mathscr{C}^{\sss (K,\pi)}_{1}(n-1)|} \leq \frac{K}{|\mathscr{C}^{\sss (K,\pi)}_{1}(n-1)|}.
\end{equation*}

Thus,
\begin{equation*}
    |\Psi_n - \Psi_{n-1}| \leq \frac{2K}{|\mathscr{C}^{\sss (K,\pi)}_{1}(n-1)|}.
\end{equation*}

We will derive a concentration inequality for the martingale $(\Psi_n)_{n \ge N}$. Although the jump sizes of this martingale are bounded by $K$, for our purposes, we need more refined control than that provided by the Azuma-Hoeffding inequality, and, rather, we will rely on estimates on the box and quadratic variations of the martingale. For any $n \geq N$, the box variation of $(\Psi_n)_{n\geq 1}$ can be bounded as
\begin{align*}
    [\Psi]_n &:= \sum_{j=N}^{n} |\Psi_j - \Psi_{j-1}|^2 \leq \frac{2K}{|\mathscr{C}^{\sss (K,\pi)}_{1}(N-1)|} \sum_{j=N}^{n} |\Psi_j - \Psi_{j-1}|\\
    &\le \frac{2K}{|\mathscr{C}^{\sss (K,\pi)}_{1}(N-1)|}\left[ X_n + \sum_{j=N}^{n} \mathbb{E} [X_j - X_{j-1} \mid \mathscr{F}_{j-1}] \right],
    \end{align*}
where we use that $X_j$ is almost surely non-decreasing, so that
    \[
    |\Psi_j - \Psi_{j-1}|\leq X_j - X_{j-1}+ \mathbb{E} [X_j - X_{j-1} \mid \mathscr{F}_{j-1}],
    \]
combined with a telescoping sum identity. Thus,
\begin{align*}
    [\Psi]_n &\le \frac{2K}{|\mathscr{C}^{\sss (K,\pi)}_{1}(N-1)|}\left[ \Psi_n + 2 \sum_{j=N}^{n} \mathbb{E} [X_j - X_{j-1} \mid \mathscr{F}_{j-1}] \right]\\
    &\le \frac{2K}{|\mathscr{C}^{\sss (K,\pi)}_{1}(N-1)|}\left[ \Psi_n + 2 \sum_{j=N}^{n} \frac{1}{|\mathscr{C}^{\sss (K,\pi)}_{1}(j-1)|} \mathbb{E} [|\mathscr{C}^{\sss (K,\pi)}_{1}(j)| - |\mathscr{C}^{\sss (K,\pi)}_{1}(j-1)| \mid \mathscr{F}_{j-1}] \right]. 
\end{align*}
By using \eqref{eqn:S}, where we note that the $|\mathscr{C}^{\sss (K,\pi)}_{1}(j-1)|$ factors cancel, we obtain
\begin{align*}
[\Psi]_n &\le \frac{2K}{|\mathscr{C}^{\sss (K,\pi)}_{1}(N-1)|}\left[ \Psi_n + 2 \sum_{j=N}^{n} \frac{1}{j-1} \left(2\pi^2 \susceptibilitypi{j-1}{\sss 2,K} + 2\pi\right) \right].
\end{align*}
Using that 
    \eqn{
    \susceptibilitypi{j-1}{\sss 2,K}\leq \frac{K}{j-1}\sum_{i\geq 1} |\mathscr{C}^{\sss (K,\pi)}_{\sss \geq i}(j-1)|=K,
    }
we obtain

\begin{align*}
[\Psi]_n &\le \frac{2K}{|\mathscr{C}^{\sss (K,\pi)}_{1}(N-1)|}\left[ \Psi_n + 2 \sum_{j=N}^{n} \frac{1}{j-1} \left(2\pi^2 K + 2\pi\right) \right]\\
&\le \frac{2K}{|\mathscr{C}^{\sss (K,\pi)}_{1}(N-1)|}\left[ \Psi_n + C K \log n \right].
\end{align*}

Similarly, we bound the quadratic variation of the martingale as
\begin{align*}
    \langle \Psi \rangle_n &:= \sum_{j=N}^{n} \mathbb{E} \left[ (\Psi_j - \Psi_{j-1})^2 \mid \mathscr{F}_{j-1} \right]\\
    & \le \frac{2K}{|\mathscr{C}^{\sss (K,\pi)}_{1}(N-1)|}\sum_{j=N}^{n} \mathbb{E} [X_j - X_{j-1} \mid \mathscr{F}_{j-1}] \le \frac{2CK^2}{|\mathscr{C}^{\sss (K,\pi)}_{1}(N-1)|}\log n.
\end{align*}
Now, we will use a concentration inequality in \cite{BerTou08}*{Theorem 2.1}, which says that for any martingale sequence $(M_n)_{n \ge 1}$ adapted to some filtration $(\mathcal{F}_n)_{n \ge 1}$ with box variation $([M]_n)_{n \ge 1}$ and quadratic variation $(\langle M \rangle_n)_{n \ge 1}$, and any positive $x,y$,
\begin{equation}\label{eq:martconc}
    \prob\left(|M_n| \ge x, \, [M]_n + \langle M \rangle_n \le y\right) \le 2e^{-x^2/(2y)}.
\end{equation}

Using \eqref{eq:martconc}, along with the above bounds on the box and quadratic variations of the martingale $(\Psi_n)_{n \ge 1}$, we obtain
\begin{align}
\label{concentration-bound-1}
    &\mathbb{P} \big( \delta \log n \leq |\Psi_n| \leq 2\delta \log n \mid \mathscr{F}_N \big)\\
    &\le \mathbb{P}\left(|\Psi_n| \geq \delta \log n, \, [\Psi]_n + \langle \Psi \rangle_n \le \frac{2K}{|\mathscr{C}^{\sss (K,\pi)}_{1}(N-1)|}(2\delta \log n + 2CK \log n)~\Big|~\mathscr{F}_N\right)\nn\\
    & \leq 2 \exp \left( - \frac{\delta^2 |\mathscr{C}^{\sss (K,\pi)}_{1}(N-1)| (\log n)^2}{8 K (\delta + C K) (\log n)} \right)\nn\\
    & = 2\, n^{- \frac{\delta^2}{8K(\delta + C K)} |\mathscr{C}^{\sss (K,\pi)}_{1}(N-1)|}.\nn
\end{align}
In general, for any $l \in \mathbb{N}_0$,
\begin{equation*}
    \mathbb{P} \left( 2^l \delta \log n \leq |\Psi_n| \leq 2^{l+1} \delta \log n \mid \mathscr{F}_N \right)
    \leq 2\, n^{- \frac{(2^l \delta)^2}{8K(2^l \delta + C K)}|\mathscr{C}^{\sss (K,\pi)}_{1}(N-1)|}\leq 
    n^{-\big(\frac{2^l \delta}{16K}\wedge \frac{(2^l \delta)^2}{16C K^2}\big)|\mathscr{C}^{\sss (K,\pi)}_{1}(N-1)|}.
\end{equation*}
By a union bound, for any $n \geq N,$
\begin{equation*}
    \mathbb{P} \left( |\Psi_n| \geq \delta \log n \mid \mathscr{F}_N \right) \leq C_1(\delta, K) n^{-\big(\frac{\delta}{16K}\wedge \frac{\delta^2}{16C K^2}\big)|\mathscr{C}^{\sss (K,\pi)}_{1}(N-1)|},
\end{equation*}
where the constant $C(\delta,K)$ depends on $\delta$ and $K$.
\smallskip

Choose
\begin{equation*}
    N = \inf \left\{ n \geq \alpha_K \vee 2 + 1 \colon |\mathscr{C}^{\sss (K,\pi)}_{1}(N-1)| \geq \frac{32K}{\delta}\vee \frac{32C K^2}{\delta^2}\right\}.
\end{equation*}
Then,
\begin{equation*}
    \sum_{n=N}^{\infty} \mathbb{P} \big(|\Psi_n| \geq \delta \log n \mid \mathscr{F}_N \big) < \infty.
\end{equation*}
Hence, by the Borel-Cantelli lemma, almost surely, for all sufficiently large $n$,
\begin{equation}\label{eqn:star}
    -\delta \log n \leq \log |\mathscr{C}^{\sss (K,\pi)}_{1}(n)| - \sum_{j=1}^{n} \mathbb{E} \left[ \log |\mathscr{C}^{\sss (K,\pi)}_{1}(j)| - \log |\mathscr{C}^{\sss (K,\pi)}_{1}(j-1)| \mid \mathscr{F}_{j-1} \right] \leq \delta \log n.
\end{equation}
Thus, to obtain a lower bound on $|\clusterpi{1}{n}|$ in the almost sure sense, it suffices to obtain a lower bound on
\begin{equation*}
    \sum_{j=1}^{n} \mathbb{E} \left[ \log |\mathscr{C}^{\sss (K,\pi)}_{1}(j)| - \log |\mathscr{C}^{\sss (K,\pi)}_{1}(j-1)| \mid \mathscr{F}_{j-1} \right].
\end{equation*}
\smallskip

Recalling $X_n := \log |\mathscr{C}^{\sss (K,\pi)}_{1}(n)|$, we note that
\begin{equation*}
    X_{n+1} - X_{n} \geq \frac{1}{|\mathscr{C}^{\sss (K,\pi)}_{1}(n)| + K + 1} (|\mathscr{C}^{\sss (K,\pi)}_{1}(n+1)| - |\mathscr{C}^{\sss (K,\pi)}_{1}(n)|).
\end{equation*}
By \eqref{eqn:S}, we obtain
\begin{equation}\label{lb1}
    \mathbb{E} \left[ X_{n+1} - X_n \mid \mathscr{F}_n \right] \geq \frac{1}{n} \frac{(2\pi^2 \susceptibilitypi{n}{\sss 2,K} + 2\pi) |\mathscr{C}^{\sss (K,\pi)}_{1}(n)|}{|\mathscr{C}^{\sss (K,\pi)}_{1}(n)| + K + 1}
    - \frac{1}{n} \left( \frac{2\pi^2}{n} |\mathscr{C}^{\sss (K,\pi)}_{1}(n)|^2 + \frac{\pi^2}{n} |\mathscr{C}^{\sss (K,\pi)}_{1}(n)| \right).
\end{equation}
Define
\begin{equation*}
    \Theta_n^K := \sum_{j=1}^{n} \left( \frac{2\pi^2}{j^2} |\mathscr{C}^{\sss (K,\pi)}_{1}(j)|^2 + \frac{\pi^2}{j^2} |\mathscr{C}^{\sss (K,\pi)}_{1}(j)| \right), \quad n \geq 1.
\end{equation*}
Take any
\begin{equation*}
    \varepsilon \in \left( 0, C \wedge \frac{\sqrt{8\pi^2 - 8\pi + 1}}{6} \right).
\end{equation*}

Fix $K \geq 1$ such that, with 
$
    \susceptibilitypi{\infty}{\sss 2,K} := \mathbb{E} \left[ |\clusterpi{\sss\varnothing}{\infty}| \indic{|\clusterpi{\sss\varnothing}{\infty}| \leq K}\right]
$, where we recall that $|\clusterpi{\sss\varnothing}{\infty}|$ is the percolated component size of the root in the local limit, and
\begin{equation*}
    \susceptibilitypi{\infty}{2} := \mathbb{E} [|\clusterpi{\sss\varnothing}{\infty}|] = \frac{1 - 4\pi - \sqrt{8\pi^2 - 8\pi + 1}}{4\pi^2},
\end{equation*}
we have
\begin{equation*}
    \susceptibilitypi{\infty}{2} - \susceptibilitypi{\infty}{\sss 2,K} < \varepsilon.
\end{equation*}
For $L > 0$, define
\begin{equation*}
    T_L := \inf \left\{ n \geq \alpha_K + L \colon \left| 2\pi^2 \susceptibilitypi{n}{\sss 2,K} + 2\pi - \alpha(\pi) \right| \geq 2\varepsilon \right\}.
\end{equation*}
As $\susceptibilitypi{n}{\sss 2,K} \convas \susceptibilitypi{\infty}{\sss 2,K}$,
\begin{equation}\label{lb2}
    \lim_{L \to \infty} \mathbb{P} \left( T_L < \infty \right) = 0.
\end{equation}
By the almost sure weak upper bound on the maximal component in Proposition \ref{prop:max-as-up-bound},
\begin{equation}\label{lb3}
    \limsup_{n \to \infty}\Theta_n^K \le 
    \sum_{j=1}^{\infty} \left( \frac{2\pi^2}{j^2} |\mathscr{C}^{\sss (K,\pi)}_{1}(j)|^2 + \frac{\pi^2}{j^2} |\mathscr{C}^{\sss (K,\pi)}_{1}(j)| \right) < \infty \quad \text{almost surely}.
\end{equation}

Define
\begin{equation*}
    \tau_\varepsilon := \inf \left\{ n \geq 1 \colon \frac{|\mathscr{C}^{\sss (K,\pi)}_{1}(n)|}{|\mathscr{C}^{\sss (K,\pi)}_{1}(n)| + K + 1} \geq 1 - \varepsilon \right\}.
\end{equation*}
Then, for any $L > 0$, on the event $\{ n > \tau_\varepsilon \vee (\alpha_K + L) \} \cap \{ T_L = \infty \}$, by \eqref{lb1},
\begin{align*}
    \sum_{j=1}^{n} \mathbb{E} \left[ X_j - X_{j-1} \mid \mathscr{F}_{j-1} \right] &\geq \sum_{j=\tau_\varepsilon \vee (\alpha_K + L)}^{n-1} \frac{1}{j} \frac{(2\pi^2 \susceptibilitypi{j-1}{\sss 2,K} + 2\pi) |\mathscr{C}^{\sss (K,\pi)}_{1}(j)|}{|\mathscr{C}^{\sss (K,\pi)}_{1}(j)| + K + 1} - \Theta_{n-1}^K \\
    &\geq (1 - \varepsilon) \sum_{j=\tau_\varepsilon \vee (\alpha_K + L)}^{n-1} \frac{\alpha(\pi) - 2\varepsilon}{j} - \Theta_{n-1}^K\\
    &\geq (1 - \varepsilon)(\alpha(\pi) - 2\varepsilon)\int_1^n\frac{1}{x}dx - \Theta_{n-1}^K - (1 - \varepsilon) \sum_{j=1}^{\tau_\varepsilon \vee (\alpha_K + L) - 1} \frac{\alpha(\pi) - 2\varepsilon}{j}\\
    &= (1 - \varepsilon)(\alpha(\pi) - 2\varepsilon)\log n - \Theta_{n-1}^K - (1 - \varepsilon) \sum_{j=1}^{\tau_\varepsilon \vee (\alpha_K + L) - 1} \frac{\alpha(\pi) - 2\varepsilon}{j}.
\end{align*}
Using this in conjunction with \eqref{eqn:star}, on the event $\{ n > \tau_\varepsilon \vee (\alpha_K + L) \} \cap \{ T_L = \infty \}$, we obtain, for all sufficiently large $n$,
\begin{equation*}
    \log |\mathscr{C}^{\sss (K,\pi)}_{1}(n)| \geq \left[ (1 - \varepsilon) (\alpha(\pi) - 2\varepsilon) - \varepsilon \right] \log n - \Theta_{n-1}^K - (1 - \varepsilon) \sum_{j=1}^{\tau_\varepsilon \vee (\alpha_K + L) - 1} \frac{\alpha(\pi) - 2\varepsilon}{j}.
\end{equation*}
In particular, by \eqref{lb3}, for any $L > 0$ and $\eps>0$,
\begin{equation}
    \mathbb{P} \left( \limsup_{n \to \infty} \frac{|\mathscr{C}^{\sss (K,\pi)}_{1}(n)|}{n^{(1 - \varepsilon) (\alpha(\pi) - 2\varepsilon) - \varepsilon}} > 0 \right) \geq \mathbb{P} \left( T_L = \infty \right).
\end{equation}
Taking $L \to \infty$ and using \eqref{lb2}, recalling $|\mathscr{C}^{\sss (K,\pi)}_{1}(n)| \leq |\clusterpi{1}{n}|$ and noting that $\varepsilon \in \left( 0, C \wedge \frac{\sqrt{8\pi^2 - 8\pi + 1}}{6} \right)$ is arbitrary, we conclude that, for any $\delta>0$,
\begin{equation*}
    \mathbb{P} \left( \liminf_{n \to \infty} \frac{|\clusterpi{1}{n}|}{n^{(1 - \delta) \alpha(\pi)}} = \infty \right) = 1.
\end{equation*}
This completes the proof of Proposition \ref{complb}.
\end{proof}

A key technical step in our analysis concerns the behavior of the {\em third susceptibility} in continuous time, which we recall from \eqref{susdef} and investigate next:

\begin{lem}[Integrated third susceptibility]
\label{s3bd}
    For any $\vep \in (0, \alpha(\pi) \wedge (1- 2\alpha(\pi)))$, almost surely,
    $$
    \int_0^{\infty}\e^{-(\alpha(\pi) - \vep)t} \csusceptibilitypi{t}{3} dt < \infty.
    $$
\end{lem}
\begin{proof}
    The action of the generator of the associated continuous-time Markov process on $\csusceptibilitypi{t}{3}$ is given by
    \begin{align*}
        &\mathcal{L} \csusceptibilitypi{t}{3} = N(t) \left(\frac{1}{N(t) + 1} - \frac{1}{N(t)}\right)\sum_{i=1}^{N(t)}|\clusterpi{\sss \geq i}{t}|^3\\
        & \quad + \sum_{i < j} \left((|\clusterpi{\sss \geq i}{t}| + |\clusterpi{\sss \geq j}{t}| + 1)^3 - |\clusterpi{\sss \geq i}{t}|^3 - |\clusterpi{\sss \geq j}{t}|^3\right)\frac{2\pi^2|\clusterpi{\sss \geq i}{t}||\clusterpi{\sss \geq j}{t}|}{N(t)^2}\\
        &\quad + \sum_{i} \left((|\clusterpi{\sss \geq i}{t}|  + 1)^3 - |\clusterpi{\sss \geq i}{t}|^3\right)\left(\frac{2\pi(1-\pi)|\clusterpi{\sss \geq i}{t}|}{N(t)} + \frac{\pi^2|\clusterpi{\sss \geq i}{t}|^2}{N(t)^2}\right) + (1-\pi)^2\\
        &= -\frac{N(t)}{N(t) + 1}\csusceptibilitypi{t}{3}\\
        & \quad + \sum_{i < j} \left(3|\clusterpi{\sss \geq i}{t}|^2( |\clusterpi{\sss \geq j}{t}| + 1) + 3|\clusterpi{\sss \geq i}{t}|( |\clusterpi{\sss \geq j}{t}| + 1)^2\right)\frac{2\pi^2|\clusterpi{\sss \geq i}{t}||\clusterpi{\sss \geq j}{t}|}{N(t)^2}\\
        &\quad + \sum_{i < j} \left(( |\clusterpi{\sss \geq j}{t}| + 1)^3 - |\clusterpi{\sss \geq j}{t}|^3\right)\frac{2\pi^2|\clusterpi{\sss \geq i}{t}||\clusterpi{\sss \geq j}{t}|}{N(t)^2}\\
        &\quad + \sum_{i} \left((|\clusterpi{\sss \geq i}{t}|  + 1)^3 - |\clusterpi{\sss \geq i}{t}|^3\right)\left(\frac{2\pi(1-\pi)|\clusterpi{\sss \geq i}{t}|}{N(t)} + \frac{\pi^2|\clusterpi{\sss \geq i}{t}|^2}{N(t)^2}\right) + (1-\pi)^2\\
        & = -\frac{N(t)}{N(t) + 1}\csusceptibilitypi{t}{3}\\
        & \quad + 3\sum_{i < j} \left(|\clusterpi{\sss \geq i}{t}| |\clusterpi{\sss \geq j}{t}| + |\clusterpi{\sss \geq j}{t}|^2 + |\clusterpi{\sss \geq i}{t}| + 1 + 2|\clusterpi{\sss \geq j}{t}|\right)\frac{2\pi^2|\clusterpi{\sss \geq i}{t}|^2|\clusterpi{\sss \geq j}{t}|}{N(t)^2}\\
        &\quad + \sum_{i < j} \left( 3|\clusterpi{\sss \geq j}{t}|^2 +  3|\clusterpi{\sss \geq j}{t}| + 1\right)\frac{2\pi^2|\clusterpi{\sss \geq i}{t}||\clusterpi{\sss \geq j}{t}|}{N(t)^2}\\
        &\quad + \sum_{i} \left(3|\clusterpi{\sss \geq i}{t}|^2 +  3|\clusterpi{\sss \geq i}{t}| + 1\right)\left(\frac{2\pi|\clusterpi{\sss \geq i}{t}|}{N(t)} + \frac{\pi^2|\clusterpi{\sss \geq i}{t}|^2}{N(t)^2}\right) + (1-\pi)^2.
    \end{align*}
    Simplifying each of the three above sums, we obtain
    \begin{align*}
        &3\sum_{i < j} \left(|\clusterpi{\sss \geq i}{t}| |\clusterpi{\sss \geq j}{t}| + |\clusterpi{\sss \geq j}{t}|^2 + |\clusterpi{\sss \geq i}{t}| + 1 + 2|\clusterpi{\sss \geq j}{t}|\right)\frac{2\pi^2|\clusterpi{\sss \geq i}{t}|^2|\clusterpi{\sss \geq j}{t}|}{N(t)^2}\\
        &= \left(6\pi^2 \csusceptibilitypi{t}{2} + 3\pi^2\right) \csusceptibilitypi{t}{3} + 6\pi^2(\csusceptibilitypi{t}{2})^2 + 3\pi^2\csusceptibilitypi{t}{2}
        - \frac{6\pi^2\csusceptibilitypi{t}{5} + 9\pi^2\csusceptibilitypi{t}{4} + 3\pi^2 \csusceptibilitypi{t}{3}}{N(t)},
    \end{align*}
    \begin{align*}
        &\sum_{i < j} \left( 3|\clusterpi{\sss \geq j}{t}|^2 +  3|\clusterpi{\sss \geq j}{t}| + 1\right)\frac{2\pi^2|\clusterpi{\sss \geq i}{t}||\clusterpi{\sss \geq j}{t}|}{N(t)^2}\\
        &\quad = 3\pi^2\csusceptibilitypi{t}{3} + 3\pi^2\csusceptibilitypi{t}{2} + \pi^2 - \sum_{i} \left(3|\clusterpi{\sss \geq i}{t}|^2 +  3|\clusterpi{\sss \geq i}{t}| + 1\right)\frac{\pi^2|\clusterpi{\sss \geq i}{t}|^2}{N(t)^2},
    \end{align*}
    and
    \begin{align*}
        &\sum_{i} \left(3|\clusterpi{\sss \geq i}{t}|^2 +  3|\clusterpi{\sss \geq i}{t}| + 1\right)\left(\frac{2\pi(1-\pi)|\clusterpi{\sss \geq i}{t}|}{N(t)} + \frac{\pi^2|\clusterpi{\sss \geq i}{t}|^2}{N(t)^2}\right)\\
        & \quad = 6\pi(1-\pi) \csusceptibilitypi{t}{3} + 6\pi(1-\pi)\csusceptibilitypi{t}{2} + 2\pi(1-\pi)
        + \sum_{i} \left(3|\clusterpi{\sss \geq i}{t}|^2 +  3|\clusterpi{\sss \geq i}{t}| + 1\right)\frac{\pi^2|\clusterpi{\sss \geq i}{t}|^2}{N(t)^2}.
    \end{align*}
    Using the above, we obtain
    \begin{equation}\label{gens3}
        \mathcal{L} \csusceptibilitypi{t}{3} = (3\alpha(\pi) - 1) \csusceptibilitypi{t}{3} + 6\pi^2(\csusceptibilitypi{t}{2} - \susceptibilitypi{\infty}{2})\csusceptibilitypi{t}{3} + R(t),
    \end{equation}
    where
    \begin{align*}
        R(t) &= \frac{\csusceptibilitypi{t}{3}}{N(t) + 1} + 6\pi^2(\csusceptibilitypi{t}{2})^2 + 6\pi \csusceptibilitypi{t}{2} + 1
        - \frac{6\pi^2\csusceptibilitypi{t}{5} + 9\pi^2\csusceptibilitypi{t}{4} + 3\pi^2 \csusceptibilitypi{t}{3}}{N(t)}\\
        &\le 6\pi^2(\csusceptibilitypi{t}{2})^2 + (6\pi + 1) \csusceptibilitypi{t}{2}.
    \end{align*}
    Take any $\vep \in (0, \alpha(\pi) \wedge (1- 2\alpha(\pi)))$, and  any $\delta>0$ such that $\vep + (1 + 6\pi^2)\delta < \alpha(\pi) \wedge (1 - 2\alpha(\pi))$. Write $$
    Z(t) := \e^{-(\alpha(\pi) - \vep - \delta)t} \csusceptibilitypi{t}{3}, \quad t \ge 0.
    $$
    By \eqref{gens3}, 
    \begin{align}\label{genless}
    \mathcal{L} Z(t) &\leq -(\alpha(\pi) - \vep - \delta) Z(t) + (3\alpha(\pi) - 1 + 6\pi^2|\csusceptibilitypi{t}{2} - \susceptibilitypi{\infty}{2}|)Z(t) + \e^{-(\alpha(\pi) - \vep)t}R(t)\nonumber\\
    &= -(1- 2\alpha(\pi) - \vep - \delta - 6\pi^2|\csusceptibilitypi{t}{2} - \susceptibilitypi{\infty}{2}|)Z(t) + \e^{-(\alpha(\pi) - \vep - \delta)t}R(t).
    \end{align}
    Fix any $M,L>0$. Define the stopping time
    $$
    \tau := \inf\{ t \ge M \colon \csusceptibilitypi{t}{2} \ge L \text{ or } |\csusceptibilitypi{t}{2} - \susceptibilitypi{\infty}{2}| \ge \delta\}.
    $$
    For any $t \ge M$, using \eqref{genless}, 
    \begin{align*}
        \mathbb{E}\left[Z(t \wedge \tau)\right] &\le \mathbb{E}[Z(M)] +  \int_M^{t\wedge \tau}\e^{-(\alpha(\pi) - \vep - \delta)s}R(s)ds\\
        &\le \mathbb{E}[Z(M)] + \int_M^\infty (6\pi^2L^2 + (6\pi + 1)L)\e^{-(\alpha(\pi) - \vep - \delta)s}ds\\
        &= \mathbb{E}[Z(M)] + \frac{6\pi^2L^2 + (6\pi + 1)L}{\alpha(\pi) - \vep - \delta}\e^{-(\alpha(\pi) - \vep - \delta)M}.
    \end{align*}
    In particular, for any $T >M$,
    \begin{align*}
        \mathbb{E}\left[\int_M^{T \wedge \tau} \e^{-(\alpha(\pi) - \vep)t} \csusceptibilitypi{t}{3}dt\right] &= \mathbb{E}\left[\int_M^{T \wedge \tau} \e^{-\delta t}Z(t)dt\right]\\
        &\le \int_M^\infty \e^{-\delta t}\mathbb{E}\left[Z(t \wedge \tau)\right]dt \le C(M,L,\delta) < \infty,
    \end{align*}
    for some finite constant $C(M,L,\delta)$. By the monotone convergence theorem, we obtain
    $$\mathbb{E}\left[\int_M^{\tau} \e^{-(\alpha(\pi) - \vep)t} \csusceptibilitypi{t}{3}dt\right] < \infty.
    $$
    In particular,
    \begin{align*}
        \mathbb{P}\left(\int_0^\infty \e^{-(\alpha(\pi) - \vep)t} \csusceptibilitypi{t}{3}dt < \infty\right) \ge \mathbb{P}\left(\csusceptibilitypi{t}{2} \le L \text{ and } |\csusceptibilitypi{t}{2} - \susceptibilitypi{\infty}{2}| \le \delta \text{ for all } t \ge M\right).
    \end{align*}
The probability on the rhs approaches one as $L \to \infty$ and $M \to \infty$. Thus, Lemma \ref{s3bd} follows.
\end{proof}

\subsection{Almost sure positivity of the limit}
\label{sec-almost-sure-positive-limit}
In this section, we use Proposition \ref{complb}  and Lemma \ref{s3bd} to prove the almost sure positivity of the almost sure limit $\zeta_1$ of $\e^{-\alpha(\pi)t}|\clusterpi{1}{t}|$, thereby completing the proof of Theorem \ref{thm-max-comp}(a). We do this by tracing the stochastic evolution of the {\em reciprocal} of the size of the first component.
\smallskip

We apply the generator of the Markov process governing the continuous-time evolution on $|\clusterpi{1}{t}|^{-1}$ to obtain
\begin{align*}
    \mathcal{L}|\clusterpi{1}{t}|^{-1} &= \sum_{j=2}^{N(t)} \left[\frac{1}{|\clusterpi{1}{t}| + |\clusterpi{\sss \ge j}{t}| + 1} - \frac{1}{|\clusterpi{1}{t}|}\right]\frac{2\pi^2|\clusterpi{1}{t}||\clusterpi{\sss \ge j}{t}|}{N(t)}\\
    &\quad + \left[\frac{1}{|\clusterpi{1}{t}| + 1} - \frac{1}{|\clusterpi{1}{t}|}\right]\left(2\pi(1-\pi)|\clusterpi{1}{t}| + \frac{\pi^2|\clusterpi{1}{t}|}{N(t)}\right)\\
    &\quad = \frac{1}{|\clusterpi{1}{t}|}\left[-\left(\sum_{j=1}^{N(t)}\frac{2\pi^2|\clusterpi{\sss \ge j}{t}|\left(|\clusterpi{\sss \ge j}{t}| + 1\right)}{N(t)} + 2\pi(1-\pi)\right) + R_1(t)\right],
\end{align*}
where
\begin{align*}
    R_1(t) &= \frac{2\pi^2}{N(t)}\sum_{j=2}^{N(t)}\frac{|\clusterpi{\sss \ge j}{t}|(|\clusterpi{\sss \ge j}{t}| + 1)^2}{|\clusterpi{1}{t}| + |\clusterpi{\sss \ge j}{t}| + 1} + \frac{2\pi(1 - \pi)}{|\clusterpi{1}{t}|+1}\\
    &\quad - \frac{\pi^2|\clusterpi{1}{t}|}{N(t)\left(|\clusterpi{1}{t}| + 1\right)} + \frac{2\pi^2|\clusterpi{1}{t}|\left(|\clusterpi{1}{t}| + 1\right)}{N(t)}.
\end{align*}
Now, simplifying the first term in the generator expression, we obtain
\begin{align*}
    \mathcal{L}|\clusterpi{1}{t}|^{-1} = \frac{1}{|\clusterpi{1}{t}|}\left[-\left(2\pi^2S_2^\pi(t) + 2\pi\right) + R_1(t)\right]
     = \frac{1}{|\clusterpi{1}{t}|}\left[-\alpha(\pi)t + R_2(t)\right],
\end{align*}
where
$$
R_2(t) = R_1(t) + 2\pi^2 (\susceptibilitypi{\infty}{2} - S_2^\pi(t)).
$$
From the above, we conclude that
\begin{equation*}
    M^{(-1)}(t) := \exp\left(\alpha(\pi)t - \int_0^tR_2(u)du\right)\frac{1}{|\clusterpi{1}{t}|}
\end{equation*}
is a positive martingale and hence converges almost surely to a finite limit, but we already know that
$$
\e^{\alpha(\pi)t}\frac{1}{|\clusterpi{1}{t}|} \convas \frac{1}{\zeta_1}.
$$
Thus, to show the almost sure positivity of $\zeta_1$, it suffices to show that, almost surely,
\begin{align}\label{r2fin}
    \int_0^\infty|R_2(u)| du < \infty,
\end{align}
which is what we will do now.
\smallskip

Fix any $\vep \in (0, \alpha(\pi) \wedge (1-2\alpha(\pi)))$. From Proposition \ref{complb}, translated to continuous time, we obtain a finite random constant $A$ such that $|\clusterpi{1}{t}| \ge A \e^{(\alpha(\pi) - \vep)t}$ for all $t \ge 0$. Using this fact along with Lemma \ref{s3bd}, we conclude that, almost surely,
\begin{align}\label{ribd}
    \int_0^\infty \left(\frac{2\pi^2}{N(t)}\sum_{j=2}^{N(t)}\frac{|\clusterpi{\sss \ge j}{t}|(|\clusterpi{\sss \ge j}{t}| + 1)^2}{|\clusterpi{1}{t}| + |\clusterpi{\sss \ge j}{t}| + 1}\right) dt &\le 8\pi^2\int_0^\infty \frac{\csusceptibilitypi{t}{3}}{|\clusterpi{1}{t}|}dt\nonumber\\
    &\le \frac{8\pi^2}{A}\int_0^{\infty} \e^{-(\alpha(\pi) - \vep)t}\csusceptibilitypi{t}{3}dt < \infty.
\end{align}
Moreover, using Propositions \ref{prop:ubc1} and \ref{complb}, almost surely,
\begin{align}\label{riibd}
    \int_0^{\infty}\left(\frac{2\pi(1 - \pi)}{|\clusterpi{1}{t}|+1} + \frac{\pi^2|\clusterpi{1}{t}|}{N(t)\left(|\clusterpi{1}{t}| + 1\right)} + \frac{2\pi^2|\clusterpi{1}{t}|\left(|\clusterpi{1}{t}| + 1\right)}{N(t)}\right)dt < \infty.
\end{align}
Finally, by \eqref{ctS}, almost surely,
\begin{align}\label{riiibd}
    \int_0^\infty 2\pi^2|\susceptibilitypi{\infty}{2} - S_2^\pi(t)|dt < \infty .
\end{align}
We conclude that \eqref{r2fin} follows from \eqref{ribd}, \eqref{riibd} and \eqref{riiibd}, thereby completing the proof of Theorem \ref{thm-max-comp}(a) for $i=1$. The proof for $i\geq 2$ is similar, and is therefore omitted.
\qed

\section{Asymptotics of the maximal component}
\label{sec-max-component}
The main goal of this section is to prove Theorem \ref{thm-max-comp}(b), which we will complete in Section \ref{sec-weak-persistence}.
A crucial ingredient is a uniform bound on the size of the components of `late' vertices, as we will discuss in Section \ref{sec-latecomplem}. 

\subsection{Late components cannot be large}
\label{sec-latecomplem}

For $i\geq 1$, recall the definition (in continuous time) of $\clusterpi{\sss \geq i}{t}$ in  \eqref{eqn:cc-less-def}. The following result proves a uniform bound on the size of the components of `late' vertices by controlling the higher-order moments of the normalized collection $\{e^{-\alpha(\pi) t}|\clusterpi{\sss \geq i}{t}|\}$ uniformly in $i > M$ for some large $M$, as well as in $t \ge 0$:

\begin{theorem}[Components of late vertices cannot be large]
\label{latecomplem}
For any $\pi < \pi_c$, $\eta>0$,
    \begin{equation*}
    \lim_{M \to \infty}\mathbb{P}\left(\max_{i>M}\sup_{t \ge 0}  \e^{-\alpha(\pi)t}|\clusterpi{\sss \geq i}{t}| > \eta\right) = 0.
    \end{equation*}
\end{theorem}
\begin{proof}
    Since the proof of Theorem \ref{latecomplem} is technical, let us outline the main ingredients:
\begin{enumeratea}
    \item A stochastic coupling lemma allowing one to compare the maximal size of all components with smallest index above some value $i$ with the maximal component of the whole network started at time $\tau_i$, the arrival time of the $i^{\rm th}$ vertex in continuous time. 
    \item \label{it:b} Uniform `a priori' upper bounds on the size of the normalized maximal component using our tree-graph inequalities formulated in Lemma \ref{lem-higher-moment-bound}. Note that these bounds are suboptimal in a sense because the exponent $\beta$ in the lemma has to be strictly greater than $\alpha(\pi)$.
    \item  We exploit a semi-martingale decomposition of $[\clusterpi{\sss \geq i}{t}]^k$, for large but finite $i$ and appropriately chosen $k$, to obtain uniform bounds on the drift and quadratic variation in this decomposition up to carefully chosen stopping times. Combining this with the a priori bounds in \eqref{it:b} furnishes the uniform control for all $i >M$ and $t \ge 0$, proving the result.
\end{enumeratea}
We now proceed with the proof. We first need some notation. Define
\begin{equation}
\label{eqn:3156}
    |\clusterpi{ \max, \sss\geq i}{t}| = \sup_{j\geq i}|\clusterpi{\sss \geq j}{t}|, \qquad S_{\ell, \sss\geq i}^\pi(t) = \frac{1}{N(t)} \sum_{j=i+1}^{N(t)} |\clusterpi{\sss \geq j}{t}|^{\ell}.  
\end{equation}
Thus, the full susceptibility $S_{2}^\pi(t)$ satisfies $S_{2}^\pi(t) = S_{2, \sss\geq 1}^\pi(t)$. Recall our notation $w(u) = 2\pi^2 S_2^{\pi}(u) + 2\pi$, and define
    \eqn{
    \Lambda(t) = \int_0^t w(u) du.}
It will be convenient to occasionally switch between discrete and continuous time. In doing so, we will slightly abuse notation, and write the time index ``$t$'' for continuous and ``$n$'' for discrete time.

\begin{lem}[Coupling of the process beyond a stage with original process]
    \label{lem:coupling}
    Fix any $i \ge 1$. Recall that $\tau_i$ denotes the arrival time of vertex $i$ in continuous time. Conditionally on $\mathscr{F}_{\tau_i}$, there exists a coupling between the original percolated network process $(\clusterpi{}{t})_{t \ge 0}$, and another process with the same dynamics, but now started from time $\tau_i$, with the $i^{\rm th}$ vertex being the root, such that the presence of an edge between vertices $i_1 \ge i$ and $i_2 \ge i$ in the original process implies its presence in the coupled process. In particular, conditionally on $\mathscr{F}_{\tau_i}$,  the process $(|\clusterpi{\max, \sss\geq i}{t + \tau_i}|)_{t \ge 0}$ is pathwise stochastically dominated by $(|\clusterpi{\max}{t}|)_{t \ge 0}$.
\end{lem}
\begin{proof}
The coupling is explicitly constructed as follows. Let $ i \le i_1 < i_2$. In the original model, the probability of vertex $i_2$ attaching to $i_1$ is $\pi/i_2$. In the coupled process, this probability is set to $\pi/(i_2-i)$. To prescribe the coupling precisely, for every such $i_1, i_2$, sample a uniform random variable $U$ on $[0,1]$. If $U \le \pi/i_2$, the edge $(i_1, i_2)$ is placed in both the original and coupled processes. If $\pi/i_2 < U \le \pi/(i_2-i)$, the edge is placed in the coupled process but not the original process. Otherwise, the edge is not placed in either process. For $i_1 < i < i_2$, there is no edge $(i_1, i_2)$ in the coupled process, but this edge is present in the original process with probability $\pi/i_2$.
\end{proof}
We next state a simple bound on the moment generating function of the birth times of the Yule process:

\begin{lem}[Moment generating function of birth times in Yule processes]
\label{lem:yule-proc-bd}
    Let $\tau_i$ denote the time for the $i^{\rm th}$ arrival in a rate-one Yule process. Then, for any $\theta >0$, there exists a $C_\theta < \infty$ such that $\E[\e^{-\theta \tau_i}] \leq C_\theta i^{-\theta}$ for all $i\geq 1$.
\end{lem}
\begin{proof}
    By construction, $\tau_i \stackrel{d}{=} \sum_{j=1}^i \frac{E_j}{j}$ where $(E_j)_{j\geq 1}$ are iid exponential random variables with mean 1. Thus, 
    $$
    \mathbb{E}\big[\e^{-\theta \tau_i}\big] = \prod_{l=1}^i \frac{l}{l +\theta}=\e^{-\sum_{l=1}^i \log(1-\frac{\theta}{j})} \le C_\theta i^{-\theta},
    $$
    where the last line arises from $\log(1+x)\leq 1+x$. 
\end{proof}

\begin{lem}[A priori bounds on growth of maxima]
\label{lem:initial-bds}
    \begin{enumeratea}
        \item \label{it:aa} For any $\delta \in (0,1)$ and $\ell> 2/\delta$, there exists a  constant $C_{\ell, \delta} < \infty$ such that, for all $j\geq 1$ and $B>0$,
        \[
    \pr\bigg(\sup_{t\geq 0} \e^{-(\alpha(\pi)+\delta)t})|\clusterpi{\sss\geq j}{t}| \geq B\bigg) \leq \frac{C_{\ell, \delta}}{B^\ell}\frac{1}{ j^{(\alpha(\pi)+\delta)\ell -1}}.\]
    
    \item \label{it:bb} Fix $\eps\in (0,1)$ and integer $k\geq 1$. Then, for all $\ell > \max\set{\frac{4k}{\eps}, \frac{2}{\alpha(\pi)}}$, there exists a constant $C_{\ell, \eps}^\prime < \infty$ such that, for all $i \ge 1$,
    \[\pr\bigg(\sup_{t\geq \tau_i} \e^{-(\alpha(\pi) +\frac{\eps}{2k})(t-\tau_i)}|\clusterpi{\max, \sss\geq i}{t}| \geq i^{\eps}\bigg) \leq \frac{C_{\ell, \eps}^\prime}{i^{\eps \ell}}.  \]
    \end{enumeratea}
    
\end{lem}
\begin{proof}
    We start with the proof of \eqref{it:aa}. We switch to discrete time, and  note that 
    \begin{equation}
    \label{eqn:4853}
        \sup_{t\geq 0} e^{-(\alpha(\pi)+\delta)t})|\clusterpi{\sss\geq j}{t}| =\sup_{t\geq \tau_j}\e^{-(\alpha(\pi)+\delta)t}|\clusterpi{\sss\geq j}{t}| = \sup_{m\geq j}\e^{-(\alpha(\pi)+\delta)\tau_m}|\clusterpi{\sss\geq j}{\tau_m}|. 
    \end{equation} 
    For the rest of the argument, we abbreviate $\alpha=\alpha(\pi)$ to reduce notational overhead.  Next, note that 
    \begin{align}
\mathbb{P}\left( \sup_{n \ge j} n^{-(\alpha + \delta)} \left| \clusterpi{\sss\geq j}{n} \right| \ge B \right)
&\le \sum_{n=j}^\infty \mathbb{P}\left( \left| \clusterpi{\sss\geq j}{n} \right| \ge B n^{\alpha + \delta} \right) 
\le \sum_{n=j}^\infty \frac{ \mathbb{E}\!\left[ \left| \clusterpi{\sss\geq j}{n} \right|^{\ell} \right] }{ B^{\ell} n^{(\alpha + \delta) \ell} } \\
&\le \frac{C''_{\ell,\delta}}{B^{\ell}} \sum_{n=j}^\infty \frac{1}{n^{(\alpha + \delta)\ell}} \left( \frac{n}{j} \right)^{(\alpha+\frac{\delta}{2})\ell} 
= \frac{C''_{\ell,\delta}}{B^{\ell} j^{\ell(\alpha + \frac{\delta}{2})}}  \sum_{n=j}^\infty \frac{1}{n^{\delta \ell / 2}},\nn
\end{align}
for some constant $C''_{\ell,\delta}>0$,
where the first inequality in the last line follows by Lemma \ref{lem-higher-moment-bound} with $k=\ell$ and $ \beta=\alpha+\delta/2$. The assumption that $\ell\delta>2$, so that $\sum_{n=j}^\infty \frac{1}{n^{\delta \ell / 2}}  \le C j^{1-\delta\ell/2}$ for some constant $C>0$, combined with \eqref{eqn:4853}, completes the proof.
\smallskip

Let us next prove \eqref{it:bb} using the estimate in \eqref{it:aa}.   First note that 
\begin{align}
\mathbb{P}\!\left(\sup_{t \ge \tau_i} e^{-(\alpha + \frac{\varepsilon}{2k})(t - \tau_i)} \left| \clusterpi{\max, \sss\geq i}{t} \right| \ge i^{\varepsilon} \right)
&\le \mathbb{P}\!\left(\sup_{t \ge 0} \e^{-(\alpha + \frac{\varepsilon}{2k})t} \left| \clusterpi{ \max}{t} \right| \ge i^{\varepsilon} \right), \\
&\le \sum_{j=1}^\infty \mathbb{P}\!\left(\sup_{t \ge \tau_j} \e^{-(\alpha + \frac{\varepsilon}{2k})t} \left| \clusterpi{\sss\geq j}{t} \right| \ge i^{\varepsilon} \right),\nn
\end{align}
where we have used Lemma \ref{lem:coupling} for the first inequality, and the union bound for the second. Finally, we use part \eqref{it:aa} with $\delta = \eps/(2k), B=i^{\eps}$ for each term in the final series, and choosing $\ell > \max\set{\frac{4k}{\eps}, \frac{2}{\alpha}}$ so that the resultant series converges. This completes the proof of part (b).  
\end{proof}
We now have the technical estimates required to complete the proof of Theorem \ref{latecomplem}. Recall that, throughout the proof, we abbreviate $\alpha = \alpha(\pi)$. As described in the outline of the proof, the goal is to analyze the $k^{\rm th}$ powers of $|\clusterpi{\sss \geq i}{\cdot}|$ for an appropriate choice of $k$ on a collection of ``good events'' defined via appropriate stopping times. Let us now give the details of this argument.
\smallskip

Fix a constant $L\geq 1$. Next, choose $\eps \in (0,\min(\alpha, 1-2\alpha))$ (as in the setting of Lemma \ref{s3bd}) small enough, and $k \ge 1$ large enough, such that
\begin{equation}
    \label{eqn:k-bd}
    \alpha\left(\frac{k}{2}-1\right) -\eps k \ge 2.
\end{equation}
Define the stopping times 
\begin{equation}\label{eq:sigmaL}
\sigma_L := \inf\left\{ t \ge \tau_i \colon  
\int_0^t \e^{-(\alpha -\eps)u} \left( 1 + \csusceptibilitypi{u}{3} \right) \, du > L 
\ \text{ or } \
2\pi^2 \int_0^t \left| \susceptibilitypi{\infty}{2} - \csusceptibilitypi{u}{2} \right| \, du > L 
\right\},
\end{equation}
and, for $i\geq 1$, 
\begin{equation}
\label{eq:gvrL-i}
\varrho_L^{(i)} := \inf\left\{ t \ge \tau_i \colon  
\left| \clusterpi{\max, \sss\geq i}{t} \right| 
> i^{\varepsilon} \e^{(\alpha + \frac{\varepsilon}{2k})(t - \tau_i)} \right\} 
\wedge \sigma_L.
\end{equation}

Fix $i\geq 1$ and, with $k$ as in \eqref{eqn:k-bd}, consider $|\clusterpi{\sss \geq i}{t}|^k$. As in \eqref{dynamics-connected-component-continuous} and Lemma \ref{s3bd}, applying the generator $\cL$ of the underlying Markov process,
    \begin{align}\label{genk}
        \mathcal{L} |\clusterpi{\sss\geq i}{t}|^k &= \sum_{j=i+1}^{N(t)}\left((|\clusterpi{\sss\geq i}{t}|+ |\clusterpi{\sss \ge j}{t}| + 1)^k - |\clusterpi{\sss\geq i}{t}|^k\right)\frac{2\pi^2|\clusterpi{\sss\geq i}{t}||\clusterpi{\sss \ge j}{t}|}{N(t)}\notag\\
        &\quad + 2\pi(1-\pi)\left((|\clusterpi{\sss\geq i}{t}|+ 1)^k - |\clusterpi{\sss\geq i}{t}|^k\right)|\clusterpi{\sss\geq i}{t}|\notag\\
        &\quad + \left((|\clusterpi{\sss\geq i}{t}|+ 1)^k - |\clusterpi{\sss\geq i}{t}|^k\right) \frac{\pi^2|\clusterpi{\sss\geq i}{t}|^2}{N(t)}\notag\\
        & = 2\pi^2\sum_{j=i+1}^{N(t)}\sum_{l=0}^{k-1}{k \choose l} |\clusterpi{\sss\geq i}{t}|^l(|\clusterpi{\sss \ge j}{t}| + 1)^{k-l}\frac{|\clusterpi{\sss\geq i}{t}||\clusterpi{\sss \ge j}{t}|}{N(t)}\notag\\
        &\quad + 2\pi(1-\pi)\sum_{l=0}^{k-1}{k \choose l} |\clusterpi{\sss\geq i}{t}|^l|\clusterpi{\sss\geq i}{t}| + \pi^2 \sum_{l=0}^{k-1}{k \choose l} |\clusterpi{\sss\geq i}{t}|^l \frac{|\clusterpi{\sss\geq i}{t}|^2}{N(t)}\notag\\
        &= kw(t) |\clusterpi{\sss\geq i}{t}|^k + R_i^{\sss(k)}(t),
    \end{align}
    where we recall that $w(t) = 2\pi^2 \csusceptibilitypi{t}{2} + 2\pi$, and
    \begin{align*}
        0 &\le R_i^{\sss(k)}(t) \le A_1 \sum_{l=3}^{k+1}(S_{l,\sss\geq i}^\pi(t)+1)|\clusterpi{\sss\geq i}{t}|^{k-l+2} + A_2\sum_{l=3}^{k+1}|\clusterpi{\sss\geq i}{t}|^{k-l+2}\\
        & \leq A(1+\csusceptibilitypi{t}{3})[|\clusterpi{\max, \sss\geq i}{t}|]^{k-1},
    \end{align*}
    for positive constants $A_1,A_2, A$ depending only on $k$. 

Recall that $\Lambda(u)=\int_0^u w(d)du$, and note that
 \begin{align}
 \int_{\tau_i}^{\varrho_L^{\sss (i)}} & \e^{-k \Lambda(u)} R_i^{\sss (k)}(u)\, du
\le 
A \int_{\tau_i}^{\varrho_L^{\sss (i)}} 
\e^{-k \Lambda(u)} \big(1 + \csusceptibilitypi{u}{3}\big)
\big|\clusterpi{\max,\sss\geq i}{u}\big|^{k-1} \, du \nonumber \\[1em]
 &\leq A\, \e^{-(k-1)\!\big(\alpha + \frac{\varepsilon}{k-1}\big)\!\tau_i}
\int_{\tau_i}^{\varrho_L^{\sss (i)}}
\e^{-(\alpha - \varepsilon)u} \big(1 + \csusceptibilitypi{u}{3}\big)
\bigg(
\e^{-(k-1)\!\big(\alpha + \frac{\varepsilon}{k-1}\big)\!(u-\tau_i)}
\big|\clusterpi{\max, \sss\geq i}{u}\big|^{k-1}
\bigg) \nonumber \\[4pt]
&\hspace{10em}\times 
\exp\!\bigg(
2\pi^2 k \int_0^u \!\big(\susceptibilitypi{\infty}{2}
- \csusceptibilitypi{v}{2}\big)\,dv
\bigg) du \nonumber \\
&\leq AL\e^{kL}i^{(k-1)\eps}\e^{-(k-1)\alpha \tau_i}, \label{eqn:one}
\end{align}
where the last bound follows using the definition of the stopping time $\varrho_L^{\sss(i)}$ in \eqref{eq:gvrL-i}.  
\smallskip

Define the process $Y_{\sss\geq i}^{\sss(k)}(t) := \e^{-k\Lambda(t) } [|\clusterpi{\sss\geq i}{t}|]^k$, and consider the semi-martingale decomposition into a martingale term and a predictable drift process as 
    \eqn{
    dY_{\sss\geq i}^{\sss(k)}(t) = d\cM_{\sss\geq i}^{\sss(k)}(t) + b_{\sss\geq i}^{\sss(k)}(t) dt, \qquad b_{\sss\geq i}^{\sss(k)}(t) \leq \e^{-k\Lambda(t)} R_i^{\sss(k)}(t).
    }
Further, write $\cB_{\sss\geq i}(t) = \int_0^t  b_{\sss\geq i}^{\sss(k)}(u) du$ and note that \eqref{eqn:one} implies 
\begin{equation}
    \cB_{\sss\geq i}(t\wedge \varrho_L^{\sss(i)}) \leq  ALe^{kL}i^{(k-1)\eps}\e^{-(k-1)\alpha \tau_i} \quad \text{for all } \ t \ge 0. 
\end{equation}
Next, note that the quadratic variation of the martingale is given by 
\begin{align}\label{gen2k}
\frac{d \left\langle \cM_{
\sss\geq i}^{(k)} \right\rangle(t)}{dt}
&= \e^{-2k \Lambda(t)} 
\lim_{h \to 0} \frac{1}{h} 
\mathbb{E}\!\left[
\left(|\clusterpi{\sss\geq i}{t+h}|^k - |\clusterpi{\sss\geq i}{t}|^k \right)^{\!2}
\Big| \mathcal{F}_t\right]\notag\\
&= \e^{-2k \Lambda(t)}\sum_{j=i+1}^{N(t)}\left((|\clusterpi{\sss\geq i}{t}|+ |\clusterpi{\sss \ge j}{t}| + 1)^k - |\clusterpi{\sss\geq i}{t}|^k\right)^2\frac{2\pi^2|\clusterpi{\sss\geq i}{t}||\clusterpi{\sss \ge j}{t}|}{N(t)}\notag\\
        &\quad + 2\pi(1-\pi)\e^{-2k \Lambda(t)}\left((|\clusterpi{\sss\geq i}{t}|+ 1)^k - |\clusterpi{\sss\geq i}{t}|^k\right)^2|\clusterpi{\sss\geq i}{t}|\notag\\
        &\quad + \e^{-2k \Lambda(t)}\left((|\clusterpi{\sss\geq i}{t}|+ 1)^k - |\clusterpi{\sss\geq i}{t}|^k\right)^2 \frac{\pi^2|\clusterpi{\sss\geq i}{t}|^2}{N(t)}.
\end{align}
The first term in \eqref{gen2k} can be bounded, similarly as in \eqref{genk}, by
\begin{align*}
    &\e^{-2k \Lambda(t)}\sum_{j=i+1}^{N(t)}\left((|\clusterpi{\sss\geq i}{t}|+ |\clusterpi{\sss \ge j}{t}| + 1)^k - |\clusterpi{\sss\geq i}{t}|^k\right)^2\frac{2\pi^2|\clusterpi{\sss\geq i}{t}||\clusterpi{\sss \ge j}{t}|}{N(t)}\\
    & = \e^{-2k \Lambda(t)}\sum_{j=i+1}^{N(t)}\left[\sum_{l=0}^{k-1}{k \choose l} |\clusterpi{\sss\geq i}{t}|^l(|\clusterpi{\sss \ge j}{t}| + 1)^{k-l}\right]^2\frac{2\pi^2|\clusterpi{\sss\geq i}{t}||\clusterpi{\sss \ge j}{t}|}{N(t)}\\
    &\le 2\pi^2 k \e^{-2k \Lambda(t)}\sum_{j=i+1}^{N(t)}\sum_{l=0}^{k-1}{k \choose l}^2 |\clusterpi{\sss\geq i}{t}|^{2l}(|\clusterpi{\sss \ge j}{t}| + 1)^{2k-2l}\frac{|\clusterpi{\sss\geq i}{t}||\clusterpi{\sss \ge j}{t}|}{N(t)}\\
    &\le A_1' \e^{-2k \Lambda(t)} \sum_{l=0}^{k-1}|\clusterpi{\sss\geq i}{t}|^{2l+1}\left(1+S_{2k-2l+1, \geq i}^\pi(t)\right)\\
    &\le A_2' \e^{-2k \Lambda(t)}(1+\csusceptibilitypi{t}{3})|\clusterpi{\max, \sss\geq i}{t}|^{2k-1},
\end{align*}
for constants $A_1',A_2'$ depending only on $k$, where the first inequality follows from the Cauchy-Schwarz inequality and the next few bounds are obtained similarly as in \eqref{genk}.
\smallskip

Similarly, the sum of the last two terms in \eqref{gen2k} is bounded by
\eqn{
A_3'\e^{-2k \Lambda(t)}|\clusterpi{\max, \sss\geq i}{t}|^{2k-1}.
}
Combining the above bounds, we obtain a constant $A^\prime=A^\prime(k)< \infty$ such that
    \eqn{
    \frac{d \left\langle \cM_{\sss\geq i}^{(k)} \right\rangle(t)}{dt} \leq A^\prime \e^{-2k\Lambda(t)}(1+\csusceptibilitypi{t}{3})|\clusterpi{\max, \sss\geq i}{t}|^{2k-1}.
    }
Consequently, for the stopped martingale $\hat{\cM}_{\sss\geq i}^{\sss(k)}(\cdot) = \cM_{\sss\geq i}^{\sss(k)}(\cdot\wedge \varrho_L^{\sss(i)})$, we have another constant $A^*$ such that the quadratic variation process satisfies, for all $t \ge 0$, 
\begin{align}
    \left\langle \hat{\cM}_{\sss\geq i}^{(k)} \right\rangle(t) &\leq A^\prime \int_{\tau_i}^{\varrho_L^{\sss(i)}} \e^{-2k \Lambda(u)} \big(1 + \csusceptibilitypi{u}{3}\big)
\big|\clusterpi{\max,\sss\geq i}{u}\big|^{2k-1} \, du \nonumber\\
& \leq  A^*L\e^{2kL}i^{(2k-1)\eps}\e^{-(2k-1)\alpha \tau_i}\label{eqn:two},
\end{align}
where the last inequality follows the same reasoning as in \eqref{eqn:one}.
\smallskip

Fix $\eta \in (0,1)$. Consider the bounds in \eqref{eqn:one}
and \eqref{eqn:two} and define the constant $\bar A = \max(A, \sqrt{A^*})$. Define the `good' event $\cG^{\sss(i)}$ by
    \eqn{\cG^{\sss(i)}:= \set{\bar A L\e^{kL}i^{k\eps}e^{-(k-1)\alpha \tau_i} \leq \frac{\eta^k}{2i^{\alpha k/2}}}.
    }
Note that by \eqref{eqn:two}, on the event $\cG^{\sss(i)}$, the quadratic variation of the stopped process satisfies
\begin{equation}
    \label{eqn:44431}
     \left\langle \hat{\cM}_{
     \sss\geq i}^{(k)} \right\rangle(\infty) \leq \frac{\eta^{2k}}{4i^{\alpha k}}.
\end{equation}
Now note that, by \eqref{eqn:one},
\begin{align*}
\pr\left(\sup_{t\geq \tau_i}e^{-\Lambda(t)}|\clusterpi{\sss\geq i}{t\wedge \varrho_L^{\sss(i)}}| \geq \eta\right) &= \pr\left(\sup_{t\geq \tau_i}Y_{\sss\geq i}(t\wedge \varrho_L^{\sss(i)}) \geq \eta^k\right) \\
&\leq \pr\left([\cG^{\sss(i)}]^c\right) + \pr\left(\cG^{\sss(i)} \cap \set{\sup_{t\geq \tau_i}|\hat{\cM}_{\sss\geq i}^{\sss(k)}| > \eta^{k}\Big(1-\frac{1}{2i^{\alpha k/2}}\Big)}\right). 
\label{eqn:rhol-bound}  
\end{align*}
For the second term, by Doob's $L^2$ maximal inequality, and using \eqref{eqn:44431}, there is a constant $C$ not depending on $i$ such that
\begin{equation}
    \label{eqn:4449}
    \pr\left(\cG^{\sss(i)} \cap \set{\sup_{t\geq \tau_i}|\hat{\cM}_{\sss\geq i}^{\sss(k)}| > \eta^{k}\Big(1-\frac{1}{2i^{\alpha k/2}}\Big)}\right) \leq \frac{C}{i^{\alpha k}}.
\end{equation}
For the first term, by Markov's inequality and then using Lemma \ref{lem:yule-proc-bd},
\begin{equation}
\label{eqn:pls-check}
\pr([\cG^{\sss(i)}]^c) \leq    \frac{2\, i^{\frac{\alpha k}{2} + \varepsilon k}}{\eta^k} 
\cdot 
\overline{A} L\, e^{kL}\, 
\mathbb{E}\!\left[ e^{-(k-1)\alpha \tau_i} \right]\leq \frac{\hat{A}_{k,L}}{\eta^k}\frac{1}{i^{\left(\frac{k}{2}-1\right)\alpha - \eps k}}
\end{equation}
for some constant $\hat{A}_{k,L}$ depending only on $k,L$.
Now \eqref{eqn:4449} and \eqref{eqn:pls-check}, combined with the choice of $k$ and $\eps$ in \eqref{eqn:k-bd}, give that, for some constant $C = C(k,L)< \infty$, 
\begin{equation}
\label{eqn:tau-i-bd}
    \pr\Big(\sup_{t\geq \tau_i}\e^{-\Lambda(t)}|\clusterpi{\sss\geq i}{t\wedge \varrho_L^{\sss(i)}}| \geq \eta\Big) \leq C\left(1+\frac{1}{\eta^k}\right)\frac{1}{i^2}.
\end{equation}
Next note that 
\begin{align}
    \mathbb{P}\!\left(\sup_{t \ge \tau_i} \e^{-\Lambda(t)} \big|\clusterpi{\sss\geq i}{t \wedge \sigma_L}\big| \ge \eta \right)
\le 
\mathbb{P}&\!\left(\sup_{t \ge \tau_i} \e^{-\Lambda(t)} \big|\clusterpi{\sss\geq i}{t \wedge \varrho_L^{\sss (i)}}\big| \ge \eta \right)\nonumber\\
&+ 
\mathbb{P}\!\left(\sup_{t \ge \tau_i} \e^{-(\alpha + \frac{\varepsilon}{2k})(t - \tau_i)} \big|\clusterpi{\max, \sss\geq i}{t}\big| \ge i^{\eps} \right). \label{eqn:4326}
\end{align}
For the second term in \eqref{eqn:4326}, by Lemma \ref{lem:initial-bds} \eqref{it:bb} with $\ell > \max\set{\frac{4k}{\eps}, \frac{2}{\alpha}}$, and noting that this implies $\ell \eps >2$, there exists an appropriate constant $C<\infty$ such that
\begin{equation}
    \label{eqn:4324}
    \mathbb{P}\!\left(\sup_{t \ge \tau_i} \e^{-(\alpha + \frac{\varepsilon}{2k})(t - \tau_i)} \big|\clusterpi{\max, \sss\geq i}{t}\big| \ge i^{\eps} \right) \leq \frac{C}{i^2}. 
\end{equation}
Combining this with \eqref{eqn:tau-i-bd} for the first term in \eqref{eqn:4326}, and using these bounds in \eqref{eqn:4326}, gives for a constant $C^\prime = C^\prime(k,L)<\infty$ and not depending on $i$, 
\begin{equation}
     \mathbb{P}\!\left(\sup_{t \ge \tau_i} e^{-\Lambda(t)} \big|\clusterpi{\sss\geq i}{t \wedge \sigma_L}\big| \ge \eta \right) \leq C^\prime\left(1+\frac{1}{\eta^k}\right)\frac{1}{i^2}. 
\end{equation}
Thus, for any fixed $M$, by the union bound, 
\begin{equation}
\label{eqn:5731}
    \mathbb{P}\!\left(
\sup_{i \ge M} \sup_{t \ge \tau_i} 
\e^{-\Lambda(t)} \big|\clusterpi{\sss\geq i}{t \wedge \sigma_L}\big| \ge \eta
\right)
\le 
A_{k,L}^\circ 
\left( 1 + \frac{1}{\eta^k} \right) 
\frac{1}{M},
\end{equation}
for a suitable constant $A_{k,L}^\circ$.
In particular, 
\[ \mathbb{P}\!\left(
\sup_{i \ge M} \sup_{t \ge \tau_i} 
\e^{-\Lambda(t)} \big|\clusterpi{\sss\geq i}{t}\big| \ge \eta
\right)
\le 
A_{k,L}^\circ 
\left( 1 + \frac{1}{\eta^k} \right) 
\frac{1}{M} +\pr(\sigma_L< \infty).\]
Thus, for every $L>0$, 
\[\limsup_{M\to\infty} \mathbb{P}\!\left(
\sup_{i \ge M} \sup_{t \ge \tau_i} 
\e^{-\Lambda(t)} \big|\clusterpi{\sss\geq i}{t}\big| \ge \eta
\right) \leq \pr(\sigma_L <\infty). \]
Using \eqref{ctS} and Lemma \ref{s3bd}, $\limsup_{L\to\infty} \pr(\sigma_L < \infty) = 0$. In particular, 
\begin{equation}\label{Lambd}
  \limsup_{M\to\infty} \mathbb{P}\!\left(
\sup_{i \ge M} \sup_{t \ge \tau_i} 
\e^{-\Lambda(t)} \big|\clusterpi{\sss\geq i}{t}\big| \ge \eta
\right) =0.  
\end{equation}
Finally, 
\begin{align*}
\sup_{i \ge M} \sup_{t \ge \tau_i} 
e^{-\alpha t} \big|\clusterpi{\sss\geq i}{t}\big| &= \sup_{i \ge M} \sup_{t \ge \tau_i} 
e^{-\Lambda(t)} \big|\clusterpi{\sss\geq i}{t}\big| e^{\Lambda(t) - \alpha t}\\ 
&\le e^{\int_0^\infty|\alpha - w(u)|du}\sup_{i \ge M} \sup_{t \ge \tau_i} e^{-\Lambda(t)} \big|\clusterpi{\sss\geq i}{t}\big|.
\end{align*}
Thus, by \eqref{ctS} and \eqref{Lambd}, we conclude
\eqn{
\limsup_{M\to\infty} \mathbb{P}\!\left(
\sup_{i \ge M} \sup_{t \ge \tau_i} 
\e^{-\alpha t} \big|\clusterpi{\sss\geq i}{t}\big| \ge \eta
\right) =0.
}
This completes the proof of Theorem \ref{latecomplem}. 
\end{proof}

\subsection{Weak persistence and almost sure convergence of maximal component size}
\label{sec-weak-persistence}
We first prove the weak persistence result: Given any $\vep>0$, there exists $K(\vep)$ such that 
\begin{equation}\label{wp}
    \pr\left(\clusterpi{\max}{n} = \clusterpi{i}{n} \text{ for some } i> K(\vep) \text{ for infinitely many } n\right) < \vep.
\end{equation}
Fix any $\vep>0$. 
By Theorem \ref{latecomplem}, for any $\eta>0$, we can choose $K(\vep,\eta) \in \mathbb{N}$ such that
\begin{equation}\label{per1}
    \mathbb{P}\left(\max_{i>K(\vep,\eta)}\sup_{t \ge 0}  \e^{-\alpha(\pi)t}|\clusterpi{\sss \geq i}{t}| > \eta\right) < \vep/2.
\end{equation}
By Theorem \ref{thm-max-comp}(a), we know that, for any $i \in \mathbb{N}$,
\begin{equation*}
    n^{-\alpha(\pi)} |\clusterpi{i}{n}| \convas \zeta_i > 0.
\end{equation*}
In particular, choosing $\eta = \eta(\vep)>0$ small enough and moving to continuous time, we can obtain $J(\vep) \in \mathbb{N}$ such that
    \begin{equation}\label{per2}
    \pr\left(\inf_{t \ge \tau_{J(\vep)}} \e^{-\alpha(\pi)t}|\clusterpi{1}{t}| \le \eta\right) < \vep/2.
    \end{equation}
Thus, writing $K(\vep) = K(\vep, \eta(\vep))$ and interchanging between continuous- and discrete-time dynamics,
\begin{align*}
&\pr\left(\clusterpi{\max}{n} = \clusterpi{i}{n} \text{ for some } i> K(\vep) \text{ for infinitely many } n\right)\\
    &\quad \le \pr\left(\exists n > J(\vep) \text{ such that } \max_{i > K(\vep)}|\clusterpi{\sss \ge i}{n}| = |\clusterpi{\max}{n}|\right)\\
    &\quad  = \pr\left(\exists t > \tau_{J(\vep)} \text{ such that } \max_{i > K(\vep)}|\clusterpi{\sss \ge i}{t}| = |\clusterpi{\max}{t}|\right)\\
    &\quad \le \pr\left(\inf_{t \ge \tau_{J(\vep)}} \e^{-\alpha(\pi)t}|\clusterpi{1}{t}| \le \eta\right) + \mathbb{P}\left(\max_{i>K(\vep)}\sup_{t \ge 0}  \e^{-\alpha(\pi)t}|\clusterpi{\sss \geq i}{t}| > \eta\right) < \vep,
\end{align*}
proving \eqref{wp}.

Finally, we prove the almost sure convergence of the maximal component. Note that, for any $\vep>0$, on the event 
$$
\mathcal{E}:= \{\clusterpi{\max}{n} = \clusterpi{i}{n} \text{ for some } i\le K(\vep) \text{ for all sufficiently large } n\},
$$
we have
\begin{align*}
\lim_{n \to \infty}n^{-\alpha(\pi)}|\clusterpi{\max}{n}| &= \lim_{n \to \infty}\max_{i\le K(\vep)} n^{-\alpha(\pi)}|\clusterpi{i}{n}|\\
&= \max_{i\le K(\vep)}\lim_{n \to \infty} n^{-\alpha(\pi)}|\clusterpi{i}{n}| = \max_{i\le K(\vep)} \zeta_i \leq \max_{i\ge 1} \zeta_i.
\end{align*}
As $\vep>0$ is arbitrary and $\pr(\mathcal{E}) \ge 1-\vep$, we conclude
$$
\pr\left(\limsup_{n \to \infty}n^{-\alpha(\pi)}|\clusterpi{\max}{n}| \leq \max_{i\ge 1} \zeta_i\right)=1.
$$
Also, for every $i\geq 1$ fixed,
    \[
    n^{-\alpha(\pi)}|\clusterpi{\max}{n}|\geq n^{-\alpha(\pi)}|\clusterpi{i}{n}|
    \convas \zeta_i,
    \]
so that
    \[
    \liminf_{n\rightarrow \infty} n^{-\alpha(\pi)}|\clusterpi{\max}{n}|\geq \max_{i\geq 1}\zeta_i.
    \]
This completes the proof of Theorem \ref{thm-max-comp}(b).
\qed
\bigskip

\paragraph{\bf Acknowledgement.}
Banerjee was partially supported by the NSF CAREER award DMS-2141621. Bhamidi was partially supported by NSF DMS-2113662, DMS-2413928, and DMS-2434559. Banerjee and Bhamidi were partially funded by NSF RTG grant DMS-2134107. The work of van der Hofstad and Ray is supported in part by the Netherlands Organisation for Scientific Research (NWO) through the Gravitation {\sc Networks} grant 024.002.003. Part of
this material is based upon work supported by the National Science Foundation under Grant
No. DMS-1928930, while Banerjee, Bhamidi, van der Hofstad and Ray were in residence at the Simons Laufer
Mathematical Sciences Institute in Berkeley, California, during the Spring 2025 semester.  The work of Ray is supported in part by the European Union's Horizon 2020 research and innovation programme under the Marie Sk\l{}odowska-Curie grant agreement no.\ 945045.

\bibliographystyle{plain}
\bibliography{bibBooks}

@PREAMBLE{ {\providecommand{\noopsort}[1]{}} }

@STRING{ JSP = {J. Statist. Phys.}}

@preamble{"\def\cprime{$'$} "}

@article {AddBroGol12,
    AUTHOR = {Addario-Berry, L. and Broutin, N. and Goldschmidt, C.},
     TITLE = {The continuum limit of critical random graphs},
   JOURNAL = {Probab. Theory Related Fields},
  FJOURNAL = {Probability Theory and Related Fields},
    VOLUME = {{\bf 152}},
      YEAR = {2012},
    NUMBER = {3-4},
     PAGES = {367--406},
      ISSN = {0178-8051},
     CODEN = {PTRFEU},
   MRCLASS = {05C80 (05C12 60C05 60F17)},
  MRNUMBER = {2892951},
MRREVIEWER = {Jonathan Cutler},
       DOI = {10.1007/s00440-010-0325-4},
       URL = {http://dx.doi.org.dianus.libr.tue.nl/10.1007/s00440-010-0325-4},
}

@article {AddBroGolMie17,
    AUTHOR = {Addario-Berry, L. and Broutin, N. and Goldschmidt,
              C. and Miermont, G.},
     TITLE = {The scaling limit of the minimum spanning tree of the complete
              graph},
   JOURNAL = {Ann. Probab.},
  FJOURNAL = {The Annals of Probability},
    VOLUME = {{\bf 45}},
      YEAR = {2017},
    NUMBER = {5},
     PAGES = {3075--3144},
      ISSN = {0091-1798,2168-894X},
   MRCLASS = {60C05 (60F05)},
  MRNUMBER = {3706739},
MRREVIEWER = {Leonid\ V.\ Bogachev},
       DOI = {10.1214/16-AOP1132},
       URL = {https://doi.org/10.1214/16-AOP1132},
}

@article{AizNew84,
    author={Aizenman, M. and Newman, C.M.},
    title={Tree graph inequalities and critical behavior in percolation models},
    journal= JSP,
    volume={{\bf 36}},
    year={1984},
    pages={107--143}
}

@article{Aldo97,
    author={Aldous, D.},
    title={Brownian excursions, critical random graphs and the multiplicative coalescent},
    journal={Ann. Probab.},
    volume={{\bf 25}},
    year={1997},
    pages={812--854},
    number={2},
    fjournal={The Annals of Probability},
    issn={0091-1798},
    mrclass={60C05 (60J50)},
    mrnumber={MR1434128 (98d:60019)}
}

@article {Aldo99,
    AUTHOR = {Aldous, D.},
     TITLE = {Deterministic and stochastic models for coalescence
              (aggregation and coagulation): a review of the mean-field
              theory for probabilists},
   JOURNAL = {Bernoulli},
  FJOURNAL = {Bernoulli. Official Journal of the Bernoulli Society for
              Mathematical Statistics and Probability},
    VOLUME = {{\bf 5}},
      YEAR = {1999},
    NUMBER = {1},
     PAGES = {3--48},
      ISSN = {1350-7265},
   MRCLASS = {60K35 (82C31)},
  MRNUMBER = {MR1673235 (2001c:60153)},
MRREVIEWER = {Jochen Geiger},
       DOI = {10.2307/3318611},
       URL = {http://dx.doi.org/10.2307/3318611},
}

@article {BanBha21,
    AUTHOR = {Banerjee, S. and Bhamidi, S.},
     TITLE = {Persistence of hubs in growing random networks},
   JOURNAL = {Probab. Theory Related Fields},
  FJOURNAL = {Probability Theory and Related Fields},
    VOLUME = {{\bf 180}},
      YEAR = {2021},
    NUMBER = {3-4},
     PAGES = {891--953},
      ISSN = {0178-8051,1432-2064},
   MRCLASS = {60C05 (60J28 60J85)},
  MRNUMBER = {4288334},
       DOI = {10.1007/s00440-021-01066-0},
       URL = {https://doi-org.dianus.libr.tue.nl/10.1007/s00440-021-01066-0},
}

@article {BanBha22,
    AUTHOR = {Banerjee, S. and Bhamidi, S.},
     TITLE = {Root finding algorithms and persistence of {J}ordan centrality
              in growing random trees},
   JOURNAL = {Ann. Appl. Probab.},
  FJOURNAL = {The Annals of Applied Probability},
    VOLUME = {{\bf 32}},
      YEAR = {2022},
    NUMBER = {3},
     PAGES = {2180--2210},
      ISSN = {1050-5164,2168-8737},
   MRCLASS = {05C80 (05C85 60J80 60J85)},
  MRNUMBER = {4430011},
MRREVIEWER = {Nicolas\ Curien},
       DOI = {10.1214/21-aap1731},
       URL = {https://doi-org.dianus.libr.tue.nl/10.1214/21-aap1731},
}

@unpublished{BanBhaHazHofRay26,
    author={Banerjee, Sayan and Bhamidi, Shankar and Hazra, Rajat and van der Hofstad, Remco and Ray, Rounak},
    title={Nonequilibrium coagulation models and percolation on Preferential attachment},
    journal= {Work in progress},
    year={2026}
}

@unpublished{BhaBroSenWan14,
  title={Scaling limits of random graph models at criticality: Universality and the basin of attraction of the {E}rd{\H{o}}s-{R}{\'e}nyi random graph},
  author={Bhamidi, S. and Broutin, N. and Sen, S. and Wang, X.},
  note={arXiv:1411.3417 [math.PR]},
  year={2014}
}

@article {BasBhaBroSenWan25,
    AUTHOR = {Baslingker, J. and Bhamidi, S. and Broutin,
              N. and Sen, S. and Wang, X.},
     TITLE = {Scaling limits and universality: critical percolation on
              weighted graphs converging to an {$L^3$} graphon},
   JOURNAL = {Trans. Amer. Math. Soc.},
  FJOURNAL = {Transactions of the American Mathematical Society},
    VOLUME = {{\bf 378}},
      YEAR = {2025},
    NUMBER = {5},
     PAGES = {3157--3228},
      ISSN = {0002-9947,1088-6850},
   MRCLASS = {60C05 (05C80 60K35)},
  MRNUMBER = {4891374},
       DOI = {10.1090/tran/9267},
       URL = {https://doi.org/10.1090/tran/9267},
}

@article {BerTou08,
    AUTHOR = {Bercu, B. and Touati, A.},
     TITLE = {Exponential inequalities for self-normalized martingales with
              applications},
   JOURNAL = {Ann. Appl. Probab.},
  FJOURNAL = {The Annals of Applied Probability},
    VOLUME = {{\bf 18}},
      YEAR = {2008},
    NUMBER = {5},
     PAGES = {1848--1869},
      ISSN = {1050-5164,2168-8737},
   MRCLASS = {60E15 (60G15 60G42 60J80)},
  MRNUMBER = {2462551},
       DOI = {10.1214/07-AAP506},
       URL = {https://doi-org.dianus.libr.tue.nl/10.1214/07-AAP506},
}

@article {BerBorChaSab14,
    AUTHOR = {Berger, N. and Borgs, C. and Chayes, J. and
              Saberi, A.},
     TITLE = {Asymptotic behavior and distributional limits of preferential
              attachment graphs},
   JOURNAL = {Ann. Probab.},
  FJOURNAL = {The Annals of Probability},
    VOLUME = {{\bf 42}},
      YEAR = {2014},
    NUMBER = {1},
     PAGES = {1--40},
      ISSN = {0091-1798},
   MRCLASS = {Preliminary Data},
  MRNUMBER = {3161480},
       DOI = {10.1214/12-AOP755},
       URL = {http://dx.doi.org/10.1214/12-AOP755},
}

@book {Bert06,
    AUTHOR = {Bertoin, J.},
     TITLE = {Random fragmentation and coagulation processes},
    SERIES = {Cambridge Studies in Advanced Mathematics},
    VOLUME = {{\bf 102}},
 PUBLISHER = {Cambridge University Press},
   ADDRESS = {Cambridge},
      YEAR = {2006},
     PAGES = {viii+280},
      ISBN = {978-0-521-86728-3; 0-521-86728-2},
   MRCLASS = {60-02 (60G09 60J27 82C31)},
  MRNUMBER = {MR2253162 (2007k:60004)},
MRREVIEWER = {Nicolas Fournier},
}

@article {BhaSen24,
    AUTHOR = {Bhamidi, S. and Sen, S.},
     TITLE = {Geometry of the minimal spanning tree in the heavy-tailed
              regime: new universality classes},
   JOURNAL = {Probab. Theory Related Fields},
  FJOURNAL = {Probability Theory and Related Fields},
    VOLUME = {{\bf 188}},
      YEAR = {2024},
    NUMBER = {3-4},
     PAGES = {729--804},
      ISSN = {0178-8051,1432-2064},
   MRCLASS = {60C05 (05C80)},
  MRNUMBER = {4716340},
       DOI = {10.1007/s00440-024-01259-3},
       URL = {https://doi-org.dianus.libr.tue.nl/10.1007/s00440-024-01259-3},
}

@book{Boll01,
    author={Bollob{\'a}s, B.},
    title={Random graphs},
    publisher={Cambridge University Press},
    year={2001},
    volume={{\bf 73}},
    series={Cambridge Studies in Advanced Mathematics},
    edition={Second},
    pages={xviii+498},
    isbn={0-521-80920-7; 0-521-79722-5},
    mrclass={05C80 (60C05)},
    mrnumber={MR1864966 (2002j:05132)}
}

@article{BolJanRio05,
    author={Bollob{\'a}s, B. and Janson, S. and Riordan, O.},
    title={The phase transition in the uniformly grown random graph has infinite order},
    journal={Random Structures Algorithms},
    volume={{\bf 26}},
    year={2005},
    pages={1--36},
    number={1-2},
    fjournal={Random Structures \& Algorithms},
    issn={1042-9832},
    mrclass={05C80 (60C05 90B15)},
    mrnumber={MR2116573 (2006a:05148)},
    mrreviewer={Michael Krivelevich}
}

@article {BolJanRio07,
    AUTHOR = {Bollob{\'a}s, B. and Janson, S. and Riordan, O.},
     TITLE = {The phase transition in inhomogeneous random graphs},
   JOURNAL = {Random Structures Algorithms},
  FJOURNAL = {Random Structures \& Algorithms},
    VOLUME = {{\bf 31}},
      YEAR = {2007},
    NUMBER = {1},
     PAGES = {3--122},
      ISSN = {1042-9832},
   MRCLASS = {05C80 (60C05 68R10)},
  MRNUMBER = {MR2337396},
}

@article {BolRio05,
    AUTHOR = {Bollob\'as, B. and Riordan, O.},
     TITLE = {Slow emergence of the giant component in the growing {$m$}-out
              graph},
   JOURNAL = {Random Structures Algorithms},
  FJOURNAL = {Random Structures \& Algorithms},
    VOLUME = {{\bf 27}},
      YEAR = {2005},
    NUMBER = {1},
     PAGES = {1--24},
      ISSN = {1042-9832,1098-2418},
   MRCLASS = {05C80 (60C05)},
  MRNUMBER = {2149293},
MRREVIEWER = {Andrzej\ Ruci\'nski},
       DOI = {10.1002/rsa.20060},
       URL = {https://doi-org.dianus.libr.tue.nl/10.1002/rsa.20060},
}

@article {BriCalLug23,
    AUTHOR = {Briend, S. and Calvillo, F. and Lugosi, G.},
     TITLE = {Archaeology of random recursive dags and {C}ooper-{F}rieze
              random networks},
   JOURNAL = {Combin. Probab. Comput.},
  FJOURNAL = {Combinatorics, Probability and Computing},
    VOLUME = {{\bf 32}},
      YEAR = {2023},
    NUMBER = {6},
     PAGES = {859--873},
      ISSN = {0963-5483,1469-2163},
   MRCLASS = {05C80 (60C05)},
  MRNUMBER = {4653727},
MRREVIEWER = {David\ B.\ Penman},
       DOI = {10.1017/s0963548323000184},
       URL = {https://doi-org.dianus.libr.tue.nl/10.1017/s0963548323000184},
}

@article {BriGirLugSul25,
    AUTHOR = {Briend, S. and Giraud, C. and Lugosi, G. and
              Sulem, D.},
     TITLE = {Estimating the history of a random recursive tree},
   JOURNAL = {Bernoulli},
  FJOURNAL = {Bernoulli. Official Journal of the Bernoulli Society for
              Mathematical Statistics and Probability},
    VOLUME = {{\bf 31}},
      YEAR = {2025},
    NUMBER = {4},
     PAGES = {3260--3284},
      ISSN = {1350-7265,1573-9759},
   MRCLASS = {62G05 (68Q87 91G70)},
  MRNUMBER = {4931370},
       DOI = {10.3150/24-BEJ1846},
       URL = {https://doi-org.dianus.libr.tue.nl/10.3150/24-BEJ1846},
}

@article {BubDevLug17,
    AUTHOR = {Bubeck, S\'ebastien and Devroye, Luc and Lugosi, G\'abor},
     TITLE = {Finding {A}dam in random growing trees},
   JOURNAL = {Random Structures Algorithms},
  FJOURNAL = {Random Structures \& Algorithms},
    VOLUME = {{\bf 50}},
      YEAR = {2017},
    NUMBER = {2},
     PAGES = {158--172},
      ISSN = {1042-9832,1098-2418},
   MRCLASS = {05C80 (05C05 60C05 62F10)},
  MRNUMBER = {3607120},
MRREVIEWER = {Hai\ Yan\ Chen},
       DOI = {10.1002/rsa.20649},
       URL = {https://doi-org.dianus.libr.tue.nl/10.1002/rsa.20649},
}

@article{CalHopKleNewStr01,
    author={Callaway, D. S. and Hopcroft, J. E. and
    Kleinberg, J. M. and Newman, M.~E.~J. and Strogatz, S. H.},
    title={Are randomly grown graphs really random?},
    journal={Phys. Rev. E},
    volume={{\bf 64}},
    year={2001},
    pages={041902},
}

@ARTICLE{CouBau03,
	author = {Coulomb, S. and Bauer, M.},
	title = {Asymmetric evolving random networks},
	year = {2003},
	journal = {European Physical Journal B},
	volume = {{\bf 35}},
	number = {3},
	pages = {377--389},
	doi = {10.1140/epjb/e2003-00290-4},
	type = {Article},
	source = {Scopus}
}

@article {CraXu21,
    AUTHOR = {Crane, H. and Xu, M.},
     TITLE = {Inference on the history of a randomly growing tree},
   JOURNAL = {J. R. Stat. Soc. Ser. B. Stat. Methodol.},
  FJOURNAL = {Journal of the Royal Statistical Society. Series B.
              Statistical Methodology},
    VOLUME = {{\bf 83}},
      YEAR = {2021},
    NUMBER = {4},
     PAGES = {639--668},
      ISSN = {1369-7412,1467-9868},
   MRCLASS = {62H12 (05C80 60C05 60J20 62F15 62F25)},
  MRNUMBER = {4319997},
MRREVIEWER = {Mina\ Towhidi},
       DOI = {10.1111/rssb.12428},
       URL = {https://doi-org.dianus.libr.tue.nl/10.1111/rssb.12428},
}

@article {DerMor09,
    AUTHOR = {Dereich, S. and M{\"o}rters, P.},
     TITLE = {Random networks with sublinear preferential attachment: degree
              evolutions},
   JOURNAL = {Electron. J. Probab.},
  FJOURNAL = {Electronic Journal of Probability},
    VOLUME = {{\bf 14}},
      YEAR = {2009},
     PAGES = {no. 43, 1222--1267},
      ISSN = {1083-6489},
   MRCLASS = {05C80 (05C90 60C05 60K30 90B15)},
  MRNUMBER = {2511283},
       DOI = {10.1214/EJP.v14-647},
       URL = {https://doi-org.dianus.libr.tue.nl/10.1214/EJP.v14-647},
}

@article {DerMor13,
    AUTHOR = {Dereich, S. and M{\"o}rters, P.},
     TITLE = {Random networks with sublinear preferential attachment: the
              giant component},
   JOURNAL = {Ann. Probab.},
  FJOURNAL = {The Annals of Probability},
    VOLUME = {{\bf 41}},
      YEAR = {2013},
    NUMBER = {1},
     PAGES = {329--384},
      ISSN = {0091-1798},
   MRCLASS = {05C80 (60C05 90B15)},
  MRNUMBER = {3059201},
       DOI = {10.1214/11-AOP697},
       URL = {http://dx.doi.org/10.1214/11-AOP697},
}

@article{DorGolMen08,
  author={Dorogovtsev, S.N. and Goltsev, A.V. and Mendes, J.F.F.},
  title={{Critical phenomena in complex networks}},
  journal={Reviews of Modern Physics},
  volume={{\bf 80}},
  number={4},
  pages={1275--1335},
  year={2008},
  publisher={APS}
}

@book {Durr06,
    AUTHOR = {Durrett, R.},
     TITLE = {Random graph dynamics},
    SERIES = {Cambridge Series in Statistical and Probabilistic Mathematics},
 PUBLISHER = {Cambridge University Press},
      YEAR = {2007},
     PAGES = {x+212},
      ISBN = {978-0-521-86656-9; 0-521-86656-1},
   MRCLASS = {05C80 (05-02 60-02 60C05 60G50 60K35 82C41)},
  MRNUMBER = {MR2271734},
}

@incollection {Durr03,
    AUTHOR = {Durrett, R.},
     TITLE = {Rigorous result for the {CHKNS} random graph model},
 BOOKTITLE = {Discrete random walks ({P}aris, 2003)},
     PAGES = {95--104},
 PUBLISHER = {Association of Discrete Mathematics and Theoretical Computer Sciience},
      YEAR = {2003},
   MRCLASS = {05C80 (60C05 82B26)},
  MRNUMBER = {2042377},
MRREVIEWER = {Mark R. Jerrum},
}

@article {EckMorOrt18,
    AUTHOR = {Eckhoff, M. and M\"orters, P. and Ortgiese, M.},
     TITLE = {Near critical preferential attachment networks have small
              giant components},
   JOURNAL = {J. Stat. Phys.},
  FJOURNAL = {Journal of Statistical Physics},
    VOLUME = {{\bf 173}},
      YEAR = {2018},
    NUMBER = {3-4},
     PAGES = {663--703},
      ISSN = {0022-4715,1572-9613},
   MRCLASS = {82B26 (05C80 60J80 60J85 90B15)},
  MRNUMBER = {3876903},
       DOI = {10.1007/s10955-018-2054-5},
       URL = {https://doi.org/10.1007/s10955-018-2054-5},
}

@article{ErdRen60,
    author={Erd{\H{o}}s, P. and R{\'e}nyi, A.},
    title={On the evolution of random graphs},
    journal={Magyar Tud. Akad. Mat. Kutat{\'o} Int. K{\"o}zl.},
    volume={{\bf 5}},
    year={1960},
    pages={17--61},
    mrclass={05.40},
    mrnumber={MR0125031 (23 \#A2338)},
}

@article {Foun07,
    AUTHOR = {Fountoulakis, N.},
     TITLE = {Percolation on sparse random graphs with given degree
              sequence},
   JOURNAL = {Internet Math.},
  FJOURNAL = {Internet Mathematics},
    VOLUME = {{\bf 4}},
      YEAR = {2007},
    NUMBER = {4},
     PAGES = {329--356},
      ISSN = {1542-7951},
   MRCLASS = {05C80 (05C07 60K35)},
  MRNUMBER = {2522948},
MRREVIEWER = {Maria Deijfen},
       URL = {http://projecteuclid.org.dianus.libr.tue.nl/euclid.im/1243430810},
}

@unpublished{GarHazHofRay22,
  title={Universality of the local limit in preferential attachment models},
  author={Garavaglia, A. and Hazra, R. and {\noopsort{Hofstad}{van der Hofstad}}, R. and Ray, R.},
  note={arXiv:2212.05551 [math.PR]},
  year={2022}
}

@book{Grim99,
    author={Grimmett, G.},
    TITLE = {Percolation},
    SERIES = {Grundlehren der mathematischen Wissenschaften [Fundamental
              Principles of Mathematical Sciences]},
    VOLUME = {321},
   EDITION = {Second},
 PUBLISHER = {Springer-Verlag, Berlin},
      YEAR = {1999},
     PAGES = {xiv+444},
      ISBN = {3-540-64902-6},
   MRCLASS = {60K35 (60-02 82B43)},
  MRNUMBER = {1707339},
MRREVIEWER = {Neal\ Madras},
       DOI = {10.1007/978-3-662-03981-6},
       URL = {https://doi-org.dianus.libr.tue.nl/10.1007/978-3-662-03981-6},
}

@unpublished{HazHofRay23,
  title={Percolation on preferential attachment models},
  author={Hazra, R. and {\noopsort{Hofstad}{van der Hofstad}}, R. and Ray, R.},
  note={arXiv:2312.14085 [math.PR]},
  year={2023}
}

@book {Hofs17,
    AUTHOR = {{\noopsort{Hofstad}{van der Hofstad}}, R.},
     TITLE = {Random graphs and complex networks. {V}olume 1},
    SERIES = {Cambridge Series in Statistical and Probabilistic Mathematics},
 PUBLISHER = {Cambridge University Press},
      YEAR = {2017},
      VOLUME={{\bf 43}},
     PAGES = {xvi+321},
      ISBN = {978-1-107-17287-6},
   MRCLASS = {05-01 (05C80 05C82)},
  MRNUMBER = {3617364},
       DOI = {10.1017/9781316779422},
       URL = {https://doi-org.dianus.libr.tue.nl/10.1017/9781316779422},
}

@book {Hofs24,
    AUTHOR = {{\noopsort{Hofstad}{van der Hofstad}}, R.},
     TITLE = {Random graphs and complex networks. {V}olume 2},
    SERIES = {Cambridge Series in Statistical and Probabilistic Mathematics},
 PUBLISHER = {Cambridge University Press},
      YEAR = {2024},
      VOLUME={{\bf 54}},
     PAGES = {xvi+490},
      ISBN = {978-1-107-17400-9},
   MRCLASS = {05-01 (05C80 05C82)},
       DOI = {10.1017/9781316795552},
       URL = {https://doi-org.dianus.libr.tue.nl/10.1017/9781316795552},
}

@unpublished{HofZhu25,
  title = {Logarithmic typical distances in preferential attachment models},
  author = {{\noopsort{Hofstad}{van der Hofstad}}, R. and Zhu, H.},
  note = {ar{X}iv:2502.07961 [math.PR]},
  year = {2025}
}

@article {Jans08b,
    AUTHOR = {Janson, S.},
     TITLE = {The largest component in a subcritical random graph with a
              power law degree distribution},
   JOURNAL = {Ann. Appl. Probab.},
  FJOURNAL = {The Annals of Applied Probability},
    VOLUME = {{\bf 18}},
      YEAR = {2008},
    NUMBER = {4},
     PAGES = {1651--1668},
      ISSN = {1050-5164},
   MRCLASS = {05C80 (60C05)},
  MRNUMBER = {MR2434185},
}

@article {Jans09b,
    AUTHOR = {Janson, S.},
     TITLE = {Susceptibility of random graphs with given vertex degrees},
   JOURNAL = {J. Combin.},
  FJOURNAL = {Journal of Combinatorics},
    VOLUME = {{\bf 1}},
      YEAR = {2010},
    NUMBER = {3-4},
     PAGES = {357--387},
      ISSN = {2156-3527},
   MRCLASS = {05C80 (60J80)},
  MRNUMBER = {2799217 (2012c:05278)},
MRREVIEWER = {Lyuben R. Mutafchiev},
}

@article {Jans09c,
    AUTHOR = {Janson, S.},
     TITLE = {On percolation in random graphs with given vertex degrees},
   JOURNAL = {Electron. J. Probab.},
  FJOURNAL = {Electronic Journal of Probability},
    VOLUME = {{\bf 14}},
      YEAR = {2009},
     PAGES = {no. 5, 87--118},
      ISSN = {1083-6489},
   MRCLASS = {60C05 (05C80 60K35)},
  MRNUMBER = {2471661 (2010b:60023)},
MRREVIEWER = {Maria Deijfen},
       DOI = {10.1214/EJP.v14-603},
       URL = {http://dx.doi.org.janus.libr.tue.nl/10.1214/EJP.v14-603},
}

@article{JanKnuLucPit93,
    author={Janson, S. and Knuth, D.E. and {\L}uczak, T. and Pittel, B.},
    title={The birth of the giant component},
    journal={Random Structures Algorithms},
    volume={{\bf 4}},
    year={1993},
    pages={231--358},
    number={3},
    fjournal={Random Structures \& Algorithms},
    issn={1042-9832},
    mrclass={05C80 (60C05)},
    mrnumber={MR1220220 (94h:05070)},
    note={With an introduction by the editors},
}

@ARTICLE{KimKraKahRed02,
	author = {Kim, J. and Krapivsky, P.L. and Kahng, B. and Redner, S.},
	title = {Infinite-order percolation and giant fluctuations in a protein interaction network},
	year = {2002},
	journal = {Physical Review E - Statistical Physics, Plasmas, Fluids, and Related Interdisciplinary Topics},
	volume = {{\bf 66}},
	number = {5},
	pages = {4},
	doi = {10.1103/PhysRevE.66.055101},
	type = {Article},
	publication_stage = {Final},
	source = {Scopus},
	note = {All Open Access, Green Open Access}
}

@article {Kurt70,
    AUTHOR = {Kurtz, T.G.},
     TITLE = {Solutions of ordinary differential equations as limits of pure
              jump {M}arkov processes},
   JOURNAL = {J. Appl. Probability},
  FJOURNAL = {Journal of Applied Probability},
    VOLUME = {{\bf 7}},
      YEAR = {1970},
     PAGES = {49--58},
      ISSN = {0021-9002,1475-6072},
   MRCLASS = {60.60},
  MRNUMBER = {254917},
MRREVIEWER = {J.\ A.\ Beekman},
       DOI = {10.2307/3212147},
       URL = {https://doi-org.dianus.libr.tue.nl/10.2307/3212147},
}

@article {Kurt71,
    AUTHOR = {Kurtz, T. G.},
     TITLE = {Limit theorems for sequences of jump {M}arkov processes
              approximating ordinary differential processes},
   JOURNAL = {J. Appl. Probability},
  FJOURNAL = {Journal of Applied Probability},
    VOLUME = {8},
      YEAR = {1971},
     PAGES = {344--356},
      ISSN = {0021-9002,1475-6072},
   MRCLASS = {60.60},
  MRNUMBER = {287609},
MRREVIEWER = {J.\ A.\ Beekman},
       DOI = {10.1017/s002190020003535x},
       URL = {https://doi-org.dianus.libr.tue.nl/10.1017/s002190020003535x},
}

@article {Lucz90a,
    AUTHOR = {{\L}uczak, T.},
     TITLE = {Component behavior near the critical point of the random graph
              process},
   JOURNAL = {Random Structures Algorithms},
  FJOURNAL = {Random Structures \& Algorithms},
    VOLUME = {{\bf 1}},
      YEAR = {1990},
    NUMBER = {3},
     PAGES = {287--310},
      ISSN = {1042-9832},
   MRCLASS = {05C80 (60C05)},
  MRNUMBER = {MR1099794 (92c:05139)},
MRREVIEWER = {Graham Brightwell},
}

@article {LugPer19,
    AUTHOR = {Lugosi, G\'abor and Pereira, Alan S.},
     TITLE = {Finding the seed of uniform attachment trees},
   JOURNAL = {Electron. J. Probab.},
  FJOURNAL = {Electronic Journal of Probability},
    VOLUME = {{\bf 24}},
      YEAR = {2019},
     PAGES = {Paper No. 18, 15},
      ISSN = {1083-6489},
   MRCLASS = {60C05 (05C05 05C80 62H12)},
  MRNUMBER = {3925458},
       DOI = {10.1214/19-EJP268},
       URL = {https://doi-org.dianus.libr.tue.nl/10.1214/19-EJP268},
}

@unpublished{MorSch25,
  title={The largest subcritical component in inhomogeneous random graphs of preferential attachment type},
  author={M\"orters, P. and Schleicher, N.},
  note={ar{X}iv: 2503.05469 [math.PR]},
  year={2025}
}

@ARTICLE{NavKin11,
	author = {Navlakha, S. and Kingsford, C.},
	title = {Network archaeology: Uncovering ancient networks from present-day interactions},
	year = {2011},
	journal = {PLoS Computational Biology},
	volume = {{\bf 7}},
	number = {4},
	doi = {10.1371/journal.pcbi.1001119},
	type = {Article},
	source = {Scopus}
	}

@book {Norr98,
    AUTHOR = {Norris, J. R.},
     TITLE = {Markov chains},
    SERIES = {Cambridge Series in Statistical and Probabilistic Mathematics},
    VOLUME = {{\bf 2}},
      NOTE = {Reprint of 1997 original},
 PUBLISHER = {Cambridge University Press},
   ADDRESS = {Cambridge},
      YEAR = {1998},
     PAGES = {xvi+237},
      ISBN = {0-521-48181-3},
   MRCLASS = {60J10 (60-01 60J27)},
  MRNUMBER = {1600720 (99c:60144)},
MRREVIEWER = {M. G. Shur},
}

@article {Pema07,
    AUTHOR = {Pemantle, R.},
     TITLE = {A survey of random processes with reinforcement},
   JOURNAL = {Probab. Surv.},
  FJOURNAL = {Probability Surveys},
    VOLUME = {{\bf 4}},
      YEAR = {2007},
     PAGES = {1--79 (electronic)},
      ISSN = {1549-5787},
   MRCLASS = {60J20 (60G50 60J10 60K40 62L20)},
  MRNUMBER = {MR2282181 (2007k:60230)},
MRREVIEWER = {M. Iosifescu},
}

@article {Rior05,
    AUTHOR = {Riordan, O.},
     TITLE = {The small giant component in scale-free random graphs},
   JOURNAL = {Combin. Probab. Comput.},
  FJOURNAL = {Combinatorics, Probability and Computing},
    VOLUME = {{\bf 14}},
      YEAR = {2005},
    NUMBER = {5-6},
     PAGES = {897--938},
      ISSN = {0963-5483},
   MRCLASS = {05C80},
  MRNUMBER = {2174664},
       DOI = {10.1017/S096354830500708X},
       URL = {http://dx.doi.org.dianus.libr.tue.nl/10.1017/S096354830500708X},
}

@article {Worm95,
    AUTHOR = {Wormald, N.},
     TITLE = {Differential equations for random processes and random graphs},
   JOURNAL = {Ann. Appl. Probab.},
  FJOURNAL = {The Annals of Applied Probability},
    VOLUME = {{\bf 5}},
      YEAR = {1995},
    NUMBER = {4},
     PAGES = {1217--1235},
      ISSN = {1050-5164},
   MRCLASS = {05C80 (34F05 60C05)},
  MRNUMBER = {MR1384372 (97c:05139)},
MRREVIEWER = {Zbigniew Palka},
}

\newpage


\appendix

\section{Local limits for uniform and preferential attachment graphs}
\label{sec-local-limit-UA}
In this section, we show that the uniform attachment model with a general number $m$ of out-edges per vertex locally converges, and identify the local limit. We follow \cite{HazHofRay23}, which is based on \cite{GarHazHofRay22}. In turn, \cite{GarHazHofRay22} is inspired by \cite{BerBorChaSab14}. An alternative approach can be found in \cite{Rior05}.
This appendix is organised as follows. In Section \ref{sec-local-limit-PA}, we describe the local limit of the preferential attachment model. In Section \ref{sec-local-limit-UA-rep}, we extend this to the uniform attachment model for general $m$. In Section \ref{sec-local-limit-UA-ext}, we discuss two related settings. The first is the local limit of {\em percolation} on dynamic random graphs. The second is the related description of the local limits of uniform and preferential attachment models in terms of a two-type killed branching random walk, which arises by taking a logarithmic time transformation.

\subsection{Local limit of preferential attachment models}
\label{sec-local-limit-PA}
Berger, Borgs, Chayes and Saberi proved in \cite{BerBorChaSab14} that the P\'olya point tree is the local limit of the preferential attachment model (PAM). Informally, the uniform attachment model can be obtained from the preferential attachment model by considering the model in \eqref{eqn:linear-att-def} with $a=1$, and taking the limit of $\delta\rightarrow \infty$. In this limit, the local limit simplifies considerably. The reason is that for uniform attachment, the end points of the edges of arriving vertices are chosen {\em independently}, while for preferential attachment, they are chosen according to their current degrees. Remarkably, as proved in \cite{BerBorChaSab14}, by a P\'olya-urn representation, this can be seen as choosing the vertices independently {\em conditionally} on some randomness given by independent Beta variables. Thus, many of the computations in \cites{BerBorChaSab14, GarHazHofRay22} are performed {\em conditionally} on these Beta variables, a step that is not necessary for the uniform attachment model.
\smallskip

We describe the local limit of the PAM, followed by an explanation of how this local limit can be derived. We follow \cite[Section 2.2]{HazHofRay23} almost verbatim. We start by defining the vertex set of this PPT:

\begin{defn}[Ulam-Harris set and its ordering]\label{def:ulam}
	{\rm Let $\N_0=\N\cup\{0\}$. {The {\em Ulam-Harris set}} is
		\begin{equation*}
			\mathcal{U}=\bigcup\limits_{n\in\N_0}\N^n.
		\end{equation*}
		For $x = x_1\cdots x_n\in\N^n$ and $k\in\N$, we denote {the element $x_1\cdots x_nk$ by $xk\in\N^{n+1}$.} The~{\em root} of the Ulam-Harris set is denoted by $\emp\in\N^0$.
		
		For any $x\in \mathcal{U}$, we say that $x$ has length $n$ if $x\in\N^n$. {The lexicographic \textit{ordering}} between the elements of the Ulam-Harris set {is as follows:}
		\begin{itemize}
			\item[(a)] for any two elements $x,y\in\mathcal{U}$, {$x>y$ when the length of $x$ is more than that of $y$;}
			\item[(b)] if $x,y\in \N^n$ for some $n$, then $x>y$ if there exists $i\leq n,$ such that $x_j=~y_j$ for all $j<i$ and $x_i>~y_i$.\hfill$\blacksquare$
	\end{itemize}}
\end{defn}

We use the elements of the Ulam-Harris set $\mathcal{U}$ to identify nodes in a rooted tree, since the notation in Definition~\ref{def:ulam} allows us to denote the relationships between children and parents, where for $x\in\mathcal{U}$, we denote the $k^{\rm th}$ child of $x$ by the element $xk$. 

\medskip

\paragraph{\bf P\'olya Point Tree ($\PPT$).}
The $\PPT(m,\delta)$ is an {\em infinite multi-type rooted random tree}, where $m$ and $\delta>-m$ are the parameters of preferential attachment models. It is a multi-type branching process, with a mixed continuous and discrete type space. We now describe its properties one by one.
\medskip

\paragraph{\bf Descriptions of the distributions and parameters used.}
\begin{enumerate}
	\item[{\btr}] {Define} $\chi = \frac{m+\delta}{2m+\delta}$.
	\smallskip
	\item[{\btr}] {Let} $\Gamma_{\sf{in}}(m)$ denote a Gamma distribution with parameters $m+\delta$ and $1$.
	\smallskip
	\item[{\btr}] {Let} $\Gamma_{\sf{in}}^\prime(m)$ denote the size-biased distribution of $\Gamma_{\sf{in}}(m)$, which is also a Gamma distribution with parameters $m+\delta+1$ and $1$.
\end{enumerate}
\medskip

\paragraph{\bf Features of the vertices of the $\PPT$.}
Below, to avoid confusion, we use `node' for a vertex in the PPT and `vertex' for a vertex in the PAM. We now discuss the properties of the nodes in $\PPT(m,\delta)$. Every node except the root in the $\PPT$ has
\begin{enumerate}
	\item[{\btr}] a {\em label} $z$ in the Ulam-Harris set $\mathcal{N}$ (recall Definition~\ref{def:ulam});
	\smallskip
	\item[{\btr}] a {\em birth time}\footnote{In \cite{GarHazHofRay22}, the birth time was called the {\em age}.} $A_{z}\in[0,1]$;
	\smallskip
	\item[{\btr}] a positive number $\Gamma_{z}$ called its {\em strength};
	\smallskip
	\item[{\btr}] a {label in $\{{\rO},{\rY}\}$} depending on the birth time of the {node} and its parent, {with $\rY$ denoting that the node is younger than its parent and $\rO$ denoting that the node is older than its parent.}
\end{enumerate}
Based on its label being $\rO$ or $\rY$, every {node} $\omega$ has a number $m_{-}(\omega)$ associated to it. If $\omega$ has type $\rO$, then $m_{-}(\omega)=m$, and $\Gamma_{\omega}$ is distributed as $\Gamma_{\sf{in}}^\prime( m )$,
while if $\omega$ has type $\rY$, then $m_{-}(\omega)=m-1$, and given $m_{-}(\omega), \Gamma_{\omega}$ is distributed as $\Gamma_{\sf{in}}( m )$.
\medskip

\begingroup
\allowdisplaybreaks
\paragraph{\bf Construction of the PPT} We next use the above definitions to construct the PPT using an {\em exploration process}.
The root is special in the tree. It has label $\emp$ and its birth time $A_\emp$ is an independent uniform random variable in $[0,1]$. The root $\emp$ has no label in $\{\rO,\rY\}$, but since we let $m_{-}(\emp)=m$, we can think of $\emp$ as having label $\rO$.
Then the {children of the root in the} P\'{o}lya point tree {are} constructed as follows:
\endgroup	
\begin{enumerate}
	\item Sample $U_1,\ldots,U_{m_{-}(\emp)}$ uniform random variables on $[0,1]$, independent of the rest;  
	\item To nodes $\emp1,\ldots, \emp m_{-}(\emp),$ assign the birth times $U_1^{1/\chi} A_\emp,\ldots,U_{m_{-}(\emp)}^{1/\chi} A_\emp$ and type $\rO$;
	\item Assign birth times $A_{\emp(m_{-}(\emp)+1)},\ldots, A_{\emp(m_{-}(\emp)+\din_\emp)}$ to nodes $\emp(m_{-}(\emp)+1),\ldots, \emp(m_{-}(\emp)+\din_\emp)$. These birth times are the occurrence times given by a conditionally independent Poisson point process on $[A_\emp,1]$ defined by the {\em random} intensity
    \eqn{\label{def:rho_emp}
		\rho_{\emp}(x) = {(1-\chi)}{\Gamma_\emp}\frac{x^{-\chi}}{A_\emp^{1-\chi}},}
	and $\din_\emp$ being the total number of points of this process. Assign type $\rY$ to them;
	\item Draw an edge between $\emp$ and each of $\emp 1,\ldots, \emp(m_{-}(\emp)+\din_\emp)$; 
	\item Label $\emp$ as explored and nodes $\emp1,\ldots, \emp(m_{-}(\emp)+\din_\emp)$ as unexplored.
\end{enumerate}	
Then, recursively over the elements in the set of unexplored nodes, we perform the following breadth-first exploration:
\begin{enumerate}
	\item Let $\omega$ denote the smallest currently unexplored node in the Ulam-Harris ordering;
	\item Sample $m_{-}(\omega)$ i.i.d.\ random variables $U_{\omega1},\ldots,U_{\omega m_{-}(\omega)}$ independently from all the previous steps and from each other, uniformly on $[0,1]$. To nodes $\omega 1,\ldots,\omega m_{-}(\omega)$ assign the birth times $U_{\omega1}^{1/\chi}A_\omega,\ldots,U_{\omega m_{-}(\omega)}^{1/\chi}A_\omega$ and label $\rO$, and set them unexplored;
	\item Let $A_{\omega(m_{-}(\omega)+1)},\ldots,A_{\omega(m_{-}(\omega)+\din_\omega)}$ be the random $\din_\omega$ points given by a conditionally independent Poisson process on $[A_{\omega},1]$ with random intensity
	\eqn{
		\label{for:pointgraph:poisson}
		\rho_{\omega}(x) = (1-\chi)\Gamma_\omega\frac{x^{-\chi}}{A_\omega^{1-\chi}}.
	}
	Assign these birth times to $\omega(m_{-}(\omega)+1),\ldots,\omega(m_{-}(\omega)+\din_\omega)$. {Assign them type $\rY$,} and set them unexplored;
	\item Draw an edge between $\omega$ and each one of the nodes $\omega 1,\ldots,\omega(m_{-}(\omega)+\din_\omega)$;
	\item Set $\omega$ as explored.
\end{enumerate}
We call the resulting tree the {\em P\'olya point tree with parameters $m$ and $\delta$,} and denote it by $\PPT(m,\delta)$.

\subsection{Local limit of uniform attachment models}
\label{sec-local-limit-UA-rep}
From the proof of the local limit result in \cite{GarHazHofRay22}, it is clear that the Gamma random variables in \eqref{def:rho_emp} and \eqref{for:pointgraph:poisson} arise from the Beta variables in the P\'olya description of the preferential attachment model. Indeed, in this description, vertex $k$ has a `strength' given by a Beta random variable $\psi_k$ with parameters $m+\delta$ and $(2k-3)m+\delta(j-1)$ (recall, e.g., \cite[Section 5.3.3]{Hofs24}). As $k\rightarrow \infty$,
    \eqn{
    \label{conv-Beta-Gamma}
    k\psi_k \convd \Gamma,
    }
where $\Gamma$ is a Gamma random variable with parameters $r=m+\delta$ and $\lambda=2m+\delta.$ Furthermore, the {\em size-biased} version of $\psi_k$, which also appears in the P\'olya urn representation of the PAM, also satisfies the convergence in \eqref{conv-Beta-Gamma}, but now with $r=m+\delta+1$, which is the same as the sized-biased version of the Gamma random variable with parameters $r=m+\delta$ and $\lambda=2m+\delta.$ This explains how the Gamma variables arising in the local limit for the preferential attachment model are directly related to the Beta variables in the P\'olya description of it. For uniform attachment, however, the edge attachments are {\em independent} rather than {\em conditionally independent given the Beta variables}, which significantly simplifies the analysis. Formally, the uniform attachment case can be obtained by taking the limit $\delta\rightarrow \infty.$
\smallskip

Indeed, we define the local limit of the uniform attachment graph as for the $\PPT$, but where 
\begin{enumerate}
	\item[{\btr}] $\chi$ is replaced by 1;
    \item[{\btr}] $(1-\chi)\Gamma_\emp$ and $(1-\chi)\Gamma_\omega$ are  replaced by $m$.
\end{enumerate}
Therefore, \eqref{def:rho_emp} and \eqref{for:pointgraph:poisson} are both replaced by
\eqn{
		\label{for:pointgraph:poisson-UA}
		\rho_{\omega}(x) =\frac{m}{x}.
	}
This can be understood by noting that a Gamma variable $\Gamma_r$ with parameters $r$ and $1$, as $r\rightarrow \infty$, satisfies
    \eqn{
    \Gamma_{r}/r\convas 1.
    }
Therefore, in \eqref{def:rho_emp} and \eqref{for:pointgraph:poisson}, we note that $1-\chi=m/(2m+\delta)$ and $r=m+\delta$ or $r=m+1+\delta$, to conclude that $(1-\chi)\Gamma_{\omega}\convas m$ as $\delta\rightarrow \infty$. Further, the number of attachments in the uniform attachment model of a vertex with birth time $\lceil un\rceil$ in the time interval $[\lceil an\rceil, \lceil bn\rceil]$, for any $u<a<b$, is given by
    \eqn{
    \label{num-attachemnts-UA}
    \sum_{i=m \lceil an\rceil}^{m\lceil bn\rceil} I_i,
    }
where $(I_i)_{i\geq \lceil un\rceil}$ are independent Bernoulli random variables with success probability $1/\lceil i/m\rceil$. By a Poisson approximation of a sum of independent Bernoulli variables, we can approximate the random variable in \eqref{num-attachemnts-UA} by a Poisson random variable with parameter
    \eqn{
    \int_a^b \frac{m}{x} dx.
    }
Further, the number of attachments in disjoint intervals are {\em independent}, which explains \eqref{for:pointgraph:poisson-UA}.
\smallskip

This shows how one can obtain the local limit for the uniform attachment model from that of the PAM:
\paragraph{\bf Proof of local limit for uniform attachment.} The proof that the above is the local limit of the uniform attachment model can be performed in exactly the same way as in \cite{GarHazHofRay22}.
\qed

The above also gives us the constructs and tools needed to prove Lemma \ref{lem:size-identity}:

\begin{proof}[Proof of Lemma \ref{lem:size-identity}]
  We note that the summands in the sum over $\ell$ correspond to the expected number of offspring in the $\ell^{\rm th}$ generation of the \(\pi\)-percolated local limit of the uniform attachment graph, which has offspring kernel $\kappa$, as we will prove below. Therefore, 
    \eqn{\label{for:s2:UB-5}
        \expec\big[ \susceptibilitypi{n}{2} \big] \leq 1+ \sum\limits_{\ell=1}^{\infty} \pi^\ell \expec \Big[ |\partial B_\ell(\varnothing)| \Big] = \expec \big[ |\clusterpi{\varnothing}{\infty}| \big]~.
    }
To identify the kernel $\kappa$, we follow the proof of \cite[Lemma 5.25]{Hofs24}. Let $s,t\in \{\rY,\rO\}$ and $x,y\in [0,1]$. Let $\kappa\big((x,s), ([0,y],t)\big)$ denote the expected number of children of type in $([0,y],t)$ of an individual of type $(x,s)$, and let
	\eqn{
	\label{kappa-UA-fin-def}
	\kappa\big((x,s), (y,t)\big)=\frac{{\mathrm d}}{{\mathrm d} y} \kappa\big((x,s), ([0,y],t)\big)
	}
denote the integral kernel of the offspring operator of the multi-type branching process. Recall that older children in the tree are located at iid standard uniform random variables times the birth time of the vertex. Thus, for $[a,b]\subseteq [0,x]$,
	\eqn{
	\label{kappa-UA-1}
	\kappa\big((x,\rO), ([a,b],\rO)\big)
	=m \int_{a/x}^{b/x} {\mathrm d}y=m(b-a).
	}
while
	\eqn{
	\label{kappa-UA-3}
	\kappa\big((x,\rY), ([a,b],\rO)\big)
	=(m-1)\int_{a/x}^{b/x}\frac{1}{y}{\mathrm d}y=(m-1)(\log(b)-\log(a)),
	}
Further, for $[a,b]\subseteq [x,1]$,
	\eqan{
	\label{kappa-UA-2}
	\kappa\big((x,\rO), ([a,b],\rY)\big)
	&=m\expec\Big[\Poi\big(\int_{a}^{b} \frac{1}{y}{\mathrm d}y\big)\Big]=m(\log(b)-\log(a)).
	}
while, for $[a,b]\subseteq [x,1]$,
	\eqan{
	\label{kappa-UA-4}
	\kappa\big((x,\rY), ([a,b],\rY)\big)
	&=\expec\Big[\Poi\Big(\int_{a}^{b} \frac{1}{y}{\mathrm d}y\Big)\Big]=m(\log(b)-\log(a)).
	}
We now use \eqref{kappa-UA-fin-def}, and see that all terms in $\kappa\big((x,s), (y,t)\big)$ are of the form of the rhs of \eqref{def-kappa-UA}, with $c_{st}=m$
except for $c_{\srY\srO}$, which equals $c_{\srY\srO}=m-1$. Together with \eqref{kappa-UA-fin-def}, this yields \eqref{def-kappa-UA}.
\end{proof}

\subsection{Equivalent formulation of local limits and proof of Proposition \ref{prop:expec-s2-convg}(a)}
\label{sec-local-limit-UA-ext}
In this section, we discuss two extensions of the previous results. In the first, we derive the killed two-type branching random walk description of the local limits of preferential and uniform attachment models. In the second, we discuss the local limit of {\em percolation} on such random graphs. These descriptions prove Proposition \ref{prop:expec-s2-convg}(a). 
\medskip

\paragraph{\bf Logarithmic transformation of local limit.}
We instead consider the {\em location} of nodes to be $\log(A_\omega)$, where $A_\omega$ is the birth time of the Ulam-Harris node $\omega$. This leads to the following changes. First, the root is located at $-E_{\omega}$. Then, the dynamics is as follows:
\begin{itemize}
\item[(2')] The location of the nodes of label $\rO$ is 
$-E_1/\chi,\ldots,-E_1/\chi$ away from the location of $\emp$;
\item[(3')] The $\din_\omega$ points with label $\rY$ have location compared to their parent that is given by a conditionally independent Poisson process on $(0,\infty)$ with random intensity
	\eqn{
    \label{for:pointgraph:poisson-logarithmic}
		\lambda_{\omega}(x) = (1-\chi)\Gamma_\omega \e^{(1-\chi) x},
	}
with nodes of location larger than 0 being killed. 
\end{itemize}
Item (2') follows, since the multiplicative nature of the birth times compared to the parents in item (2) of the PPT is changed into an {\em additive} nature, which can be seen as a displacement. Item (3') follows, since the number of younger children with location in $[a,b]$ of a parent with location equal to $c$, where $c<a<b$, is the number of children with birth time in $[\e^a, \e^b]$ of a node in the PPT with birth time $\e^c$, which is equal to 
    \eqn{
    \label{logarithmic-transformation-PA}
    \int_{\e^a}^{\e^b} (1-\chi)\Gamma_\omega\frac{x^{-\chi}}{\e^{(1-\chi)c}}\dif x
    =\Gamma_\omega (\e^{(1-\chi)(b-c)}-\e^{(1-\chi)(a-c)}),
    }
so that the {\em displacement} of the location equals an inhomogeneous Poisson point process with intensity $(1-\chi)\Gamma_\omega \e^{(1-\chi)x}$. This process needs to be killed at zero, since the original process was killed at 1. This describes the effect of the logarithmic transformation on the local limit of the preferential attachment model. We next extend this to the uniform attachment model, for which $\chi=1$, and the intensity is given by \eqref{for:pointgraph:poisson-UA}. Then \eqref{logarithmic-transformation-PA} is replaced by
    \eqn{
    \label{logarithmic-transformation-UA}
    \int_{\e^a}^{\e^b} \frac{m}{x}\dif x
    =m(b-a),
    }
so that the attachment process of label $\rY$ nodes is a homogeneous Poisson process.
\smallskip

We conclude that, after the logarithmic transformation, the local limits of the PAM and UA for general $m\geq 1$ are killed branching random walks, where nodes with non-negative location are killed, i.e., the barrier is at zero.
This proves Proposition \ref{prop:expec-s2-convg}(a) for $\pi=1$; see also Remark \ref{rem-barrier-zero}.
\medskip

\paragraph{\bf Local limit of percolation on uniform and preferential attachment models.}
We next extend the analysis to include {\em percolation} on the preferential and uniform attachment graphs. We start by describing the effect on the local limits. First, by the thinning properties of Poisson processes, when each node remains with probability $\pi$, the intensity of the Poisson process is simply multiplied by $\pi.$ For uniform attachment, by \eqref{logarithmic-transformation-UA}, this explains the intensity $2\pi$ in \eqref{intensity-Y-children-UA}, which, for general $m$, becomes $m\pi$. The number of $\rO$ children of a node $\omega$ in the tree becomes ${\rm Bin}(m_{-}(\omega),\pi),$ since each of these children is removed with probability $\pi$. The changes for preferential attachment are similar.
\smallskip

To prove that this is indeed the local limit of percolation on the preferential and uniform attachment model, we use the notion of {\em marked local convergence}; see \cite[Chapter 2]{Hofs24}. We associate independent Bernoulli$(\pi)$ random variables to all the edges in the graph. Then, the percolated subgraph is obtained by only considering edges with mark equal to 1. It is not hard to see that, when edges are equipped with {\em independent} random variables that take on finitely many values, local convergence of the original graph implies marked local convergence of the marked graph. This shows that the local limit of a bond percolated graph is percolation on the local limit. This completes the proof of local convergence of the percolated uniform and preferential attachment graphs. This proves Proposition \ref{prop:expec-s2-convg}(a) for general $\pi\in[0,1].$
\qed

\end{document}